\documentclass[12pt,reqno]{amsart}
\usepackage[margin=1in]{geometry}
\usepackage{amsmath,amssymb,amsthm,graphicx,amsxtra, setspace}
\usepackage[utf8]{inputenc}
\usepackage{mathrsfs}
\usepackage{hyperref}
\usepackage{upgreek}
\usepackage{mathtools}
\usepackage[dvipsnames]{xcolor}
\usepackage[mathcal]{euscript}
\usepackage{amsmath,units}
\usepackage{pgfplots}
\usepackage{tikz}
\usepackage{verbatim}
\allowdisplaybreaks
\usepackage{dsfont}
\usepackage{fontenc}
\usepackage{textcomp}
\usepackage{marvosym}
\usepackage{eurosym}
\usepackage{upgreek}
\usepackage[pagewise]{lineno}

\DeclareMathAlphabet{\mathpzc}{OT1}{pzc}{m}{it}
\DeclareMathOperator*{\esssup}{ess\,sup}

\usepackage[cyr]{aeguill}

\colorlet{darkblue}{blue!50!black}

\hypersetup{
	colorlinks,%
	citecolor=blue,%
	filecolor=red,%
	linkcolor=red,%
	urlcolor=blue,%
	pdfnewwindow=true,%
	pdfstartview={FitH}
}

\newtheorem{theorem}{Theorem}[section]
\newtheorem{lemma}[theorem]{Lemma}
\newtheorem{proposition}[theorem]{Proposition}

\newtheorem{definition}[theorem]{Definition}

\newtheorem{remark}[theorem]{Remark}
\newtheorem{hypothesis}[theorem]{Hypothesis}

\let\originalleft\left
\let\originalright\right
\renewcommand{\left}{\mathopen{}\mathclose\bgroup\originalleft}
\renewcommand{\right}{\aftergroup\egroup\originalright}

\renewcommand{\d}{\/\mathrm{d}\/}

\def\L{\mathbb{L}}
\def\A{\mathrm{A}}
\def\I{\mathrm{I}}

\def\C{\mathrm{C}}
\def\f{\boldsymbol{f}}

\def\D{\mathrm{D}}
\def\y{\mathbf{y}}
\def\q{\boldsymbol{q}}
\def\Y{\mathbf{Y}}
\def\Z{\boldsymbol{Z}}
\def\E{\mathbb{E}}

\def\x{\boldsymbol{x}}

\def\g{\boldsymbol{g}}
\def\p{\boldsymbol{p}}

\def\y{\boldsymbol{y}}
\def\z{\boldsymbol{z}}
\def\v{\boldsymbol{v}}

\def\W{\mathrm{W}}

\def\V{\mathbb{V}}
\def\R{\mathbb{R}}
\def\wi{\widetilde}

\def\U{\mathbb{U}}

\def\a{\boldsymbol{a}}
\def\H{\mathbb{H}}

\renewcommand{\d}{\/\mathrm{d}\/}

\newcommand{\Addresses}{{
		\footnote{
			
		\noindent \textsuperscript{1,2}Department of Mathematics, Indian Institute of Technology Roorkee-IIT Roorkee,
		Haridwar Highway, Roorkee, Uttarakhand 247667, INDIA.\par\nopagebreak
		\noindent  \textit{e-mail:} \texttt{Manil T. Mohan: maniltmohan@ma.iitr.ac.in, maniltmohan@gmail.com}
		
		\textit{e-mail:} \texttt{Sagar Gautam: sagar\_g@ma.iitr.ac.in, sagargautamkm@gmail.com}
		
		\noindent \textsuperscript{*}Corresponding author.

			\textit{Key words:} convective Brinkman-Forchheimer equations, tamed Navier-Stokes equations, viscosity solutions, Hamilton-Jacobi-Bellman equations, dynamic programming principle. 
			
			Mathematics Subject Classification (2010): Primary 35D40, 49L25; Secondary 76B75, 76D05

}}}

\begin{document}
	
	
	\title[Viscosity solutions of 2D and 3D CBF equations]{Optimal control of convective Brinkman-Forchheimer equations: Dynamic programming equation and Viscosity solutions
		\Addresses}
		\author[S. Gautam and M. T. Mohan]
	{Sagar Gautam\textsuperscript{1} and Manil T. Mohan\textsuperscript{2*}}

	\maketitle
   	
	\begin{abstract}
It has been pointed out in the work [F. Gozzi et.al., \emph{Arch. Ration. Mech. Anal.} {163}(4) (2002), 295--327] that the existence and uniqueness of viscosity solutions to the first-order Hamilton-Jacobi-Bellman equation (HJBE) associated with the three-dimensional Navier-Stokes equations (NSE) have not been resolved due to the lack of global solvability and continuous dependence results. However, by adding a damping term to NSE, the so-called \emph{damped Navier-Stokes equations} fulfills the  requirement of existence and uniqueness of global strong solutions. In this work, we address this issue in the context of the following two- and three-dimensional convective Brinkman-Forchheimer (CBF) equations (damped NSE)  in $\mathbb{T}^d,\ d\in\{2,3\}$: 
	\begin{align*}
		\frac{\partial\mathpzc{U}}{\partial t}-\mu \Delta\mathpzc{U}+(\mathpzc{U}\cdot\nabla)\mathpzc{U}+\alpha\mathpzc{U}+\beta|\mathpzc{U}|^{r-1}\mathpzc{U}+\nabla p=\boldsymbol{f}, \ \nabla\cdot\mathpzc{U}=0, 
	\end{align*}
	where $\mu,\alpha,\beta>0$, $r\in[1,\infty)$. 
	We study the global well-posedness of the infinite-dimensional first-order HJBE arising from an optimal control problem for CBF equations. Employing the dynamic programming approach, we prove the existence and uniqueness of viscosity solutions in both two and three-dimensions.   
	We first prove the existence of a viscosity solution to the infinite-dimensional HJBE in the supercritical regime. For spatial dimension $d=2$, we consider the nonlinearity exponent $r\in(3,\infty)$, while for $d=3$, due to some technical difficulty, we focus on $r\in(3,5]$. In the case $r=3$, we require the condition $2\beta\mu\geq 1$ for both $d=2$ and $d=3$. Next, we derive a comparison principle for the HJB equation covering the ranges $r\in(3,\infty)$ and $r=3$ with $2\beta\mu\geq 1$ in $d\in\{2,3\}$. It ensures the uniqueness of the viscosity solution. 
	\end{abstract}

	
	\section{Introduction}\setcounter{equation}{0}
This study employs a dynamic programming method to analyse the Hamilton-Jacobi-Bellman (or dynamic programming) equation that arises in the optimal control problem of turbulent flows. This control problem has numerous industrial and engineering applications (see \cite{MDG,SSS}). In the context of this work, fluid flows refer to those described by the convective Brinkman-Forchheimer (CBF) equations, also known as  damped Navier-Stokes equations (NSE).  In the high Reynolds number regime, where turbulence appears in NSE, the CBF equations can help to stabilize numerical solutions. These equations characterize the motion of an incompressible fluid flowing through a saturated porous medium.
 First, we provide the mathematical formulation of CBF equations. Let us denote the initial and terminal time by $t$ and $T$, respectively, with $t\in[0,\infty)$ and $T\in[t,\infty)$. We denote by $\mathbb{T}^{d}=\big(\R/\mathbb{Z}\big)^{d}$, a $d$-dimensional torus with dimension $d\in\{2,3\}$. 
Then the unknowns are the velocity vector field $\mathpzc{U}(\cdot):[t, T]\times\mathbb{T}^d\to\R^d$ (the state variable) and pressure $p(\cdot):[t, T]\times\mathbb{T}^d\to\R$ which satisfies the following set of equations (known as CBF equations):

\begin{equation}\label{CBF}
	\left\{
	\begin{aligned}
		\frac{\partial\mathpzc{U}}{\partial s}-\mu \Delta\mathpzc{U}+(\mathpzc{U}\cdot\nabla)\mathpzc{U}+\alpha\mathpzc{U}+\beta|\mathpzc{U}|^{r-1}\mathpzc{U}+\nabla p&=\g(s,\a(s)), \ \text{ in } \ (t,T]\times\mathbb{T}^{d}, \\ \nabla\cdot\mathpzc{U}&=0, \ \text{ in } \ [t,T]\times\mathbb{T}^{d}, \\
		\mathpzc{U}(t)&=\mathpzc{U}_0 \ \text{in} \ \mathbb{T}^{d},\\
		\int_{\mathbb{T}^d} p(x,t)\d x&=0, \ \text{ in } \ (t,T),
	\end{aligned}
	\right.
\end{equation}
where $\a(\cdot):[0, T]\to\Theta$ is a given measurable function that acts as a control taking values in a complete metric space $\Theta$ (the control set), and $\g(\cdot,\cdot):[0, T]\times\Theta\to\R^d$ represents an external forcing. The final condition in \eqref{CBF} is included to ensure the uniqueness of the pressure.
 Furthermore, the velocity field $\mathpzc{U}(\cdot,\cdot)$ and pressure $p(\cdot,\cdot)$ are subject to the following periodic boundary conditions:
\begin{align}\label{pbc}
	\mathpzc{U}(x+e_{i},\cdot) = \mathpzc{U}(x,\cdot)  \ \text{and}  \ p(x+e_{i},\cdot) = p(x,\cdot),
\end{align}
for all $x\in\R^{d}$ and $i=1,\ldots,d,$ where $\{e_i\}_{i=1}^d$ denotes the standard basis of $\R^{d}.$ The constant $\mu>0$ is referred to as the \emph{Brinkman coefficient} and represents the effective viscosity. The constants $\alpha>0$ and $\beta>0$ represents the \emph{Darcy} and \emph{Forchheimer} coefficients, modelling the permeability of porous medium and nonlinear drag due to porosity, respectively. The term $\beta|\mathpzc{U}|^{r-1}\mathpzc{U}$, where $r\geq1$, is commonly referred to as the absorption or damping term, and the parameter $r$ is known as the absorption exponent (cf. \cite{SNA}). The case $r=3$ is called the \emph{critical exponent} because, when $\alpha=0$, the CBF equations exhibit the same scaling properties as the classical NSE (\cite[Proposition 1.1]{KWH}). The case when $r<3$ is referred to as \emph{subcritical}, while those with $r>3$ are called \emph{supercritical}. The supercritical case is also associated with the fast-growing nonlinearities. When both $\alpha$ and $\beta$ are set to zero, the system of equations described in \eqref{CBF} simplifies to the classical NSE.

   \subsection{Literature survey}
  We will now explore the existing literature on the CBF equations and the viscosity solutions of the HJBE related to fluid dynamic models. 
 
      The existence of at least one weak solution of 3D NSE was first established by Leary and Hopf (cf. \cite{Hopf, Leray}). However, the question of uniqueness for such solutions remains a major problem for the mathematical community. To address this, several mathematicians have proposed modifications to the 3D NSE (cf. \cite{SNA, SNA1, ZCQJ, KT2}). In \cite{SNA, SNA1}, the authors proposed a modified version of the NSE by incorporating an absorption term $\beta|\mathpzc{U}|^{r-1}\mathpzc{U}$, where $r>0$. They proved the existence of a weak solution for all $d\geq2$, and established uniqueness in dimension $d=2$. In the literature, the classical NSE modified by the damping term $\alpha\mathpzc{U}+\beta|\mathpzc{U}|^{r-1}\mathpzc{U}$ are commonly referred to as the CBF equations (cf. \cite{SMTM,KWH}). The authors in \cite{MTT} studied the 3D NSE with damping $\beta|\mathpzc{U}|^{r-1}\mathpzc{U}$ alongside with a pumping term  $\gamma|\mathpzc{U}|^{q-1}\mathpzc{U}$, but without linear damping $\alpha\mathpzc{U}$. They proved the existence of weak solutions for the case $r>q$ and established uniqueness when $r>3$. Furthermore, they showed the global existence of a unique strong solution for $r>3$, assuming the initial in $\mathbb{H}^1$. As with the 3D NSE, the question of the existence of a unique global (in time) weak solution for the 3D CBF equations remains open for $r\in[1,3)$ with any $\beta,\mu>0$, and also for the case $r=3$ when $2\beta\mu<1$.
	
The primary goal of this work is to investigate the infinite-dimensional HJBE of first order related to an optimal control problem for the CBF equations in two and three dimensions using a dynamic programming approach. Such an optimal control problem arises from the need to reduce turbulence in fluid flows (see \cite{SSS}). Turbulence, characterized by chaotic and irregular motion, can be quantified using enstrophy.
The enstrophy of the flow is an important quantity which determines the rate of dissipation of kinetic energy (see \cite{FMRT}). It is defined as
\begin{align*}
	\mathcal{E}(\mathpzc{U}):=\int_{\mathbb{T}^d} |\nabla\mathpzc{U}(x)|^2\d x.
\end{align*}
For the incompressible flow, one can express the enstrophy in the following form:
\begin{align*}
	\mathcal{E}(\mathpzc{U}):=\int_{\mathbb{T}^d} |\boldsymbol{\omega}(x)|^2\d x,
\end{align*}
where $\boldsymbol{\omega}:=\nabla\times\mathpzc{U}$ is the vorticity vector.

The HJBE is a nonlinear partial differential equation of first-order that generally does not have a smooth solution, even in finite-dimension (see \cite[Example 8.1, Chapter 8]{ABBP}). Therefore, mathematicians approach these problems by looking into the non-smooth or generalized solution, usually the solution in the Sobolev space $\W^{1,\infty}_{loc}$, which satisfies the HJBE in almost everywhere sense. 
In this direction, plenty of literature is available concerning a solution for such an equation; for instance, see \cite{SHB,AD,WHF,AVF,PLL}. 
Unfortunately, this solution concept is too weak and is not necessarily unique (see \cite[Chapter 8]{ABBP} for counterexamples). 

To overcome such difficulties, Crandall and Lions introduced the notion of viscosity solution in \cite{MGL}. They proved the existence of a unique viscosity solution under various hypotheses. Later, Crandall, Evans and Lions in \cite{MGEL} systematically explore several equivalent formulations of this notion in a simplified form. Lions first established in \cite{PLL} that the dynamic programming principle in the optimal control problem implies that the value function serves as the viscosity solution to the associated HJBE. This fact, combined with the uniqueness of the viscosity solution, provides a complete characterization of the value function. It is important to note that the uniqueness of the viscosity solution discussed in the works \cite{MGL, MGEL} pertains to bounded continuous functions. However, the value function may be unbounded in the context of optimal control problems. Consequently, establishing the uniqueness of the viscosity solution in these cases is more complicated. Ishii first addressed this in \cite{Hsh1} (see also \cite{MGHL1}), where he resolved this issue by establishing the comparison principle. Consequently, this proves the uniqueness of the viscosity solution, which is uniformly continuous in the spatial variable (and so possibly unbounded). Further, the authors in \cite{MGL6} obtained the existence and uniqueness of viscosity solutions in this class by using the tool of modulus of continuity of the solutions. 

There has been growing interest in and an expanding literature on the HJBE in the infinite-dimensional space in recent years. The analysis of the HJBE in infinite dimensions is more challenging because several mathematical arguments do not hold, unlike in finite dimensions. Barbu and Da Prato pioneered the foundational work in this area (for instance, see \cite{VBGDP}). They approached the problem using the semigroup and perturbation techniques within the specific class of convex functions. Later, Crandall and Lions developed the theory of viscosity solutions for infinite-dimensional HJBE with bounded Hamiltonian in a series of articles \cite{MGL1, MGL2, MGL3}. 
The existence and uniqueness of viscosity solutions for the HJB equation
with an unbounded Hamiltonian are proven in \cite{MGL4} and further explored in \cite{DT1, DT2, DT3}. Due to the non-smooth nature of these solutions, deriving an explicit feedback formula for optimal control remains a challenge.
However, in the context of optimal control  problems related to two-dimensional CBF equations, the author in \cite{TiOp2, Op1, Op3}  obtained the existence of  optimal controls and established the first-order necessary optimality conditions. Interested readers can refer \cite{MBCD} for the comprehensive study of first order HJBE and viscosity solution. 
 
Let us compare our work from \cite{FGSSA1} and \cite{SSS1}.
The HJBE for NSE involves the unbounded and bilinear operators arising from the diffusion and convective terms (see \eqref{CBF}), respectively. Consequently, finding a viscosity solution for such an HJBE is more challenging. As noted in \cite{SSS1}, the question of global solvability for the infinite-dimensional HJBE related to the NSE remained unsolved. This problem was first addressed in \cite{FGSSA1} but only in the two-dimensional case. However, the solvability of HJBE associated with 3D-NSE is not yet resolved due to the lack of global well-posedness results for a strong solution. In the current work, we are focusing on CBF equations, which include an absorption term $|\mathpzc{U}|^{r-1}\mathpzc{U}$ (as outlined in \eqref{CBF}), which significantly complicates the system and making the associated HJBE much more challenging to solve than those related to the NSE, which omits this term. 
  However, in dimensions $d\in\{2,3\}$, for supercritical case $(r>3)$ and for critical case $r=3$ with $2\beta\mu\geq1$, the absorption term $|\mathpzc{U}|^{r-1}\mathpzc{U}$, along with the  diffusion term (that is, $-\Delta\mathpzc{U}$), dominates the convective nonlinearity (see Remark \ref{BLrem}). This dominance ensures the global solvability of the CBF equations, allowing us to extend the findings of \cite{FGSSA1} to the 3D-CBF equations.

%

\subsection{Difficulties, strategies and approaches}\label{sub-tech}
Unlike the case of the whole space or a periodic domain, working with a bounded domain presents additional challenges. One of the main difficulties is that the term $\mathcal{P}(|\mathpzc{U}|^{r-1}\mathpzc{U})$  generally does not vanish on the boundary. Moreover, the Helmholtz-Hodge projection operator $\mathcal{P}$ does not necessarily commute with the Laplacian $-\Delta$ (see \cite{JCR}). Therefore, the identity
\begin{align}\label{toreq}
	&\int_{\mathcal{O}}(-\Delta\mathpzc{U}(\xi))\cdot|\mathpzc{U}(\xi)|^{r-1}\mathpzc{U}(\xi)\d \xi \nonumber\\&=\int_{\mathcal{O}}|\nabla\mathpzc{U}(\xi)|^2|\mathpzc{U}(\xi)|^{r-1}\d \xi+ \frac{r-1}{4}\int_{\mathcal{O}}|\mathpzc{U}(\xi)|^{r-3}|\nabla|\mathpzc{U}(\xi)|^2|^2\d \xi,
\end{align}
may not be directly applicable in the bounded domains. However, in $\mathbb{T}^d$, the projection $\mathcal{P}$ and the Laplacian $-\Delta$ do commute (cf. \cite[Theorem 2.22]{JCR}), making the identity \eqref{toreq} useful in deriving regularity results.

In this work, our significant difficulty is establishing the existence and uniqueness of the viscosity solutions, identified with value function, of the associated first-order HJBE for the optimal control problem of CBF equations \eqref{CBF}. The existence of a viscosity solution followed by using the dynamic programming techniques of optimal control problems; meanwhile, the uniqueness follows from the comparison principle. We follow the approach of the works \cite{FGSSA, FGSSA1}.

While proving the existence of viscosity solution, the main challenge is demonstrating the \emph{supersolution inequality}. Unlike NSE, we have to deal with the extra terms arising from convective and absorption terms in the CBF model \eqref{CBF}. 
 As compared with \cite{FGSSA, FGSSA1}, the following are the main mathematical issues that we have to face additionally:
\begin{itemize}
	\item We first emphasize that the definition of the test function (Definition \ref{testD}) considered here, slightly modified from the one considered in \cite{FGSSA}, plays an essential role in the analysis of our work.
	\item The authors in \cite{FGSSA} consider the $\|\cdot\|_{\V}-$norm in the definition of the test function for viscosity solution. Due to the zero average condition,  the $\|\cdot\|_{\V}-$norm is equivalent to $\|\nabla\cdot\|_{\H}-$ norm (see Subsection \ref{zerofunc} for the definition of $\H$ and $\V$-space). Meanwhile, for CBF equations zero-average condition does not hold in $\mathbb{T}^d$ (see Subsection \ref{zerofunc} on functional setting), and therefore, we have to deal with
	 the full $\V-$norm, that is, $\|\z\|_{\V}^2=\|\z\|_{\H} ^2 + \|\nabla\z\|_{\H}^2$.  
	\item  Therefore, it is essential to note that, while applying the chain rule to the function $\delta(\cdot)\|\Z_\lambda(\cdot)\|_{\V}^2$ (see \eqref{vdp3}-\eqref{vdp3.1}) in the proof of the Theorem \ref{extunqvisc}, the derivative of $\|\cdot\|_{\V}^2-$norm is different from the case of NSE.
	In particular, we have following Fr\'echet derivatives in $\mathbb{T}^d$:
	\begin{align}
		\mathcal{D}_{\z}(\|\z\|_{\V}^2)&=2\mathcal{A}\z, \ \text{ for NSE }, 
		\label{frA}\\
		\mathcal{D}_{\z}(\|\z\|_{\V}^2)&=2(\mathcal{A}+\I)\z, \ \text{ for CBF equations},\label{frAA}
	\end{align}
	where $\mathcal{A}$ is the linear operator defined in \eqref{stokesope} (see Subsection \ref{linope}).
	\item Moreover, in the work \cite{FGSSA1}, authors consider the $\H$-norm in the test function definition. Therefore, the only contribution comes from the Fr\'echet derivative of $\|\cdot\|_{\H}^2$, is the identity operator, that is, $\D_{\z}(\|\z\|_{\H}^2)=2\z$. Therefore, while applying the chain rule, no extra term appears due to the fact \eqref{syymB}.
\end{itemize}

We now mention how the above issues influence the analysis of our work. 
First, since zero average condition does not hold for the CBF system \eqref{CBF}, therefore $\mathcal{A}$ is not invertible as $0$ is an eigenvalue. However, $\mathcal{A}+\I$ is invertible. We use this fact to get
the weak convergence \eqref{wknm2} (see \eqref{wknmdiff}), which is the most important step in the proof of Theorem \ref{extunqvisc}.

As pointed out above, the additional difficulty arising in our work is to handle the following integrals while proving Theorem \ref{extunqvisc} (existence of viscosity solution):
\begin{itemize}
	\item The integral involving bilinear operator $\mathfrak{B}(\cdot)$, that is
	\begin{align}\label{extunqdff1}
	\frac{1}{\lambda}\int_{t_0}^{t_0+\lambda}\delta(s)(\mathfrak{B}(\Z_\lambda(s)),(\mathcal{A}+\I)\Z_\lambda(s))\d s.
	\end{align}
	\item The integral involving nonlinear operator $\mathfrak{C}(\cdot)$, that is
	\begin{align}\label{extunqdff2}
	 \frac{1}{\lambda}\int_{t_0}^{t_0+\lambda}\delta(s)(\mathfrak{C}(\Z_\lambda(s)),(\mathcal{A}+\I)\Z_\lambda(s))\d s.
	\end{align}
\end{itemize}
For NSE on $\mathbb{T}^2$, the terms \eqref{extunqdff1} and \eqref{extunqdff2} do not appear since the condition $(\mathfrak{B}(\Z),\mathcal{A}\Z)=0$ holds (see \cite{FGSSA}). Furthermore, in the work \cite{FGSSA1}, due to the presence of $\H$-norm in the definition of test function, such term does not appear since then we have $(\mathfrak{B}(\Z),\Z)=0$.
 Unfortunately, $(\mathfrak{B}(\Z),\mathcal{A}\Z)\neq0$ in 3D, and therefore, in our case, for CBF equations, we need to address and include the terms \eqref{extunqdff1} and \eqref{extunqdff2}.
 We use the following approach to sort out this difficulty:
 \begin{itemize}
 	\item We can justify the integral \eqref{extunqdff1} as `$\lambda\to0$' with the help of energy estimates \eqref{eqn-conv-1}, \eqref{eqn-conv-2} and continuous dependence estimates \eqref{ctsdep0.1}-\eqref{ctsdep0.2} (see Appendix \ref{wknBA1}) in both two and three-dimensions.
 	\item  Meanwhile, the integral \eqref{extunqdff2} is hard to justify as `$\lambda\to0$' for all $r>3$. This difficulty appears while applying H\"older's inequality to the time integral, and then the exponent $\frac{4(r+1)}{5-r}$ pops up, which forces us to restrict $3\leq r<5$ in three dimensions (see Case-II of Appendix \ref{wknCA}). While in two dimensions, such restriction is not there due to the Sobolev embedding $\V\hookrightarrow\wi\L^{r+1}$ for any $r\geq1$. 
 	\item The case $r=5$ is the most important in our work, and we discussed it separately in Appendix \ref{casere5}. Note that the definition of viscosity solution \ref{viscsoLndef} requires 
 	$\z\in\V_2=\D (\I+\mathcal{A})$. Once we establish this (see the first part of Theorem \ref{extunqvisc}), we then resolve this case by masking the use of energy estimates in $\D(\mathcal{A})$ (Proposition \ref{DAENG}) and the Gr\"onwall's inequality (see Subsection \ref{dar5}). \emph{However, in the case of stochastic CBF, where the associated HJBE is of second-order, we cannot justify the case $r=5$. It is due to the appearance of the expectation in the calculations (while deriving energy estimates) and lack of Gr\"onwall' inequality (see \cite{smtm1}).}
 \end{itemize}

Finally, we mention that the well-posedness results for strong as well as weak solutions of the problem \eqref{CBF} in dimensions $d\in\{2,3\}$, for supercritical case $(r>3)$ and for critical case $r=3$ with $2\beta\mu\geq1$, is known in the literature. 
\begin{table}[ht]
	\begin{tabular}{|c|c|c|c|c|}
		\hline
		\textbf{Case}&Dimension &$ r$& Conditions on 
		$\mu$ \& $\beta$ \\
		\hline
		\textbf{I}&$d=2,3$ &$r\in(3,\infty)$&  for any  
		$\mu>0$ and $\beta>0$  \\
		\hline
		\textbf{II}&$d=2,3$ &$r=3$&for $\mu>0$ and
		$\beta>0$ with $2\beta\mu\geq1$ \\
		\hline
	\end{tabular}
	\vskip 0.1 cm
	\caption{Values of $\mu,\beta$ and $r$ for the comparison principle.}
	\label{Table1}
\end{table}
Due to some technical difficulties (see Subsection \ref{sub-tech} below), we restrict the values of $r$ to prove \emph{the comparison principle} (Table \ref{Table1}) and to prove the \emph{existence of viscosity solutions} (Table \ref{Table2}), which are the main goals of this work.
\begin{table}[ht]
	\begin{tabular}{|c|c|c|c|c|}
		\hline
		\textbf{Case}&Dimension &$ r$& Conditions on 
		$\mu$ \& $\beta$ \\
		\hline
		\textbf{I}&$d=2$ &$r\in(3,\infty)$&  for any  
		$\mu>0$ and $\beta>0$  \\
		\hline
		\textbf{II}&$d=3$ &$r\in(3,5)$&  for any  
		$\mu>0$ and $\beta>0$  \\
		\hline
		\textbf{III}&$d=3$ &$r=5$&for any  
		$\mu>0$ and $\beta>0$ \\
		\hline
	\end{tabular}
	\vskip 0.1 cm
	\caption{Values of $\mu,\beta$ and $r$ for the existence of viscosity solution.}
	\label{Table2}
\end{table}

\subsection{Organization of the paper} 
The remainder of the paper is organised as follows: In the next section, we introduce the functional framework and the necessary function spaces to analyse CBF equations and the associated HJBE. Section \ref{abscon} is devoted to the abstract formulation of the CBF equations \eqref{CBF}, where we define both weak and strong solutions (Definition \ref{def-var-strong}) of the system \eqref{stap}. We derive energy estimates (Proposition \ref{weLLp}) and continuous dependence estimates for the solution $\Z(\cdot)$ of the system \eqref{stap} (Proposition \ref{cts-dep-soln}). In Section \ref{detstchjb}, we formulate the optimal control problem and state the continuous dependence result for the value function and dynamic programming principle (Proposition \ref{valueP}). We then define viscosity solution for the HJBE \eqref{HJBE} (Subsection \ref{viscPR}). Section \ref{compPR} is the core of this work, and we prove the \emph{comparison principle} (Theorem \ref{comparison}). In the last Section \ref{extunqvisc1}, we establish the existence and uniqueness of the viscosity solution for the HJBE \eqref{detHJB1} (Theorem \ref{extunqvisc}).

	\section{Functional Framework}\label{Sec-2}\setcounter{equation}{0}
	In this section, we introduce some necessary function spaces that will be used throughout the paper, along with the linear and nonlinear operators required to obtain the abstract formulation of the CBF system \eqref{CBF}. The functional framework employed here, is adapted from the work \cite{JCR1}.

	\subsection{Function spaces}\label{zerofunc}
	Let us denote by $\C_{\mathrm{p}}^{\infty}(\mathbb{T}^d;\R^d)$, the space of all infinitely differentiable  functions $\mathpzc{U}$ satisfying \eqref{pbc}. \emph{We do not impose the zero mean condition on $\mathpzc{U}$ because the absorption term $\beta|\mathpzc{U}|^{r-1}\mathpzc{U}$ does not satisfies this property (see \cite{MTT}). As a result, the standard Poincar\'e inequality is not applicable, and we must work with the full $\H^1$-norm instead.} 
	The completion of $\C_{\mathrm{p}}^{\infty}(\mathbb{T}^d;\R^d)$  with respect to the $\H^s$-norm is the Sobolev space $\H_{\mathrm{p}}^s(\mathbb{T}^d):=\mathrm{H}_{\mathrm{p}}^s(\mathbb{T}^d;\mathbb{R}^d)$.  From \cite[Proposition 5.39]{JCR1}, the Sobolev space of periodic functions, that is $\H_{\mathrm{p}}^s(\mathbb{T}^d)$, is same as the following: 
	$$\left\{\mathpzc{U}:\mathpzc{U}=\sum_{k\in\mathbb{Z}^d}\mathpzc{U}_{k}\mathrm{e}^{2\pi i k\cdot x /  \mathrm{L}},\ \overline{\mathpzc{U}}_{k}=\mathpzc{U}_{-k}, \  \|\mathpzc{U}\|_{{\H}^s_f}:=\left(\sum_{k\in\mathbb{Z}^d}(1+|k|^{2s})|\mathpzc{U}_{k}|^2\right)^{1/2}<\infty\right\}.$$ 
	Let us define 
	\begin{align*} 
		\mathcal{V}:=\{\mathpzc{U}\in\C_{\mathrm{p}}^{\infty}(\mathbb{T}^d;\R^d):\nabla\cdot\mathpzc{U}=0\}.
	\end{align*}
	We denote by $\H$, the closure of $\mathcal{V}$ in the Lebesgue space $\L^2(\mathbb{T}^d):=\mathrm{L}^2(\mathbb{T}^d;\R^d)$ and by $\widetilde{\L}^{p}$, the closure of $\mathcal{V}$ in the Lebesgue space $\L^p(\mathbb{T}^d):=\mathrm{L}^p(\mathbb{T}^d;\R^d)$ for $p\in(2,\infty]$, respectively. We endow the space $\H$ with the inner product and norm of $\L^2(\mathbb{T}^d),$ and are denoted by 
	\begin{align*}
		&(\mathpzc{U},\mathpzc{Y}):=(\mathpzc{U},\mathpzc{Y})_{\L^2(\mathbb{T}^d)}=\int_{\mathbb{T}^d}\mathpzc{U}(x)\cdot\mathpzc{Y}(x)\d x\\ \text{ and } \  &\|\mathpzc{U}\|_{\H}^2:=\|\mathpzc{U}\|_{\L^2(\mathbb{T}^d)}^2=\int_{\mathbb{T}^d}|\mathpzc{U}(x)|^2\d x, \ \text{ for } \ \mathpzc{U},\v\in\H.
	\end{align*}
	For $p\in(2,\infty)$, the space $\widetilde{\L}^{p}$ is equipped with the following norm of $\L^p(\mathbb{T}^d)$:
	\begin{align*}
		\|\mathpzc{U}\|_{\widetilde{\L}^p}^p:
		=\int_{\mathbb{T}^d}|\mathpzc{U}(x)|^p\d x \  \text{ for } \ \mathpzc{U}\in\wi\L^p.
	\end{align*}
	For $p=\infty$, the space $\widetilde{\L}^{\infty}$ is equipped with the following norm of $\L^{\infty}(\mathbb{T}^d)$:
		\begin{align*}
		\|\mathpzc{U}\|_{\widetilde{\L}^{\infty}}:
		=\esssup_{x\in\mathbb{T}^d}|\mathpzc{U}(x)| \  \text{ for } \ \mathpzc{U}\in\wi\L^{\infty}.
	\end{align*}
	We also define the space $\V$ as the closure of $\mathcal{V}$ in the Sobolev space $\H^1_{\mathrm{p}}(\mathbb{T}^d)$. We equip the space $\V$ with the inner product
	\begin{align*}
		(\mathpzc{U},\mathpzc{Y})_{\V}&:=(\mathpzc{U},\mathpzc{Y})_{\L^2(\mathbb{T}^d)}+(\nabla\mathpzc{U},\nabla\mathpzc{Y})_{\L^2(\mathbb{T}^d)}\nonumber\\&=
		\int_{\mathbb{T}^d}\mathpzc{U}(x)\cdot\mathpzc{Y}(x)\d x+\int_{\mathbb{T}^d}\nabla\mathpzc{U}(x)\cdot\nabla\mathpzc{Y}(x)\d x\ \text{ for } \ \mathpzc{U},\mathpzc{Y}\in\V,
     \end{align*}
     and the norm
     \begin{align*}
		\|\mathpzc{U}\|_{\V}^2:=\|\mathpzc{U}\|_{\L^2(\mathbb{T}^d)}^2+\|\nabla\mathpzc{U}\|_{\L^2(\mathbb{T}^d)}^2=\int_{\mathbb{T}^d}|\mathpzc{U}(x)|^2\d x+\int_{\mathbb{T}^d}|\nabla\mathpzc{U}(x)|^2\d x \ \text{ for } \ \mathpzc{U},\v\in\V.
	\end{align*}
%
	Let $\langle \cdot,\cdot\rangle $ denote the duality pairing between the spaces $\V$  and its dual $\V^*$, as well as between $\widetilde{\L}^p$ and its dual $\widetilde{\L}^{p'}$, where $\frac{1}{p}+\frac{1}{p'}=1$. Note that $\H$ can be identified with its own dual $\H^*$. According to \cite[Subsection 2.1]{FKS}, the sum space $\V^*+\widetilde{\L}^{p'}$ is well defined and forms a Banach when equipped with the norm 
	\begin{align*}
		\|\mathpzc{U}\|_{\V^*+\widetilde{\L}^{p'}}&:=\inf\{\|\mathpzc{U}_1\|_{\V^*}+\|\mathpzc{U}_2\|_{\wi\L^{p'}}:\mathpzc{U}=\mathpzc{U}_1+\mathpzc{U}_2, \mathpzc{U}_1\in\V^* \ \text{and} \ \mathpzc{U}_2\in\wi\L^{p'}\}\nonumber\\&=
		\sup\left\{\frac{|\langle\mathpzc{U}_1+\mathpzc{U}_2,\boldsymbol{\eta}\rangle|}{\|\boldsymbol{\eta}\|_{\V\cap\widetilde{\L}^p}}:\boldsymbol{0}\neq\boldsymbol{\eta}\in\V\cap\widetilde{\L}^p\right\},
	\end{align*}
	where the norm $\|\cdot\|_{\V\cap\widetilde{\L}^p}$ on the intersection space $\V\cap\wi\L^p$ is defined by
	$$\|\cdot\|_{\V\cap\widetilde{\L}^p}:=\max\{\|\cdot\|_{\V}, \|\cdot\|_{\wi\L^p}\}.$$ 
	Moreover, the norm $\|\cdot\|_{\V\cap\widetilde{\L}^p}$ is equivalent to the norms $\|\cdot\|_{\V}+\|\cdot\|_{\widetilde{\L}^{p}}$ and $\sqrt{\|\cdot\|_{\V}^2+\|\cdot\|_{\widetilde{\L}^{p}}^2}$. Furthermore, we have
	$$
	(\V^*+\widetilde{\L}^{p'})^*\cong	\V\cap\widetilde{\L}^p \  \text{and} \ (\V\cap\widetilde{\L}^p)^*\cong\V^*+\widetilde{\L}^{p'}.
	$$
	Moreover, we have the continuous embeddings $\V\cap\widetilde{\L}^p\hookrightarrow\V\hookrightarrow\H\cong\H^*\hookrightarrow\V^*\hookrightarrow\V^*+\widetilde{\L}^{p'},$ where the embedding $\V\hookrightarrow\H$ is compact. 
	
%

	\subsection{Helmholtz-Hodge projection operator}\label{HHLP}
	The Helmholtz-Hodge projection operator $\mathcal{P}_p: \L^p(\mathbb{T}^d) \to\wi\L^p,$ $p\in[1,\infty)$ (cf.  \cite{DFHM}, etc.)
	is a bounded linear operator that projects vector fields onto the space of divergence free fields. For $p=2$, $\mathcal{P}:=\mathcal{P}_2$ becomes an orthogonal projection in the Hilbert space $\L^2(\mathbb{T}^d)$ (see \cite[Section 2.1]{JCR}). 
	
	\subsection{Stokes operator}\label{linope}
	 The Stokes operator is a linear operator which is defined as
	\begin{equation}\label{stokesope}
		\left\{
		\begin{aligned}
	\mathcal{A}\mathpzc{U}&:=-\mathcal{P}\Delta\mathpzc{U}=-\Delta\mathpzc{U},\;\mathpzc{U}\in\D(\mathcal{A}),\\
			\D(\mathcal{A})&:=\V\cap{\H}^{2}_\mathrm{p}(\mathbb{T}^d). 
		\end{aligned}
		\right.
	\end{equation}
	For a Fourier expansion $\mathpzc{U}=\sum\limits_{k\in\mathbb{Z}^d} e^{2\pi i k\cdot x} \mathpzc{U}_{k} ,$ Parseval's identity gives:
	\begin{align*}
		\|\mathpzc{U}\|_{\H}^2=\sum\limits_{k\in\mathbb{Z}^d} |\mathpzc{U}_{k}|^2 \  \text{and} \ \|\mathcal{A}\mathpzc{U}\|_{\H}^2=(2\pi)^4\sum_{k\in\mathbb{Z}^d}|k|^{4}|\mathpzc{U}_{k}|^2.
	\end{align*}
	Therefore, we have 
	\begin{align*}
		\|\mathpzc{U}\|_{\H^2_\mathrm{p}}^2=\sum_{k\in\mathbb{Z}^d}|\mathpzc{U}_{k}|^2+\sum_{k\in\mathbb{Z}^d}|k|^{4}|\mathpzc{U}_{k}|^2= \|\mathpzc{U}\|_{\H}^2+\frac{1}{(2\pi)^4}\|\mathcal{A}\mathpzc{U}\|_{\H}^2\leq\|\mathpzc{U}\|_{\H}^2+\|\mathcal{A}\mathpzc{U}\|_{\H}^2.
	\end{align*}
	Moreover, by the definition of $\|\cdot\|_{\H^2_\mathrm{p}}$, we have $	\|\mathpzc{U}\|_{\H^2_\mathrm{p}}^2\geq\|\mathpzc{U}\|_{\H}^2+\|\mathcal{A}\mathpzc{U}\|_{\H}^2$. Hence, the norms $\|\cdot\|_{\H^2_\mathrm{p}}$ and $\|\cdot\|_{\H}+\|\mathcal{A}\cdot\|_{\H}$ are equivalent and  $\D(\I+\mathcal{A})=\H^2_\mathrm{p}(\mathbb{T}^d)=:\V_2$.

	\begin{remark}\label{rg3L3r}
   1.) For $d=2$, by using the Sobolev embedding, $\H_{\mathrm{p}}^1(\mathbb{T}^d)\hookrightarrow\L^p(\mathbb{T}^d)$, for all $p\in[2,\infty)$, we find 
		\begin{align}\label{3a71}
			\|\mathpzc{U}\|^{r+1}_{\L^{p(r+1)}(\mathbb{T}^d)}\leq C\bigg(\int_{\mathbb{T}^d}|\nabla\mathpzc{U}(x)|^2|\mathpzc{U}(x)|^{r-1}\d x+ \int_{\mathbb{T}^d}|\mathpzc{U}(x)|^{r+1}\d x\bigg),
		\end{align}
		for all $\mathpzc{U}\in\V_2$ and for any $p\in[2,\infty)$. 
		
		3.) Similarly, for $d=3$, by the Sobolev embedding $\H_{\mathrm{p}}^1(\mathbb{T}^d)\hookrightarrow\L^6(\mathbb{T}^d)$, we find
		\begin{align}\label{371}
			\|\mathpzc{U}\|^{r+1}_{\L^{3(r+1)}(\mathbb{T}^d)}&=\||\mathpzc{U}|^{\frac{r+1}{2}}\|_{\L^{6}(\mathbb{T}^d)}^2\leq C\||\mathpzc{U}|^{\frac{r+1}{2}}\|_{\H_{\mathrm{p}}^1(\mathbb{T}^d)}^2
			\nonumber\\&\leq C\bigg(\int_{\mathbb{T}^d}|\nabla\mathpzc{U}(x)|^2|\mathpzc{U}(x)|^{r-1}\d x+ \int_{\mathbb{T}^d}|\mathpzc{U}(x)|^{r+1}\d x\bigg), 
		\end{align}
		for all $\mathpzc{U}\in\V_2$. 
	\end{remark}

	\subsection{Bilinear operator}\label{bilinope}
		Let $b(\cdot,\cdot,\cdot):\V\times\V\times\V\to\R$ be a continuous trilinear form defined by
		\begin{align*}
			b(\mathpzc{U},\mathpzc{Y},\mathpzc{Z})=\int_{\mathbb{T}^d}(\mathpzc{U}(x)\cdot\nabla)\mathpzc{Y}(x)\cdot
			\mathpzc{Z}(x)\d x.
		\end{align*} 
		By employing the divergence free condition, the trilinear form $\b(\cdot,\cdot,\cdot)$ satisfies $$b(\mathpzc{U},\mathpzc{Y},\mathpzc{Y})=0, \ \text{ for all } \ \mathpzc{U},\mathpzc{Y}\in\V.$$ By the application of Riesz representation theorem, we deduce the existence of a continuous bilinear map $\mathfrak{B}(\cdot,\cdot):\V\times\V\to\R$ such that $$\langle\mathfrak{B}(\mathpzc{U},\mathpzc{Y}),\mathpzc{Z}\rangle=b(\mathpzc{U},\mathpzc{Y},\mathpzc{Z}) \ \text{ for all } \mathpzc{U},\mathpzc{Y},\mathpzc{Z}\in\V.$$ 
		Furthermore, as shown in \cite {Te}, the operator $\mathfrak{B}(\cdot)$ satisfies the following skew-symmetric properties:
		\begin{align}\label{syymB}
			\langle\mathfrak{B}(\mathpzc{U},\mathpzc{Y}),\mathpzc{Z}\rangle=-\langle\mathfrak{B}(\mathpzc{U},\mathpzc{Z}),\mathpzc{Y}\rangle \ \text{ and } \
			\langle\mathfrak{B}(\mathpzc{U},\mathpzc{Y}),\mathpzc{Y}\rangle=0,
		\end{align}
		for any $\mathpzc{U},\mathpzc{Y},\mathpzc{Z}\in\V$. Without ambiguity, we also denote $\mathfrak{B}(\mathpzc{U}) = \mathfrak{B}(\mathpzc{U}, \mathpzc{U})$. Furthermore, if $\mathpzc{U},\mathpzc{Y}\in\H$ are such that $(\mathpzc{U}\cdot\nabla)\mathpzc{Y}=\sum\limits_{j=1}^d \mathpzc{U}_j\frac{\partial \mathpzc{Y}_j}{\partial x_j}\in\L^2(\mathbb{T}^d)$, then $\mathfrak{B}(\mathpzc{U},\mathpzc{Y})=\mathcal{P}[(\mathpzc{U}\cdot\nabla)\mathpzc{Y}]$.

		\begin{remark}\label{BLrem}\cite[Theorem 2.5]{SMTM}
			1.) In view of \eqref{syymB}, along with H\"older's and Young's inequalities, we calculate
			\begin{align}\label{syymB1}
				|\langle\mathfrak{B}(\mathpzc{U})-\mathfrak{B}(\mathpzc{Y}),\mathpzc{U}-\mathpzc{Y}\rangle|\leq
				\frac{\mu }{2}\|\nabla(\mathpzc{U}-\mathpzc{Y})\|_{\H}^2+\frac{1}{2\mu }\||\mathpzc{Y}|(\mathpzc{U}-\mathpzc{Y})\|_{\H}^2.
			\end{align} 
			The term $\||\mathpzc{Y}|(\mathpzc{U}-\mathpzc{Y})\|_{\H}^2$ can be estimated by applying H\"older's and Young's inequalities as follows:
			\begin{align}\label{syymB2}
				\int_{\mathbb{T}^d}|\mathpzc{Y}(x)|^2|\mathpzc{U}(x)-\mathpzc{Y}(x)|^2\d x &\leq\frac{\beta\mu }{2}\||\mathpzc{Y}|^{\frac{r-1}{2}}(\mathpzc{U}-\mathpzc{Y})\|_{\H}^2+\frac{r-3}{r-1}\left[\frac{4}{\beta\mu (r-1)}\right]^{\frac{2}{r-3}}\|\mathpzc{U}-\mathpzc{Y}\|_{\H}^2,
			\end{align}
			for $r>3$. Using \eqref{syymB2} in \eqref{syymB1}, we find 
			\begin{align}\label{3.4}
				|\langle\mathfrak{B}(\mathpzc{U})-\mathfrak{B}(\mathpzc{Y}),\mathpzc{U}-\mathpzc{Y}\rangle|\leq
				\frac{\mu }{2}\|\nabla(\mathpzc{U}-\mathpzc{Y})\|_{\H}^2 +\frac{\beta}{4}\||\mathpzc{Y}|^{\frac{r-1}{2}}(\mathpzc{U}-\mathpzc{Y})\|_{\H}^2 +\varrho\|\mathpzc{U}-\mathpzc{Y}\|_{\H}^2,
			\end{align}
			where \begin{align}\label{eqn-varrho}
				\varrho:=\frac{r-3}{2\mu(r-1)}\left[\frac{4}{\beta\mu (r-1)}\right]^{\frac{2}{r-3}}.
				\end{align}
			
			2.) Similarly, one can establish the following inequality:
			\begin{align}\label{syymB3}
				|(\mathfrak{B}(\mathpzc{U}),\mathcal{A}\mathpzc{U})|\leq\frac{\mu}{2}\|\mathcal{A}\mathpzc{U}\|_{\H}^2+\frac{\beta}{4} \||\mathpzc{U}|^{\frac{r-1}{2}}\nabla\mathpzc{U}\|_{\H}^2+\varrho\|\nabla\mathpzc{U}\|_{\H}^2.
			\end{align}
		\end{remark}

	 \subsection{Nonlinear operator}\label{nonlinope}
		We define the operator $$\mathfrak{C}(\mathpzc{U}):=\mathcal{P}(|\mathpzc{U}|^{r-1}\mathpzc{U})\ \text{ for }\ \mathpzc{U}\in\V\cap\L^{r+1}.$$  From \cite[Remark 1.6]{Te}, the operator $\mathfrak{C}(\cdot):\V\cap\widetilde{\L}^{r+1}\to\V^*+\widetilde{\L}^{\frac{r+1}{r}}$ is well-defined. Moreover, it satisfies the identity $\langle\mathfrak{C}(\mathpzc{U}),\mathpzc{U}\rangle =\|\mathpzc{U}\|_{\widetilde{\L}^{r+1}}^{r+1}.$ Furthermore, for all $\mathpzc{U}\in\V\cap\L^{r+1}$, the map $\mathfrak{C}(\cdot):\V\cap\widetilde{\L}^{r+1}\to\V^*+\widetilde{\L}^{\frac{r+1}{r}}$ is Gateaux differentiable with Gateaux derivative given by 
		\begin{align}\label{C}
			\mathfrak{C}'(\mathpzc{U})\mathpzc{Y}&=\left\{\begin{array}{cl}\mathcal{P}(\mathpzc{Y}),&\text{ for }r=1,\\ \left\{\begin{array}{cc}\mathcal{P}(|\mathpzc{U}|^{r-1}\mathpzc{Y})+(r-1)\mathcal{P}\left(\frac{\mathpzc{U}}{|\mathpzc{U}|^{3-r}}(\mathpzc{U}\cdot\mathpzc{Y})\right),&\text{ if }\mathpzc{U}\neq \mathbf{0},\\\mathbf{0},&\text{ if }\mathpzc{U}=\mathbf{0},\end{array}\right.&\text{ for } 1<r<3,\\ \mathcal{P}(|\mathpzc{U}|^{r-1}\mathpzc{Y})+(r-1)\mathcal{P}(\mathpzc{U}|\mathpzc{U}|^{r-3}(\mathpzc{U}\cdot\mathpzc{Y})), &\text{ for }r\geq 3,\end{array}\right.
		\end{align}
		for all $\mathpzc{Y}\in\V\cap\widetilde{\L}^{r+1}$. 
		\begin{lemma}{(\textbf{Monotonicity of} $\mathfrak{C}(\cdot)$)}
			For every $r\geq1$ and for all $\mathpzc{U},\mathpzc{Y}\in\widetilde{\L}^{r+1}$, we have following estimate ({cf. \cite[Subsection 2.4]{SMTM}}):
			\begin{align}\label{monoC2}
				\langle\mathfrak{C}(\mathpzc{U})-\mathfrak{C}(\mathpzc{Y}),\mathpzc{U}-\mathpzc{Y}\rangle\geq \frac{1}{2}\||\mathpzc{U}|^{\frac{r-1}{2}}(\mathpzc{U}-\mathpzc{Y})\|_{\H}^2+\frac{1}{2}\||\mathpzc{Y}|^{\frac{r-1}{2}}(\mathpzc{U}-\mathpzc{Y})\|_{\H}^2\geq\frac{1}{2^{r-1}}\|\mathpzc{U}-\mathpzc{Y}\|_{\wi\L^{r+1}}^{r+1}.
			\end{align}
	\end{lemma} 
	
	\begin{remark}[{\cite[Lemma 2.1]{KWH}}]
		We have following equality on $\mathbb{T}^d$:
		\begin{align}\label{torusequ}
	(\mathfrak{C}(\mathpzc{U}),\mathcal{A}\mathpzc{U})=\||\mathpzc{U}|^{\frac{r-1}{2}}\nabla\mathpzc{U}\|_{\H}^{2} +4\left[\frac{r-1}{(r+1)^2}\right]\|\nabla|\mathpzc{U}|^{\frac{r+1}{2}}\|_{\H}^{2}.
		\end{align}
	\end{remark}

	\subsection{Some useful functional inequalities}
	\label{defunfmod}
	We utilize the following inequalities throughout this paper: 
	
	1.) \emph{Interpolation inequality:} Let $0\leq s_1\leq s\leq s_2\leq\infty$ and $0\leq\theta\leq1$ be such that $\frac{1}{s}=\frac{\theta}{s_1}+\frac{1-\theta}{s_2}$. Then for $\mathpzc{U}\in\L^{s_2}(\mathbb{T}^d)$, we have
	\begin{align*}
		\|\mathpzc{U}\|_{\L^s}\leq\|\mathpzc{U}\|_{\L^{s_1}}^{\theta}\|\mathpzc{U}\|_{\L^{s_2}}^{1-\theta}.
	\end{align*} 
	
	2.) \emph{Interpolation inequality in $\D(\mathcal{A}^{\gamma})$ with $\gamma\in\R$:} Let $\gamma_1,\gamma_2,\gamma_3\in\R$ with $\gamma_2=\theta\gamma_1+(1-\theta)\gamma_3$, where $\theta\in[0,1]$. Then, for $\mathpzc{U}\in\D(\mathcal{A}^{\gamma_1})\cap\D(\mathcal{A}^{\gamma_3})$, we have
	\begin{align}\label{daint}
		\|\mathcal{A}^{\gamma_2}\mathpzc{U}\|_{\H}\leq\|\mathcal{A}^{\gamma_1}\mathpzc{U}\|_{\H}^{\theta}	\|\mathcal{A}^{\gamma_3}\mathpzc{U}\|_{\H}^{1-\theta}.
	\end{align}

	
	4.)  \emph{Vector valued Leibniz rule for fractional derivatives \cite{PDAnc}:}
	Let $\gamma\in(0,1)$ and $p_1,p_2,q_1,q_2$ in $(1,\infty)$ and $r\in(1,\infty]$ such that 
	$\frac{1}{q_1}+\frac{1}{q_2}=\frac{1}{r}=\frac{1}{p_1}+\frac{1}{p_2}$. Then, we have
	\begin{align}\label{fracLeb}
		\|\mathcal{A}^{\frac{\gamma}{2}}(\mathpzc{U}\mathpzc{Y})\|_{\L^r}\leq C
		\|\mathcal{A}^{\frac{\gamma}{2}}\mathpzc{U}\|_{\L^{p_1}}\|\mathpzc{Y}\|_{\L^{p_2}}+
		\|\mathpzc{U}\|_{\L^{q_1}}\|\mathcal{A}^{\frac{\gamma}{2}}\mathpzc{Y}\|_{\L^{q_2}}.
	\end{align}
	
	5.) \emph{Agmon's inequality:} For all $\mathpzc{U}\in\H^2_{\mathrm{p}}(\mathbb{T}^d)$, $d\in\{2,3\}$, we have 
	\begin{align}\label{agmon}
		\|\mathpzc{U}\|_{\L^{\infty}(\mathbb{T}^d)}\leq C\|\mathpzc{U}\|_{\H}^{1-\frac{d}{4}}
		\|\mathpzc{U}\|_{\H^2_{\mathrm{p}}(\mathbb{T}^d)}^{\frac{d}{4}}
		=\|\mathpzc{U}\|_{\H}^{1-\frac{d}{4}}
		\|(\I+\mathcal{A})\mathpzc{U}\|_{\H}^{\frac{d}{4}}.
	\end{align}

	\section{Abstract formulation and Auxiliary results}\label{abscon} \setcounter{equation}{0}
	We begin this section by presenting an abstract formulation of the system \eqref{CBF} and establishing uniform energy estimates, followed by proofs of continuous dependence results. Let $\U$ denote a complete metric space. A measurable function $\a(\cdot):[0,T]\to\U$ is referred as \emph{admissible control}. Collecting all such controls, denoted by $\mathscr{U}$, forms the \emph{admissible class}.
   Let us set $\Z(\cdot):=\mathcal{P}\mathpzc{U}(\cdot)$, $\mathcal{P}\mathpzc{U}_0:=\z$,  $\mathcal{P}\g=\f$. On projecting the first equation in \eqref{CBF}, we obtained the following abstract controlled CBF equations for $\a(\cdot)\in\mathscr{U}$, :
	\begin{equation}\label{stap}
		\left\{
		\begin{aligned}
			\frac{\d\Z(s)}{\d s}&= -\mu\mathcal{A}\Z(s)-\mathfrak{B}(\Z(s))-\beta\mathfrak{C}(\Z(s))+
			\f(s,\a(s)),  \  \text{ in } \ (t,T]\times\H, \\
			\Z(t)&=\z\in\H.
		\end{aligned}
		\right.
	\end{equation}
  \begin{hypothesis}\label{fhypdet}
Assume that $\f:[0,T]\times\U\to\V$ be bounded and continuous. 
  \end{hypothesis}
  The above hypothesis is to remember that our optimal control problem is motivated by the minimization of enstrophy, as mentioned in the introduction. So, the cost functional associated with the optimal control problem contains the gradient term.	Let us first give the notions of solutions for the controlled CBF equations \eqref{stap} (cf. \cite{SMTM}):
	\begin{definition}\label{def-var-strong} 
		
		(i) Let $\a(\cdot)\in\mathscr{U}$. Then, a function $$\Z(\cdot)\in\mathrm{C}([t,T];\H)\cap\mathrm{L}^2(t,T;\V)\cap\mathrm{L}^{r+1}(t,T;\wi\L^{r+1}),$$
		with $\frac{\d\Z}{\d s}\in\mathrm{L}^2(t,T;\V')+ \mathrm{L}^{\frac{r+1}{r}}(t,T;\wi\L^{\frac{r+1}{r}})$ is called a \emph{weak solution} to \eqref{stap}, if for a given $\f(\cdot,\a(\cdot))\in\mathrm{L}^2(t,T;\V^{*})$ and $\Z(t)=\z\in\H$, we have 
		$$\sup_{s\in[t,T]}\|\Z(s)\|_{\H}^2+\int_t^T \|\nabla\Z(s)\|_{\H}^2\d s + \int_t^T\|\Z(s)\|_{\wi\L^{r+1}}^{r+1}\d s<+\infty,$$ and the function $\Z(\cdot)$ satisfies the following weak formulation:
		\begin{align*}
			\langle\Z(s),\phi\rangle&=\langle\y,\phi\rangle+\int_t^s  \langle-\mu\mathcal{A}\Z(\tau)-\mathfrak{B}(\Z(\tau))-\beta\mathfrak{C}(\Z(\tau))+\f(\tau,\a(\tau)),\phi\rangle\d\tau,
		\end{align*} 
		  for every $\phi\in\V\cap\wi\L^{r+1}$ and every $s\in[t,T]$.
		  
		(ii) Let $\a(\cdot)\in\mathscr{U}$. Then, a function 
		$$\Z(\cdot)\in\mathrm{C}([t,T];\V)\cap\mathrm{L}^2(t,T;\V)\cap
		\mathrm{L}^{r+1}(t,T;\wi\L^{p(r+1)})$$ is called a \emph{strong solution} of \eqref{stap} with $\Z(t)=\z\in\V$ if $\Z(\cdot)$ is a weak solution and satisfies $$\sup_{s\in[t,T]}\|\nabla\Z(s)\|_{\H}^2+\int_t^T \|\mathcal{A}\Z(s)\|_{\H}^2\d s+\int_0^T\|\Z(s)\|_{\wi\L^{p(r+1)}}^{r+1}\d s<+\infty,$$  where $p\in[2,\infty)$ for $d=2$ and $p=3$ for $d=3$.
		Moreover, \eqref{stap} is satisfied as an equality in $\H$ for a.e. $s\in[t,T]$.
	\end{definition}

    For a given time interval $s\in[t, T]$, an admissible control strategy $\a(\cdot)\in\mathscr{U}$ and an initial condition $\z\in\H$ (or $\z\in\V$), we denote the weak (or strong) solution of \eqref{stap} by $\Z(\cdot;t,\z,\a)$. We write $\Z(\cdot)$ for brevity.
    Below, we give well-posedness results for the weak and strong solutions of the controlled CBF equations \eqref{stap} and derive some energy estimates.
     \begin{proposition}\label{weLLp}
      Let $t\in[0,T]$ and $\a(\cdot)\in\mathscr{U}$.
      
     	(i)  Let $\f:[0,T]\times\U\to\V^{*}$ be bounded and continuous. Then, for a given initial data $\z$, there exists a \emph{unique weak solution} $\Z(\cdot)=\Z(\cdot;t,\z,\a(\cdot))$ of the state equation \eqref{stap}. Moreover, the following uniform energy estimates for $s\in[t,T]$ holds:
     	 \begin{align}\label{eqn-conv-1-1}
     		&\|\Z(s)\|_{\H}^2
     		+\int_t^{s}\|\nabla\Z(\tau)\|_{\H}^2\d\tau+
     		\int_t^{s}\|\Z(\tau)\|_{\widetilde{\L}^{r+1}}^{r+1}\d\tau
     		\nonumber\\&\leq C(\mu,\beta)
     		\left(\|\z\|_{\H}^2+\int_t^s\|\f(\tau,\a(\tau))\|_{\V^{*}}^2\d\tau \right)
     		e^{s-t},
     	\end{align}
     	for all $s\in[t,T]$ and
     	\begin{align}\label{eqn-conv-1}
     &\sup\limits_{s\in[t,T]}\|\Z(s)\|_{\H}^2
     	+\int_t^{T}\|\nabla\Z(\tau)\|_{\H}^2\d\tau+
     	\int_t^{T}\|\Z(\tau)\|_{\widetilde{\L}^{r+1}}^{r+1}\d\tau
     	\nonumber\\&\leq C(\mu,\beta,T)
     	\left(\|\z\|_{\H}^2+\int_t^T\|\f(\tau,\a(\tau))\|_{\V^{*}}^2\d \tau\right).
     	\end{align}
     	
     	(ii)  Let $\f:[0,T]\times\U\to\H$ be bounded and continuous. Then, for a given initial data $\z\in\V$, there exists a \emph{unique strong solution} $\Z(\cdot)=\Z(\cdot;t,\z,\a(\cdot))$ of the state equation \eqref{stap}. Moreover, the following uniform energy estimates hold:  for $s\in[t,T]$ 
     	 \begin{align}\label{eqn-conv-2-2}
     		&\|\nabla\Z(s)\|_{\H}^2
     		+\int_t^{s}\|\mathcal{A}\Z(\tau)\|_{\H}^2\d\tau+
     		\int_t^{s}\||\Z(\tau)|^{\frac{r-1}{2}}\nabla\Z(\tau)\|_{\H}^{2}\d\tau
     		\nonumber\\&\leq C(\mu,\beta)
     		\left(\|\nabla\z\|_{\H}^2+\int_t^s\|\f(\tau,\a(\tau))\|_{\H}^2
     		\d \tau\right)e^{s-t},
     	\end{align}
     and
     	 \begin{align}\label{eqn-conv-2}
     		&\sup\limits_{s\in[t,T]}\|\nabla\Z(s)\|_{\H}^2+ \int_t^{s}\|\mathcal{A}\Z(\tau)\|_{\H}^2\d\tau
     		+\int_t^{s}\||\Z(\tau)|^{\frac{r-1}{2}}\nabla\Z(\tau)\|_{\H}^{2} \d\tau
     		\nonumber\\&\leq C(\mu,\beta,T) \left(\|\nabla\z\|_{\H}^2+\int_t^T\|\f(\tau,\a(\tau))\|_{\H}^2\d\tau \right).
     	\end{align}
     \end{proposition}
    
    \begin{proof}
    	The existence and uniqueness results of weak and strong solutions as well as energy estimates has been established  in \cite[Theorem 3.6 and Theorem 4.2]{SMTM}. 
    \end{proof}

     \begin{proposition}\label{cts-dep-soln}
     Let $t\in[0,T]$ and Assume that $\f:[0,T]\times\U\to\V$ be bounded and continuous. Let $\z_1$ and $\z_2$ are in $\H$, and $\a(\cdot)\in\mathscr{U}$. Then we have the following results: 
     
     (i) There exists a constant $C$ independent of $t,\z_1,\z_2$ and $\a(\cdot)$ such that for all $s\in[t,T]$, we have
     \begin{align}\label{ctsdep0}
     	&\|\Z_1(s)-\Z_2(s)\|_{\H}^2+\int_t^s \|\nabla(\Z_1-\Z_2)(\tau)\|_{\H}^2\d\tau+\int_t^s \|\Z_1(\tau)-\Z_2(\tau)\|_{\wi\L^{r+1}}^{r+1}\d\tau\nonumber\\&\leq
       \|\z_1-\z_2\|_{\H}^2 e^{C(s-t)},
     \end{align} 
     where $\Z_1(\cdot)=\Z_1(\cdot;t,\z_1,\a(\cdot))$ and $\Z_2(\cdot)=\Z_2(\cdot;t,\z_2,\a(\cdot))$ are two weak  solutions of \eqref{stap} with $\Z_1(t)=\z_1\in\H$ and $\Z_2(t)=\z_2\in\H$. 
     
     (ii) For every $\z\in\V$, there exists a constant $C$ independent of $\a(\cdot)\in\mathscr{U}$ such that for all $s\in[t,T]$, 
     \begin{align}\label{ctsdep0.1}
     	\|\Z(s)-\z\|_{\H}^2\leq C\left(\mu,\beta,T,\|\z\|_{\V},\sup\limits_{s\in[t,T],\a\in\U}\|\f(s,\a)\|_{\H}^2\right)(s-t), 
     \end{align} 
   where $\Z(\cdot)=\Z(\cdot;t,\z,\a(\cdot))$.
     
     (iii) For every $\z\in\V,$ there exists a modulus $\omega$, independent of the controls $\a(\cdot)\in\mathscr{U}$, such that
     \begin{align}\label{ctsdep0.2}
     	\|\Z(s)-\z\|_{\V}^2\leq\omega_{\z}(s-t), \ \text{ for all } \ s\in[t,T],
     \end{align}
     where $\Z(\cdot)=\Z(\cdot;t,\z,\a(\cdot))$.
     \end{proposition}
     
     \begin{proof}
     	\vskip 0.25cm
     \noindent
     \emph{\textbf{Proof of (i):}}  Let us set $\Y(\cdot):=\Z_1(\cdot)-\Z_2(\cdot)$. Then $\Y(\cdot)$ satisfies the following system:
     	\begin{equation}\label{ctsdep}
     	\left\{
     	\begin{aligned}
    	\frac{\d\Y(s)}{\d t}&=-\mu\mathcal{A}\Y(s)- \big(\mathfrak{B}(\Z_1(s))-\mathfrak{B}(\Z_2(s))\big)\\&
    	\quad-\beta\big(\mathfrak{C}(\Z_1(s))-\mathfrak{C}(\Z_2(s))\big), 
    	 \  \text{ in } \ (t,T]\times\H, \\
     	\Y(t)&=\z_1-\z_2\in\H.
     	\end{aligned}
     	\right.
     	\end{equation}
      By using \eqref{3.4} and \eqref{monoC2}, we find
     	\begin{align}
     &	|\langle\mathfrak{B}(\Z_1)-\mathfrak{B}(\Z_2),\Z_1-\Z_2\rangle|\nonumber\\&\quad\leq
     	\frac{\mu }{2}\|\nabla(\Z_1-\Z_2)\|_{\H}^2 +\frac{\beta}{4}\||\Z_2|^{\frac{r-1}{2}}(\Z_1-\Z_2)\|_{\H}^2 +\varrho\|\Z_1-\Z_2\|_{\H}^2,\label{ctsdep1}\\
     &	\langle\mathfrak{C}(\Z_1)-\mathfrak{C}(\Z_2),\Z_1-\Z_2\rangle\nonumber\\&\quad\geq \frac{1}{2}\||\Z_2|^{\frac{r-1}{2}}(\Z_1-\Z_2)\|_{\H}^2+
     	\frac{1}{4}\||\Z_1|^{\frac{r-1}{2}}(\Z_1-\Z_2)\|_{\H}^2,\label{ctsdep11}
     \end{align}
where $\varrho$ is defined in \eqref{eqn-varrho}. In \eqref{ctsdep}, by taking the inner product with $\Y(\cdot)$ and utilizing the estimates \eqref{ctsdep1} and \eqref{ctsdep11}, and using \eqref{monoC2}, we conclude
    \begin{align}\label{ctsdep2}
    	\frac12\frac{\d}{\d s}\|\Y(s)\|_{\H}^2+\frac{\mu }{2} \|\nabla\Y(s)\|_{\H}^2+\alpha\|\Y(s)\|_{\H}^2+
    	\frac{\beta}{2^r}\|\Y(s)\|_{\wi\L^{r+1}}^{r+1} \leq\varrho\|\Y(s)\|_{\H}^2,
    \end{align}
       for a.e. $s\in[t,T]$. Then, on employing the Gronwall's inequality, we obtain 
      \begin{align}\label{ctsdep3}
      	\|\Y(s)\|_{\H}^2\leq\|\z_1-\z_2\|_{\H}^2 e^{2\varrho(s-t)},
      \end{align}
    for all $s\in[t,T]$. Substituting \eqref{ctsdep3}  into \eqref{ctsdep2}, we arrive at \eqref{ctsdep0}.
     	
     	\vskip 0.25cm
     \noindent
     \emph{\textbf{Proof of (ii):}}  Let us take $\Y(\cdot):=\Z(\cdot)-\z$, so that it satisfies the following system:
     \begin{equation}\label{strgZ}
     	\left\{
     	\begin{aligned}
     		\frac{\d\Y(s)}{\d s}&=-\mu\mathcal{A}\Z(s)- \mathfrak{B}(\Z(s))
     		-\beta\mathfrak{C}(\Z(s))+\f(s,\a(s))
     		\  \text{ in } \ (t,T]\times\H, \\
     		\Y(t)&=\boldsymbol{0}.
     	\end{aligned}
     	\right.
     \end{equation}
     For each $\z\in\V$ and $\f(\cdot,\a(\cdot))\in\mathrm{L}^2(0,T;\H)$, $\Z(\cdot)=\Z(\cdot;t,\z,\a(\cdot))$ be the strong solution of the system \eqref{stap} with initial data $\Z(t)=\z$. 
          In the first equation of the system \eqref{strgZ}, on taking the inner product with $\Y(\cdot)$ and making the use of Cauchy Schwarz and Young's inequalities, we find
         	\begin{align}\label{strgZ1}
     	    \frac12\|\Y(s)\|_{\H}^2&\leq-\mu\int_t^{s}\|\nabla\Z(\tau)\|_{\H}^2\d\tau+ \frac12\int_t^{s}\|\f(\tau)\|_{\H}^2\d\tau+\frac12\int_t^{s}
     	    \|\Y(\tau)\|_{\H}^2\d\tau
     	    \nonumber\\&\quad+
     	    \int_s^t (\mathfrak{B}(\Z(\tau)),\z)\d\tau-
     	    \beta\int_s^t (\mathfrak{C}(\Z(\tau)),\Z(\tau))\d\tau+
     	    \beta\int_s^t (\mathfrak{C}(\Z(\tau)),\z)\d\tau,
     	    \end{align}
             for all $s\in[t,T]$. 
     By the application of Cauchy-Schwarz, interpolation and Young's inequalities, we find the following estimates:
     	\begin{align}\label{ctsdep6}
     	|(\mathfrak{B}(\Z),\z)|
     	\leq C\|\Z\|_{\V}^2\|\nabla\z\|_{\H},
     	\end{align}
     	\begin{align}\label{ctsdep7}
     			\big|\big(\mathfrak{C}(\Z),\z\big)\big|
     		\leq\frac12\|\Z\|_{\wi\L^{r+1}}^{r+1}+C\|\Z\|_{\wi\L^{3(r+1)}}^{r+1}+
     		C\|\z\|_{\H}^{r+1}.
     	\end{align}
     	Plugging \eqref{ctsdep6}-\eqref{ctsdep7} into \eqref{strgZ1}, and using \eqref{eqn-conv-1} and \eqref{eqn-conv-2} yield
     	\begin{align*}
     	 \frac12\|\Y(s)\|_{\H}^2&\leq C\int_t^s \|\Z(\tau)\|_{\V}^2\d\tau+
     	 \frac12\int_t^{s}\|\f(\tau)\|_{\H}^2\d\tau+\frac12\int_t^{s}
     	 \|\Y(\tau)\|_{\H}^2\d\tau
     	 \nonumber\\&\quad+
     	 \frac{\beta}{2}\int_t^s\|\Z(\tau)\|_{\wi\L^{r+1}}^{r+1}\d\tau+
     	 \frac{\beta}{2}\int_t^s\|\Z(\tau)\|_{\wi\L^{3(r+1)}}^{r+1}\d\tau+
     	 C(s-t),
     \end{align*}
   for all $s\in[t,T]$. 
   By making the use of Gronwall's inequality and taking supremum over control strategies $\a(\cdot)$, we finally conclude 
   \begin{align*}
   	\|\Y(s)\|_{\H}^2\leq C(s-t)e^{s-t},
   \end{align*}
  for all $s\in[t,T]$, where the constant $C=C\left(\mu,\beta,T,\|\z\|_{\V},\sup\limits_{s\in[t,T],\a\in\U}\|\f(s,\a)\|_{\H}^2\right)$. This completes the proof of \eqref{ctsdep0.1}.
\vskip 0.25cm
\noindent
\emph{\textbf{Proof of (iii):}} To prove \eqref{ctsdep0.2}, we need to show that
       \begin{align}\label{ctsdep8}
       	\sup\limits_{s\in[t,t+\lambda], \ \a(\cdot)\in\mathscr{U}}
       \|\Z(s;t,\z,\a(\cdot))-\z\|_{\V}^2\to0 \ \text{ as } \ \lambda\to0.
       \end{align}
       Suppose that \eqref{ctsdep8} does not hold true. Then, there exist sequences $s_n$ and $\a_n(\cdot)\in\mathscr{U}$ such that
       \begin{align}\label{contra1}
       	s_n\to t \ \text{ and }  \ \|\Z_n(s_n)-\z\|_{\V}^2\geq\lambda
       \end{align}
       for all $n\geq1,$ where $\Z_n(s_n)=\Z(s_n;t,\z,\a_n(\cdot))$. However, it follows from \eqref{eqn-conv-1}, \eqref{eqn-conv-2} and \eqref{ctsdep0.1}, we have the following convergences (along a subsequence):
       \begin{align*}
       	\Z_n(s_n)\to\z \ \text{ in } \ \H  \  \text{ and } \
       	\Z_n(s_n)\rightharpoonup\z \ \text{ in } \ \V.
       \end{align*}
       The weak sequential convergence in $\V-$norm implies 
       \begin{align*}
       	\|\z\|_{\V}^2\leq\liminf_{n\to\infty}\|\Z_n(s_n)\|_{\V}^2,
       \end{align*}
       while the uniform energy estimate in \eqref{eqn-conv-2} provides 
       \begin{align*}
       		\|\z\|_{\V}^2\geq\limsup_{n\to\infty}\|\Z_n(s_n)\|_{\V}^2.
       \end{align*}
       These inequalities together yield that $\|\z\|_{\V}^2=\lim\limits_{n\to\infty}\|\Z_n(s_n)\|_{\V}^2$ and hence the Radon-Riesz property (\cite[Proposition 3.32, pp. 78]{HB}) yields $\Z_n(s_n)\to\z$ in $\V$, which contradicts \eqref{contra1}, and the proof of \eqref{ctsdep0.2} is completed. 
     \end{proof}
     
           \section{Dynamic Programming Method and HJBE: Viscosity Solution}\setcounter{equation}{0} \label{detstchjb}
    Dynamic programming provides a powerful framework for addressing optimal control problems. This approach involves analyzing a family of control problems that vary with respect to their initial times and initial states. By establishing a relationship between these problems, one arrives at the celebrated HJBE, which characterizes the value function associated with the optimal control problem.
    
    Consider the cost functional associated with the system \eqref{stap}, defined for a given initial time $t$, initial state $\z \in \H$, and control function $\a(\cdot) \in \mathscr{U}$, by:
    \begin{align} \label{costF1}
    	\mathfrak{J}(t, \z; \a(\cdot)) =
    	\int_t^T \mathpzc{L}(s, \Z(s), \a(s)) , \mathrm{d}s + \mathpzc{g}(\Z(T; t, \z, \a(\cdot))),
    \end{align}
    where $\mathpzc{L} : [t, T] \times \H \times \U \to \R$ is the running cost function, and $\mathpzc{g} : \H \to \R$ is the terminal cost function. Both $\mathpzc{L}$ and $\mathpzc{g}$ are assumed to be measurable.
     The objective is to minimize the cost functional \eqref{costF1} over the admissible control set $\mathscr{U}$. According to Proposition \ref{weLLp}, for any $\f \in \mathrm{L}^2(0, T; \H)$ and $\z \in \V$, the state equation \eqref{stap} admits a unique strong solution $\Z(\cdot)$ satisfying the regularity condition \eqref{eqn-conv-2}. Hence, the cost functional $\mathfrak{J}$ is well-defined.
    
    The minimization of $\mathfrak{J}$ over $\mathscr{U}$ leads, in a formal sense, to the derivation of HJBE, which provides a necessary condition for optimality through  the following partial differential equation governing the value function: 
        \begin{equation}\label{detHJB1}
    	\left\{
    	\begin{aligned}
    		\mathpzc{u}_t-(\mu\mathcal{A}\z+\mathfrak{B}(\z)+
    		\beta\mathfrak{C}(\z),\mathcal{D} \mathpzc{u})+\mathcal{F}(t,\z,\mathcal{D} \mathpzc{u})&=0, \ \text{ for } (t,\z)\in(0,T)\times\H, \\
    		\mathpzc{u}(T,\z)&=\mathpzc{g}(\z), \ \text{ for } \ \z\in\H, 
    	\end{aligned}
    	\right.
    \end{equation}
    where $\mathcal{F}$ is the Hamiltonian function given by 
    \begin{align}\label{hamfun}
    	\mathcal{F}(t,\z,\p)=\inf\limits_{\a\in\U} \left\{(\f(t,\a),\p)+\mathpzc{L}(t,\z,\a)\right\}.
    \end{align} 
 Here, $\mathpzc{u} : (0, T) \times \mathcal{H} \to \mathbb{R}$ is the unknown function, and $\mathcal{D} \mathpzc{u}$ denotes its Fr\'echet derivative in the Hilbert space $\mathcal{H}$. Note that equation \eqref{detHJB1} is a first-order, fully nonlinear partial differential equation (PDE) with respect to the gradient variable $\mathcal{D} \mathpzc{u}$. It is considered an \emph{infinite-dimensional PDE} in the sense that the state process $\mathcal{Z}(\cdot)$ in \eqref{stap} evolves in the infinite-dimensional space $\mathcal{H}$.
  \begin{definition}
  	In dynamic programming, the value function $\mathpzc{V}(\cdot)$ is defined as the optimum value of the cost functional associated with the optimal control problem \eqref{stap} and \eqref{costF1} considered as the function of initial time $t$ and initial state $\z$. Specifically, it is given by 
  	\begin{align}\label{valueF}
  		\mathpzc{V}(t,\z):=\inf\limits_{\a(\cdot)\in\mathscr{U}} \mathfrak{J}(t,\z,\a).
  	\end{align}
  \end{definition}
  
     We now analyse the continuity properties of the value function $\mathpzc{V}$. For simplicity, we assume that $\mathpzc{L}$ is time independent. This assumption helps to reduce technical difficulties and allows us to focus on the key aspects of the proof. The following assumptions on $\mathpzc{L}$, $g$, and $\f$ have been taken:
     \begin{hypothesis}\label{valueH} 
     	Assume that the running cost $\mathpzc{L} : \V \times \U \to \mathbb{R}$ and the terminal cost $\mathpzc{g} : \V \to \mathbb{R}$ are continuous. Furthermore, suppose there exists a constant $k \geq 0$, and for every $r > 0$, a modulus of continuity $\sigma_r$, such that $\mathpzc{L}$ and $\mathpzc{g}$ satisfy the following assumptions:
     	\begin{align}
     		|\mathpzc{L}(\z,\a)|, |\mathpzc{g}(\z)|&\leq C(1+\|\z\|_{\V}^k), \  \ \text{ for all } \ \z\in\V, \ \a\in\U,\label{vh1}\\
     		|\mathpzc{L}(\z,\a)-\mathpzc{L}(\x,\a)|&\leq\sigma_r(\|\z-\x\|_{\V})\  \ \text{ if } \ \|\z\|_{\V},\|\x\|_{\V}\leq r, \ \a\in\U,\label{vh2}\\
     		|\mathpzc{g}(\z)-\mathpzc{g}(\x)|&\leq\sigma_r(\|\z-\x\|_{\H}) \  \ \text{ if } \ \|\z\|_{\V},\|\x\|_{\V}\leq r.\label{vh3}
     	\end{align}
     \end{hypothesis}
     The following proposition addresses the continuity properties of $\mathpzc{V}$ and the dynamic programming principle. Since the proof follows standard argument, we refer the reader to \cite{FGSSA1} for a complete and detailed explanation.
     \begin{proposition}\label{valueP}
     	Assume that Hypothesis \ref{valueH} holds. Then the following properties of the cost functional and the value function follow:
     	
     	\textbf{(i) Local continuity of the cost functional:}
     	The cost functional $\mathfrak{J}(\cdot)$ depends continuously on the initial condition, locally in $\V$. Specifically, for every $r > 0$, there exists a modulus of continuity $\omega_r$ such that for all $t \in [0, T]$, all admissible controls $\a(\cdot) \in \mathscr{U}$, and all initial states $\z, \x \in \V$ with $\|\z\|_{\V}, \|\x\|_{\V} \leq r$, we have:
     		$$
     	\left| \mathfrak{J}(t, \z; \a(\cdot)) - \mathfrak{J}(t, \x; \a(\cdot)) \right| \leq \omega_r\left( \|\z - \x\|_{\mathcal{H}} \right).
     	$$
     	
     \textbf{(ii) Dynamic Programming Principle (DPP):}
     	The value function $\mathpzc{V}(\cdot)$ satisfies the Bellman principle of optimality. That is, for all $0 \leq t \leq \eta \leq T$ and $\z \in \V$, we have:
     		\begin{align}\label{vdpp}
     				\mathpzc{V}(t, \z) = \inf_{\a(\cdot) \in \mathscr{U}} \left\{ \int_t^\eta \mathpzc{L}(\Z(s; t, \z, \a(\cdot)), \a(s)) \, \mathrm{d}s + \mathpzc{V}(\eta, \Z(\eta; t, \z, \a(\cdot))) \right\}.
     		\end{align}
     
     \textbf{(iii) Regularity and growth of the value function:}
     	The value function $\mathpzc{V}(\cdot)$ is locally Lipschitz continuous in both time and initial data. More precisely, for each $r > 0$, there exists a modulus $\omega_r$ such that for all $t_1, t_2 \in [0, T]$ and $\z, \x \in \V$ with $\|\z\|_{\V}, \|\x\|_{\V} \leq r$, it holds that:
     	\begin{align}\label{bddV}
     			|\mathpzc{V}(t_1, \z) - \mathpzc{V}(t_2, \x)| \leq \omega_r\left( |t_1 - t_2| + \|\z - \x\|_{\mathcal{H}} \right).
     	\end{align}
     	Moreover, $\mathpzc{V}(\cdot)$ satisfies a polynomial growth bound: there exist constants $C \geq 0$ and $k \geq 0$ such that
     	\begin{align}\label{bddV1}
     			|\mathpzc{V}(t, \z)| \leq C\left(1 + \|\z\|_{\V}^k\right),
     	\end{align}
     		for all $t \in [0, T]$ and $\z \in \V$.
     \end{proposition}

     \subsection{Viscosity solution}\label{viscPR}
     	We aim to establish that value function \eqref{valueF} is the solution to the HJBE \eqref{detHJB1}. 
     	It is important to note that proving the existence of a viscosity solution relies on the dynamic programming approach of the optimal control problem. Therefore, it is restricted to the HJBE \eqref{detHJB1}. Further, for the uniqueness of the viscosity solution, we need to prove the hardcore result of this work, the comparison principle (see Section \ref{compPR}). The comparison principle we prove for the following HJBE with more general Hamiltonian $\mathcal{F}$ satisfying certain assumptions (see Hypothesis \ref{hypF14} below):
     \begin{equation}\label{HJBE}
     	\left\{
     	\begin{aligned}
     		\mathpzc{u}_t-(\mu\mathcal{A}\z+\mathfrak{B}(\z)+
     		\beta\mathfrak{C}(\z),\mathcal{D} \mathpzc{u})+\mathcal{F}(t,\z,\mathcal{D} \mathpzc{u})&=0, \ \text{ for } (t,\z)\in(0,T)\times\H,\\
     		\mathpzc{u}(T,\z)&=\mathpzc{g}(\z), \ \text{ for } \ \z\in\H.
     	\end{aligned}
     	\right.
     \end{equation}
     
     
       The main ingredient in the viscosity solution theory is the choice of the test function. The following is the definition of a test function, which has been taken, with slight modification, from \cite{FGSSA1}:
     \begin{definition}\label{testD}
     	A function $\uppsi:(0,T)\times\H\to\R$ is termed a \emph{test function} if it can be decomposed as $$\uppsi(t,\z)=\upvarphi(t,\z)+h(t,\|\z\|_{\V}),$$ where the components satisfy the following:
     	
     	\textbf{Structure of $\upvarphi$.}   
     	 $\upvarphi\in\mathrm{C}^{1,1}((0,T)\times\H)$ and
     	 the partial derivative $\upvarphi_t$ and the derivative $\mathcal{D}\upvarphi$ are uniformly continuous on every closed subinterval $[\lambda,T-\lambda]\times\H$ for every $\lambda>0$.
     
     	\textbf{Structure of $h$:}  
     $h\in\mathrm{C}^{1,1}((0,T)\times\R)$ with the additional property that for each fixed $t\in(0,T)$, the function  $h(t,\cdot)$ is even, that is $h(t,r)=h(t,-r)$ and non-decreasing on $[0,+\infty)$.
     \end{definition}
     
     \begin{remark}
     1.) Although the norm $\|\z\|_{\V}$ is not differentiable at the origin $\boldsymbol{0}$, the function $h(t, \|\z\|_{\V})$ belongs to the class $\mathrm{C}^{1,1}((0, T) \times \V)$. Nevertheless, the lack of Fréchet differentiability in the Hilbert space $\H$ necessitates a careful interpretation of expressions involving $\mathcal{D} h$ and inner products such as $(\mu\mathcal{A}\z + \mathfrak{B}(\z) + \beta\mathfrak{C}(\z), \mathcal{D} h(t, \z))$. To address this, we define the derivative operator $\mathcal{D} h$ formally via the expression:   
     $$
     \mathcal{D} h(t, \z) := \frac{h_r(t, \|\z\|_{\V})}{\|\z\|_{\V}} \mathcal{A}\z,
     $$   
     where $h_r$ denotes the partial derivative of $h$ with respect to its second argument. We define $\mathcal{D}\psi := \mathcal{D}\varphi + \mathcal{D} h$, and for all $(t, \z) \in (0, T) \times \V_2$, the quantity $\mathcal{D}\psi(t, \z)$ is well-defined, as is the inner product $(\mu\mathcal{A}\z + \mathfrak{B}(\z) + \beta\mathfrak{C}(\z), \mathcal{D} \psi(t, \z))$.
     
     2.) From this point onward, we consider the specific form $h(t, \z) = \mathfrak{h}(t)\|\z\|_{\V}^2$, where $\mathfrak{h}(\cdot) \in \mathrm{C}^1(0, T)$ and satisfies $\mathfrak{h}(t) > 0$ for all $t \in (0, T)$. As discussed in \cite{FGSSA}, a more general form such as $h(t, \z) = \mathfrak{h}(t)(1 + \|\z\|_{\V}^2)^m$, for some $m \in \mathbb{N}$, may also be used. However, both forms are mathematically equivalent in the context of proving the Theorem \ref{comparison} (comparison principle) and Theroem \ref{extunqvisc} (the existence of viscosity solutions).
     
     3.) For the choice $h(t, \z) = \mathfrak{h}(t)\|\z\|_{\V}^2$, we compute the following Fréchet derivatives:
     \begin{align}
     	\mathcal{D}_{\z} \|\z\|_{\V}^2 &= 2(\mathcal{A} + \I)\z, \label{apyv1} \\
     	\mathcal{D}_{\z}(1 + \|\z\|_{\V}^2)^m &= 2m(\mathcal{A} + \I)\z (1 +\|\z\|_{\V}^2)^{m-1}. \label{apyv2}
     \end{align}
     However, while establishing the Theorem \ref{comparison} and proving the Theroem \ref{extunqvisc}, the multiplicative factor $(1 + \|\z\|_{\V}^2)^{m - 1}$ in \eqref{apyv2} does not influence the analysis, as the norm in $\V$ is bounded (see \cite[Theorems 5.11 and 6.1]{smtm1}). Therefore, in this work, we adopt test functions of the form $\upvarphi(t, \z) + \mathfrak{h}(t)\|\z\|_{\V}^2$, which satisfy all the requirements of Definition \ref{testD}.
     \end{remark}
     
     We now provide the definition of  a viscosity solution.  Here Hamiltonian $\mathcal{F}$ is not necessarily of the form \eqref{hamfun}. We assume that $\mathcal{F}:[0, T]\times\V\times\H\to\R$ is any function. 
     \begin{definition}\label{viscsoLndef}
     	Let $\mathpzc{u}:(0,T]\times\V\to\R$ be a weakly sequentially upper-semicontinuous (respectively, a lower-semicontinuous) function. We say that $\mathpzc{u}$ is a \emph{viscosity subsolution} (respectively, \emph{supersolution}) of \eqref{HJBE} if 
     	\begin{itemize}
     		\item $\mathpzc{u}(T,\wi\z)\leq h(\wi\z)$ (respectively, $\mathpzc{u}(T,\wi\z)\geq h(\wi\z)$) for all $\wi\z\in\V$;
     		\item whenever $\mathpzc{u}-\uppsi$ has global maximum (respectively, $\mathpzc{u}+\uppsi$ has a global minimum) at a point $(t,\z)\in(0,T)\times\V$ for every test function $\uppsi$, then $\z\in\V_2$ and 
     		\begin{align*}
     			\uppsi_t(t,\z)-
     			(\mu\mathcal{A}\z+ \mathfrak{B}(\z)+\beta\mathfrak{C}(\z),\mathcal{D}\uppsi(t,\z))
     			+\mathcal{F}(t,\z,\mathcal{D}\uppsi(t,\z))\geq0
     		\end{align*}
     		(respectively, 
     		\begin{align*}
     			-\uppsi_t(t,\z)+
     			(\mu\mathcal{A}\z+\mathfrak{B}(\z)+\beta\mathfrak{C}(\z),\mathcal{D}\uppsi(t,\z))
     			+\mathcal{F}(t,\z,-\mathcal{D}\uppsi(t,\z))\leq0.)
     		\end{align*}
     	\end{itemize}
     	A function $\mathpzc{u}:(0,T]\times\V\to\R$ which is both viscosity subsolution and viscosity supersolution is called \emph{viscosity solution} of \eqref{HJBE}. We also say that $\mathpzc{u}$ satisfies \eqref{HJBE} in the viscosity sense.
     \end{definition}

	 \section{The comparison principle: Uniqueness of Viscosity solution}\label{compPR}\setcounter{equation}{0}
	 
	  This subsection proves the comparison principle for the equation \eqref{HJBE}. It guarantees that, under appropriate conditions, any viscosity subsolution remains bounded above by any viscosity supersolution. As an important consequence, the comparison principle ensures the uniqueness of viscosity solutions.

	\begin{hypothesis}\label{hypF14}
		There exists a modulus of continuity $\omega$ and moduli $\omega_r$ such that for every $r>0,$ the hamiltonian $\mathcal{F}:[0,T]\times\V\times\H\to\R$ and the terminal cost $\mathpzc{g}:\H\to\R$ satisfies the following locally Lipschitz properties:
		\begin{align}
			|\mathcal{F}(t,\z,\p)-\mathcal{F}(t,\x,\p)|&\leq\omega_r(\|\z-\x\|_{\V})+\omega(\|\z-\x\|_{\V}\|\p\|_{\H}), \ \text{ if } \ \|\z\|_{\V}\leq r,\|\x\|_{\V}\leq r,\label{F1}\\
			|\mathcal{F}(t,\z,\p)-\mathcal{F}(t,\z,\q)|&\leq\omega((1+\|\z\|_{\V})\|\p-\q\|_{\H}),\label{F2}\\
			|\mathcal{F}(t,\z,\p)-\mathcal{F}(s,\z,\p)|&\leq\omega_r(|t-s|),\ \text{ if } \ \|\z\|_{\V}\leq r,\|\z\|_{\V}\leq r, \|\p\|_{\V}\leq r,\label{F3}\\
			|\mathpzc{g}(\z)-\mathpzc{g}(\x)|&\leq\omega_r(\|\z-\x\|_{\H}), \ \text{ if } \ \|\z\|_{\V}\leq r,\|\x\|_{\V}\leq r.\label{F4}
		\end{align}
	\end{hypothesis}
	
	Let us now state the comparison principle for viscosity solutions.
	\begin{theorem}\label{comparison}
		Suppose Hypothesis \ref{hypF14} is satisfied. Let $\mathpzc{u}:(0,T]\times\V\to\R$ be a viscosity subsolution and $\mathpzc{v}:(0,T]\times\V\to\R$ be a viscosity supersolution of \eqref{HJBE}. Assume further that there exists a constant $C>0$ and $k>0$ such that 
		\begin{align}\label{bdd}
			\mathpzc{u}(t,\z), -\mathpzc{v}(t,\z), |\mathpzc{g}(\z)|\leq C(1+\|\z\|_{\V}^k).
		\end{align}
     Then, $$\mathpzc{u}\leq \mathpzc{v}\ \text{ on }\ (0,T]\times\V.$$
	\end{theorem}
	
	 \begin{proof}
	Since $\mathpzc{u}$ is a viscosity subsolution and $\mathpzc{v}$ is a viscosity supersolution of \eqref{HJBE}, therefore by definition, we have 
			\begin{equation}\label{posneg}
			\left\{
			\begin{aligned}
				\lim\limits_{t\uparrow T} \ (\mathpzc{u}(t,\z)-g(\z))^+&=0 \\
				\text{ and } 
				\lim\limits_{t\uparrow T} \ (\mathpzc{v}(t,\z)-g(\z))^-&=0,
			\end{aligned}
			\right.
			\end{equation}
			uniformly on bounded subsets of $\V$.
			 For $\upgamma>0$, we define
		\begin{align*}
			\mathpzc{u}_{\upgamma}(t,\z):=\mathpzc{u}(t,\z)-\frac{\upgamma}{t}\  \text{ and } \ \mathpzc{v}_{\upgamma}(t,\z):=\mathpzc{v}(t,\z)+\frac{\upgamma}{t}. 
		\end{align*}
        Then, $\mathpzc{u}_{\upgamma}$ is a viscosity subsolution of 
		\begin{equation}\label{newsub}
		\left\{
		\begin{aligned}
			(\mathpzc{u}_\upgamma)_t-
		(\mu\A\z+\alpha\z+\mathfrak{B}(\z)+\beta\mathfrak{C}(\z),\mathcal{D} \mathpzc{u}_{\upgamma})+\mathcal{F}(t,\z,\mathcal{D} \mathpzc{u}_{\upgamma}) &\geq\frac{\upgamma}{T^2},\\
		\mathpzc{u}_\upgamma(T,\z)&=g(\z)-\frac{\upgamma}{T},
		\end{aligned}
		\right.
		\end{equation}
		and $\mathpzc{v}_{\upgamma}$ is a viscosity supersolution of 
			\begin{equation}\label{newsup}
			\left\{
			\begin{aligned}
				(\mathpzc{v}_\upgamma)_t-
				(\mu\A\z+\alpha\z+\mathfrak{B}(\z)+\beta\mathfrak{C}(\z),\mathcal{D} \mathpzc{v}_{\upgamma})+\mathcal{F}(t,\z,\mathcal{D} \mathpzc{v}_{\upgamma})& \leq-\frac{\upgamma}{T^2},\\
				\mathpzc{v}_\upgamma(T,\z)&=g(\z)+\frac{\upgamma}{T}.
			\end{aligned}
			\right.
		\end{equation}
		Let us define the function $\Phi$ on $(0,T]\times\H$ for $\lambda,\delta,\eta>0$, by
		\begin{equation*}
			\Phi(t,s,\z,\x)=\left\{
			\begin{aligned}
			&\mathpzc{u}_{\upgamma}(t,\z)-\mathpzc{v}_{\upgamma}(s,\x)-
				\frac{\|\z-\x\|_{\H}^2}{2\lambda}-\delta e^{\mathpzc{k}_\upgamma(T-t)} \|\z\|_{\V}^2\\&-\delta e^{\mathpzc{k}_\upgamma(T-s)} \|\x\|_{\V}^2-\frac{(t-s)^2}{2\eta}, \  &&\text{ if } \ \z,\x\in\V\\
				&-\infty, \  &&\text{ if } \ \z,\x\notin\V,
			\end{aligned}
			\right.
		\end{equation*}
	where the constant $\mathpzc{k}_\upgamma$ to be chosen later. Note that $\Phi\to-\infty$ as $\max\{\|\z\|_{\V},\|\x\|_{\V}\}\to+\infty$.
	We now proceed the proof into following number of steps:
		\vskip 0.2cm
		\noindent
		\textbf{Step-I:} \emph{$\Phi$ has global maxima over $(0,T]\times(0,T]\times\H\times\H$.}  
		 We first claim that the function $\mathpzc{u}_{\upgamma}(t,\z)-\delta e^{\mathpzc{k}_\upgamma(T-t)}\|\z\|_{\V}^2 $ is weakly sequentially upper semicontinuous on $ (0,T)\times\H.$
	
		Suppose it is not. Then, there exists sequences $(t_n)_{n\geq1}$ in $(0,T)$ and $(\z_n)_{n\geq1}$ in $\H$ with $t_n\to t\in(0,T)$ and $\z_n\rightharpoonup\z\in\H$ such that
		\begin{align}\label{contra}
			\limsup\limits_{n\to\infty}\left(\mathpzc{u}_{\upgamma}(t_n,\z_n)-\delta e^{\mathpzc{k}_\upgamma(T-t_n)}\|\z_n\|_{\V}^2\right)>
			\mathpzc{u}_{\upgamma}(t,\z)-\delta e^{\mathpzc{k}_\upgamma(T-t)} \|\z\|_{\V}^2.
		\end{align}
		Now, if $\liminf\limits_{n\to\infty}\|\z_n\|_{\V}=+\infty$, then \eqref{contra} is impossible due to \eqref{bdd}. Therefore, $\liminf\limits_{n\to\infty}\|\z_n\|_{\V}<+\infty$ and there exists a subsequence (still denoted by $(t_n,\z_n)$) such that $\limsup\limits_{n\to\infty}\|\z_n\|_{\V}<+\infty$. By using the Banach-Alaoglu theorem, we then have $\z_n\rightharpoonup\z$ in $\V$ (along a subsequence), and therefore, we have $$\liminf_{n\to\infty}\delta e^{\mathpzc{k}_\upgamma(T-t_n)} \|\z_n\|_{\V}^2\geq \delta e^{\mathpzc{k}_\upgamma(T-t)}\|\z\|_{\V}^2. $$
		Moreover, from \eqref{contra}, we further have
		\begin{align*}
	\limsup\limits_{n\to\infty}\mathpzc{u}_{\upgamma}(t_n,\z_n)>\mathpzc{u}_{\upgamma}(t,\z),
		\end{align*}
		 which is a contradiction.
		  In a similar way, we can prove the weak sequential lower semicontinuity of $\mathpzc{v}_{\upgamma}(s,\x)-\delta e^{\mathpzc{k}_\upgamma(T-s)} \|\x\|_{\V}^2$ on $(0,T)\times\H$.  Finally, the definition of viscosity solution and the application of Weierstrass theorem yields that $\Phi$ attains a global maxima over $(0, T]\times(0, T]\times\H\times\H$ at some point $(\bar{t},\bar{s},\bar{\z},\bar{\x})\in(0, T]\times(0, T]\times\V\times\V$. One can assume the points of maxima to be strict maxima also. Furthermore, for a fixed $\delta>0$, the points $\bar{\z}$ and $\bar{\x}$ are bounded in $\V$  independently of $\lambda$, provided $k\leq3$.

		 In view of \cite[Theorem 3.50]{GFAS}, we have the following limits:
		\begin{align}\label{copm1}
			\lim\limits_{\eta\to0}\frac{(\bar{t}-\bar{s})^2}{2\eta}&=0 \ \text{ for fixed } \ \delta>0, \lambda>0,
		\end{align}
		and
		\begin{align}\label{copm2}
			\lim\limits_{\lambda\to0}\limsup\limits_{\eta\to0} \frac{\|\bar{\z}-\bar{\x}\|_{\H}^2}{2\lambda}&=0 \ \text{ for fixed } \ \delta>0.
		\end{align}
	  We assume that $\mathpzc{u}\not\leq\mathpzc{v}$ on $(0,T]\times\V$. Furthermore, the definition of viscosity solution yields $\bar{\z},\bar{\x}\in\V_2$, and from \eqref{copm1}-\eqref{copm2}, \eqref{F4} and \eqref{posneg},  it follows that for small $\upgamma$ and $\delta$, we have $\bar{t},\bar{s}<T$ if $\eta$ and $\lambda$ are sufficiently small (see \cite{smtm1} for detailed explanation).

	   \vskip 0.2cm
		\noindent
		\textbf{Step-II:} \emph{Applying definition of viscosity solution.} 
		Note that, the function 
		\begin{align*}
			(t,\z)\mapsto\Phi(t,\bar{s},\z,\bar{\x})&=\mathpzc{u}_{\upgamma}(t,\z)-\mathpzc{v}_{\upgamma}(\bar{s},\bar{\x})-
			\frac{\|\z-\bar{\x}\|_{\H}^2}{2\lambda}-\delta e^{\mathpzc{k}_\upgamma(T-t)} \|\z\|_{\V}^2\\&\quad-\delta e^{\mathpzc{k}_\upgamma(T-\bar{s})} \|\bar{\x}\|_{\V}^2-\frac{(t-\bar{s})^2}{2\eta}
		\end{align*} 
		has a global maximum at $(\bar{t},\bar{\z})$ in $(0,T)\times\H$. Then using the fact that $u_{\upgamma}$ is a viscosity subsolution, we have
		\begin{align}\label{supsoLdef}
        &-\delta \mathpzc{k}_\upgamma e^{\mathpzc{k}_\upgamma(T-\bar{t})}   \|\bar{\z}\|_{\V}^2+\frac{\bar{t}-\bar{s}}{\eta}
	    -
         \bigg(\mu\mathcal{A}\bar{\z},\frac{1}{\lambda}(\bar{\z}-\bar{\x})
         +2\delta e^{\mathpzc{k}_\upgamma(T-\bar{t})} (\mathcal{A}+\I)\bar{\z}\bigg)
          \nonumber\\&\quad-
         \bigg(\mathfrak{B}(\bar{\z})+\beta\mathfrak{C}(\bar{\z}),\frac{1}{\lambda}(\bar{\z}-\bar{\x})
         +2\delta e^{\mathpzc{k}_\upgamma(T-\bar{t})}(\mathcal{A}+\I)\bar{\z}\bigg)
        \nonumber\\&\quad+
         F\bigg(\bar{t},\bar{\z},\frac{1}{\lambda}(\bar{\z}-\bar{\x})
         +2\delta e^{\mathpzc{k}_\upgamma(T-\bar{t})}(\mathcal{A}+\I)\bar{\z}\bigg)
         \nonumber\\&\geq\frac{\upgamma}{T^2}.
		\end{align}
	 From assumption \eqref{F2} of Hypothesis \ref{hypF14}, we calculate
	\begin{align}\label{F2cal}
    &\bigg|F\big(\bar{t},\bar{\z},\frac{1}{\lambda}(\bar{\z}-\bar{\x})+2\delta e^{\mathpzc{k}_\upgamma(T-\bar{t})}(\mathcal{A}+\I)\bar{\z}\big)
   -F\big(\bar{t},\bar{\z},\frac{1}{\lambda}(\bar{\z}-\bar{\x})\big)\bigg| \nonumber\\&\leq
	\omega\left((1+\|\bar{\z}\|_{\V})\times2\delta e^{\mathpzc{k}_\upgamma(T-\bar{t})}\|(\mathcal{A}+\I)\bar{\z}\|_{\H}\right)
	\nonumber\\&\leq
	\frac{\upgamma}{2T^2}+2\delta e^{\mathpzc{k}_\upgamma(T-\bar{t})}
	\underbrace{C_\upgamma(1+\|\bar{\z}\|_{\V})\|(\mathcal{A}+\I)\bar{\z}\|_{\H}}_{\text{Cauchy-Schwarz}}
	\nonumber\\&\leq\frac{\upgamma}{2T^2}+
	\frac{2C_\upgamma^2 \delta}{\mu} e^{\mathpzc{k}_\upgamma(T-\bar{t})}(1+\|\bar{\z}\|_{\V}^2)+\mu\delta e^{\mathpzc{k}_\upgamma(T-\bar{t})} \|(\mathcal{A}+\I)\bar{\z}\|_{\H}^2.
		\end{align}
		Substituting \eqref{F2cal} into \eqref{supsoLdef} and rearranging the terms to get
			\begin{align}\label{supsoLdef1}
			&-\delta \mathpzc{k}_\upgamma e^{\mathpzc{k}_\upgamma(T-\bar{t})} \|\bar{\z}\|_{\V}^2+
			\frac{2C_\upgamma^2 \delta}{\mu} e^{\mathpzc{k}_\upgamma(T-\bar{t})}(1+\|\bar{\z}\|_{\V}^2)
			+\frac{\bar{t}-\bar{s}}{\eta}
			\nonumber\\&\quad-
			\bigg(\mu(\mathcal{A}+\I)\bar{\z}-\mu\bar{\z}+\mathfrak{B}(\bar{\z})+\mathfrak{C}(\bar{\z}),\frac{1}{\lambda}(\bar{\z}-\bar{\x})\bigg)
			\nonumber\\&\quad-
			\delta e^{\mathpzc{k}_\upgamma(T-\bar{t})}
			\underbrace{\big(\mu(\mathcal{A}+\I)\bar{\z}-2\mu\bar{\z}+2\mathfrak{B}(\bar{\z})+2\mathfrak{C}(\bar{\z}),(\mathcal{A}+\I)\bar{\z}\big)}_{I_1}
			\nonumber\\&\quad+
			F\big(\bar{t},\bar{\z},\frac{1}{\lambda}(\bar{\z}-\bar{\x})\big)
			\nonumber\\&\geq\frac{\upgamma}{2T^2}.
		\end{align}

	 Let us now estimate $I_1$ as follows:
	\vskip 2mm
	\noindent
	\textbf{Case-I of Table \ref{Table1}.}
	From \eqref{syymB3} and \eqref{torusequ}, we infer
	\begin{align}\label{rg3c}
		-&\mu\|(\mathcal{A}+\I)\bar{\z}\|_{\H}^2
		+2\mu\|\bar{\z}\|_{\V}^2-
		2(\mathfrak{B}(\bar{\z})+\beta\mathfrak{C}(\bar{\z}),
		(\mathcal{A}+\I)\bar{\z})\nonumber\\&
		\leq-\frac{\mu}{2}\|(\mathcal{A}+\I)\bar{\z}\|_{\H}^2-
		\beta\||\bar{\z}|^{\frac{r-1}{2}}\nabla\bar{\z}\|_{\H}^2+ 2(\mu+\varrho_1)\|\bar{\z}\|_{\V}^2-
		2\beta\|\bar{\z}\|_{\wi\L^{r+1}}^{r+1},
	\end{align}
	where $\varrho_1=\frac{r-3}{\mu(r-1)}\left[\frac{4}{\beta\mu (r-1)}\right]^{\frac{2}{r-3}}$.
	\vskip 2mm
	\noindent
\textbf{Case-II of Table \ref{Table1}.}
	On employing the Cauchy-Schwarz and Young's inequalities, we calculate
	\begin{align*}
		|(\mathfrak{B}(\bar{\z}),(\mathcal{A}+\I)\bar{\z})|&\leq\||\bar{\z}|\nabla\bar{\z}\|_{\H}
		\|(\mathcal{A}+\I)\bar{\z}\|_{\H}
	\leq\frac{\theta\mu}{2}\|(\mathcal{A}+\I)\bar{\z}\|_{\H}^2+
		\frac{1}{2\theta\mu}\||\bar{\z}|\nabla\bar{\z}\|_{\H}^2,
	\end{align*}
	for any $0<\theta<1.$ For $r=3$, we write \eqref{toreq} as
	\begin{align*}
		(\mathfrak{C}(\bar{\z}),(\mathcal{A}+\I)\bar{\z})=\||\bar{\z}|\nabla\bar{\z}\|_{\H}^{2}+ \frac{1}{2}\|\nabla|\bar{\z}|^2\|_{\H}^{2}+\|\bar{\z}\|_{\wi\L^{r+1}}^{r+1}.
	\end{align*}
	From \eqref{syymB3} and \eqref{torusequ}, we infer
	\begin{align}\label{re3c}
		-&\mu\|(\mathcal{A}+\I)\bar{\z}\|_{\H}^2+2\mu\|\bar{\z}\|_{\V}^2-
		2(\mathfrak{B}(\bar{\z})+\beta\mathfrak{C}(\bar{\z}),(\mathcal{A}+\I)\bar{\z})
		\nonumber\\&\leq -\mu(1-\theta)\|(\mathcal{A}+\I)\bar{\z}\|_{\H}^2+\mu\|\bar{\z}\|_{\V}^2-
		\left(2\beta-\frac{1}{\theta\mu}\right)\||\bar{\z}|\nabla\bar{\z}\|_{\H}^2-2\beta\|\bar{\z}\|_{\wi\L^{r+1}}^{r+1}.
	\end{align}

	On utilizing \eqref{rg3c}-\eqref{re3c} in \eqref{supsoLdef1}, we obtain
	\begin{align}\label{viscdef3}
	&\frac{\bar{t}-\bar{s}}{\eta}+\left(-\mathpzc{k}_\upgamma e^{\mathpzc{k}_\upgamma(T-\bar{t})}\|\bar{\z}\|_{\V}^2 +
		\frac{2C_\upgamma^2\delta}{\mu}(1+\|\bar{\z}\|_{\V}^2)+2 (\mu+\varrho_1)\|\bar{\z}\|_{\V}^2\right)
		\delta e^{\mathpzc{k}_\upgamma(T-\bar{t})}
		\nonumber\\&\quad-
		\bigg(\mu(\mathcal{A}+\I)\bar{\z}-\mu\bar{\z}+\mathfrak{B}(\bar{\z})+\mathfrak{C}(\bar{\z}),\frac{1}{\lambda}(\bar{\z}-\bar{\x})\bigg)+
		\mathcal{F}\big(\bar{t},\bar{\z},\frac{1}{\lambda}(\bar{\z}-\bar{\x})\big)
		\nonumber\\&\geq\frac{\upgamma}{2T^2}+\mathfrak{a}_3 \delta e^{\mathpzc{k}_\upgamma(T-\bar{t})} \||\bar{\z}|^{\frac{r-1}{2}}\nabla\bar{\z}\|_{\H}^2
		\nonumber\\&\quad+
		2\beta\delta e^{\mathpzc{k}_\upgamma(T-\bar{t})} \|\bar{\z}\|_{\wi\L^{r+1}}^{r+1}+
		\mathfrak{a}_2\mu\delta e^{\mathpzc{k}_\upgamma(T-\bar{t})} \|(\mathcal{A}+\I)\bar{\z}\|_{\H}^2,
	\end{align}
	where
	\begin{equation}\label{e3cbf1}
		\mathfrak{a}_2:=
		\left\{
		\begin{aligned}
			\frac12, \ \ &\text{for Case-I of Table \ref{Table1}},\\
			1-\theta, \ \ &\text{for Case-II of Table \ref{Table1}}.
		\end{aligned}
		\right.
	\end{equation}
	and 
	where
	\begin{equation}\label{e3cbf}
		\mathfrak{a}_3:=
		\left\{
		\begin{aligned}
			\beta, \ \ &\text{for Case-I of Table \ref{Table1}},\\
			2\beta-\frac{1}{\theta\mu}, \ \ &\text{for Case-II of Table \ref{Table1}}.
		\end{aligned}
		\right.
	\end{equation}
  	Let us now choose $$\mathpzc{k}_\upgamma=1+2\left(\frac{2 C_\upgamma^2 }{\mu}+2\mu +\mathfrak{a}_1
  \right),$$  where
  \begin{equation*}
  	\mathfrak{a}_1:=
  	\left\{
  	\begin{aligned}
  		2\varrho_1, \ \ &\text{for Case-I of Table \ref{Table1}},\\
  		0, \ \ &\text{for Case-II of Table \ref{Table1}}.
  	\end{aligned}
  	\right.
  \end{equation*} 
   Along with this, \eqref{viscdef3} gives 
	\begin{align}\label{viscdef4}
		&\frac{\bar{t}-\bar{s}}{\eta}-\frac{\delta}{2}\mathpzc{k}_\upgamma e^{\mathpzc{k}_\upgamma(T-\bar{t})} \|\bar{\z}\|_{\V}^2+\frac{\mathpzc{k}_\upgamma-1}{2}e^{\mathpzc{k}_\upgamma T}
	   -\mu\bigg((\mathcal{A}+\I)\bar{\z},\frac{\bar{\z}-\bar{\x}}{\lambda}\bigg)
	   +\mu\bigg(\bar{\z},\frac{\bar{\z}-\bar{\x}}{\lambda}\bigg) \nonumber\\&\quad-
		\bigg(\mathfrak{B}(\bar{\z})+\beta\mathfrak{C}(\bar{\z}),\frac{\bar{\z}-\bar{\x}}{\lambda}
		\bigg)+\mathcal{F}\bigg(\bar{t},\bar{\z},
		\frac{\bar{\z}-\bar{\x}}{\lambda}\bigg)
		\geq\frac{\upgamma}{2T^2}.
	\end{align}
	Similarly, the function 
	\begin{align*}
		(s,\x)\mapsto-\Phi(\bar{t},s,\bar{\z},\x)&=\mathpzc{v}_{\upgamma}(s,\x)-\mathpzc{u}_{\upgamma}(\bar{t},\bar{\z})+
		\frac{\|\bar{\z}-\x\|_{\H}^2}{2\lambda}+\delta e^{\mathpzc{k}_\upgamma(T-\bar{t})} \|\bar{\z}\|_{\V}^2\\&\quad+\delta e^{\mathpzc{k}_\upgamma(T-s)} \|\x\|_{\V}^2+\frac{(\bar{t}-s)^2}{2\eta}
	\end{align*} 
	has a global minimum at $(\bar{s},\bar{\x})$ in $(0,T)\times\H$.
	Similarly, one can obtain
	\begin{align}\label{viscdef5}
		&\frac{\bar{t}-\bar{s}}{\eta}+\frac{\delta}{2}\mathpzc{k}_\upgamma e^{\mathpzc{k}_\upgamma(T-\bar{s})} \|\bar{\x}\|_{\V}^2
		-\frac{\mathpzc{k}_\upgamma-1}{2}e^{\mathpzc{k}_\upgamma T}
	-\mu\bigg((\mathcal{A}+\I)\bar{\x},\frac{\bar{\z}-\bar{\x}}{\lambda}\bigg)
	+\mu\bigg(\bar{\x},\frac{\bar{\z}-\bar{\x}}{\lambda}\bigg) \nonumber\\&\quad-
	\bigg(\mathfrak{B}(\bar{\x})+\beta\mathfrak{C}(\bar{\x}),\frac{\bar{\z}-\bar{\x}}{\lambda}
	\bigg)+
	\mathcal{F}\bigg(\bar{s},\bar{\x},\frac{\bar{\z}-\bar{\x}}{\lambda}\bigg)
	\nonumber\\&\leq-\frac{\upgamma}{2T^2}.
	\end{align}
	On combining \eqref{viscdef4}-\eqref{viscdef5}, we find
	\begin{align}\label{viscdef6}
		&\frac{\delta}{2} \|\bar{\z}\|_{\V}^2+\frac{\delta}{2} \|\bar{\x}\|_{\V}^2
		+\frac{\mu}{\lambda}\|\nabla(\bar{\z}-\bar{\x})\|_{\H}^2
	+
		\bigg(\mathfrak{B}(\bar{\z})-\mathfrak{B}(\bar{\x}),\frac{\bar{\z}-\bar{\x}}{\lambda}\bigg)	\nonumber\\&\quad+\beta
		\bigg(\mathfrak{C}(\bar{\z})-\mathfrak{C}(\bar{\x}),\frac{\bar{\z}-\bar{\x}}{\lambda}\bigg)
		+\mathcal{F}\bigg(\bar{t},\bar{\z},\frac{\bar{\z}-\bar{\x}}{\lambda}\bigg)-
		\mathcal{F}\bigg(\bar{s},\bar{\x},\frac{\bar{\z}-\bar{\x}}{\lambda}\bigg)
		\nonumber\\&\leq-\frac{\upgamma}{T^2}.
	\end{align}
	 \vskip 0.2cm
	\noindent
	\textbf{Step-III:} \emph{Conclusion.} 
	For fixed $\upgamma$ and $\delta$, we have $\|\bar{\z}\|_{\V}, \|\bar{\x}\|_{\V}\leq R_\delta$ for some $R_\delta>0$. Let $D_{\upgamma,\delta}>0$ be such that
	\begin{align*}
	\omega_{R_\delta}(s)\leq\frac{\upgamma}{4T^2}+D_{\upgamma,\delta}s.
	\end{align*}
	In view of \eqref{F1} and \eqref{F3}, we calculate
	\begin{align}\label{Fhm}
     &\bigg|\mathcal{F}\bigg(\bar{t},\bar{\z},\frac{\bar{\z}-\bar{\x}}{\lambda}
     \bigg)-
	\mathcal{F}\bigg(\bar{s},\bar{\x},\frac{\bar{\z}-\bar{\x}}{\lambda}\bigg)
	\bigg|
	\nonumber\\&\leq
	\omega_{R_{\delta,\lambda}}(|\bar{t}-\bar{s}|)
	+\omega_{R_{\delta}}(\|\bar{\z}-\bar{\x}\|_{\V})
	+\omega\left(\|\bar{\z}-\bar{\x}\|_{\V}\frac{\|\bar{\z}-\bar{\x}\|_{\H}}
	{\lambda}\right)\nonumber\\&\leq
	\omega_{R_{\delta,\lambda}}(|\bar{t}-\bar{s}|)+\frac{3\upgamma}{4T^2}+D_{\upgamma,\delta}\|\bar{\z}-\bar{\x}\|_{\V}+C_\upgamma
	\|\bar{\z}-\bar{\x}\|_{\V}\frac{\|\bar{\z}-\bar{\x}\|_{\H}}{\lambda}
	\nonumber\\&\leq\omega_{R_{\delta,\lambda}}(|\bar{t}-\bar{s}|)+\frac{3\upgamma}{4T^2}+D_{\upgamma,\delta}
	\|\bar{\z}-\bar{\x}\|_{\V}+\frac{C_\upgamma^2} {\mu\lambda}\|\bar{\z}-\bar{\x}\|_{\H}^2
	+\frac{\mu}{4\lambda}\|\bar{\z}-\bar{\x}\|_{\V}^2.
	\end{align}
		 From \eqref{3.4} and \eqref{monoC2}, we calculate
	\begin{align}\label{rg3c1}
	&\frac{1}{\lambda}(\mathfrak{B}(\bar{\z})-\mathfrak{B}(\bar{\x}),\bar{\z}-\bar{\x})+\frac{\beta}{\lambda}
	(\mathfrak{C}(\bar{\z})-\mathfrak{C}(\bar{\x}),\bar{\z}-\bar{\x})\nonumber\\&\geq-
	\frac{\mu}{4\lambda}\|\nabla(\bar{\z}-\bar{\x})\|_{\H}^2 -\frac{\varrho_3}{\lambda}\|\bar{\z}-\bar{\x}\|_{\H}^2,
	\end{align}
where $\varrho_3:=\frac{r-3}{\mu(r-1)}\left[\frac{4}{\beta\mu (r-1)}\right]^{\frac{2}{r-3}}$. Combining \eqref{Fhm} and \eqref{rg3c1}, using in \eqref{viscdef6} and then, since $\bar{\z}$ and $\bar{\x}$ are in $\V$, we conclude that
	\begin{align}\label{viscdef7}
	&\frac{\delta}{2}\|\bar{\z}\|_{\V}^2+\frac{\delta}{2}\|\bar{\x}\|_{\V}^2
	+\frac{\mu}{2\lambda}\|\bar{\z}-\bar{\x}\|_{\V}^2
	\nonumber\\&\leq-\frac{\upgamma}{4T^2}+
	D_{\upgamma,\delta}\|\bar{\z}-\bar{\x}\|_{\V}+\omega_{R_{\delta,\lambda}}(|\bar{t}-\bar{s}|)+\left(\frac{C_\upgamma^2}{\mu\lambda}+\frac{3\mu}{4\lambda}+\frac{\varrho_3}{\lambda}\right)\|\bar{\z}-\bar{\x}\|_{\H}^2.
   \end{align}
   By using \eqref{copm1}-\eqref{copm2}, we finally have
  \begin{align*}
  	\frac{\mu}{2\lambda}\|\bar{\z}-\bar{\x}\|_{\V}^2-
  	D_{\upgamma,\delta}
  	\|\bar{\z}-\bar{\x}\|_{\V}
  \leq-\frac{\upgamma}{4T^2}+
  	\upsigma_2(\eta,\lambda;\delta,\upgamma),
  \end{align*}
	where, for fixed $\upgamma,\delta$, we have $\limsup\limits_{\lambda\to0}
	\limsup\limits_{\eta\to0}\upsigma_2(\eta,\lambda;\delta,\upgamma)=0$.
On taking the infimum, we obtain
	\begin{align}\label{contd1}
     \inf\limits_{\|\bar{\z}-\bar{\x}\|_{\V}>0}  
     \left\{\frac{\mu}{2\lambda}\|\bar{\z}-\bar{\x}\|_{\V}^2-
	D_{\upgamma,\delta}
	\|\bar{\z}-\bar{\x}\|_{\V}\right\}\leq-\frac{\upgamma}{4T^2}+
	\upsigma_2(\eta,\lambda;\delta,\upgamma).
	\end{align} 
	Now on taking $\limsup\limits_{\lambda\to0}\limsup\limits_{\eta\to0}$ in \eqref{contd1} and using the fact that $$\lim\limits_{\lambda\to0}\inf\limits_{r>0}\left(\frac{\mu r^2}{4\lambda}-D_{\upgamma,\delta}r\right)=0,$$
	we obtain
	\begin{align*}
	 0\leq-\frac{\upgamma}{4T^2} \ \text{ or } \ \upgamma<0,
	\end{align*}
	which is a contradiction to $\upgamma>0$ and hence $\mathpzc{u}\leq\mathpzc{v}$ on $(0,T]\times\V$.
		\end{proof}
		
	\section{Existence and uniqueness of viscosity solutions}\label{extunqvisc1}\setcounter{equation}{0}
	This section revisits the HJBE \eqref{HJBE} where the  Hamiltonian $F$ is given by \eqref{hamfun}. We demonstrate that
	the value function satisfies the HJBE \eqref{HJBE} in the viscosity sense.

\begin{theorem}\label{extunqvisc}
	Assume Hypothesis \ref{valueH} holds. Furthermore, consider the function $\f:[0,T]\times\U\to\V$ such that for each control parameter $\a\in\U$, the mapping $t\mapsto\f(t,\a)$ is uniformly continuous on $[0,T]$. Then, for $r$ in Table \ref{Table2}, the value function $\mathpzc{V}$ satisfies HJBE \eqref{detHJB1}, uniquely, in the viscosity sense within the class of viscosity solutions $u$ satisfying the polynomial growth condition:
	\begin{align*}
		|\mathpzc{u}(t,\z)|\leq C(1+\|\z\|_{\V}^k), \  \ (t,\z)\in(0,T)\times\V,
	\end{align*} 
	for some $C>0$ and $k\geq0$.
\end{theorem}

\begin{proof}
	Note that under the framework of Hypothesis \ref{valueH}, the Hamiltonian $F$ specified in \eqref{hamfun} satisfies all criteria as outlined in Hypothesis \ref{hypF14}. Furthermore, Proposition \ref{valueP} guarantees that the value function $\mathpzc{V}$ satisfies the conditions \eqref{vdpp}-\eqref{bddV1}, and given \eqref{bddV1}, we infer that  $\mathpzc{V}$ is weakly sequentially continuous on $(0, T]\times\V$. To complete the proof, it remains to verify that $\mathpzc{V}$ constitutes a \emph{viscosity solution} to the HJBE \eqref{HJBE}. The uniqueness of the viscosity solution follows immediately from Theorem \ref{comparison}. The strategy of the proof is as follows:
	\begin{itemize}
		\item Demonstrate that the points of extrema (minima for supersolution and maxima for subsolution) in the definition of viscosity solution lie in the space $\V_2$. 
		\item Employ the dynamic programming principle \eqref{vdpp} along with the various energy estimates for solutions of the state equation \eqref{stap} to derive the supersolution and subsolution inequalities as mentioned in the Definition \ref{viscsoLndef} through limit arguments.
	\end{itemize}
	For brevity, we present only the proof of viscosity supersolution. The subsolution case follows analogously. Consider the test function $\psi(t,\z)=\upvarphi(t,\z)+\delta(t)\|\z\|_{\V}^2$ and suppose that $\mathpzc{V}+\psi$ attains a global minimum at $(t_0,\z_0)\in(0,T)\times\V$.

	\vskip 1mm
	\noindent
	\textbf{To prove $\z_0\in\V_2$.}  In view of \eqref{vdpp}, for every $\lambda>0,$ there exists $\a_\lambda(\cdot)\in\mathscr{U}$ such that
	\begin{align}\label{vdp1}
		\mathpzc{V}(t_0,\z_0)+\lambda^2>\int_{t_0}^{t_0+\lambda}
		\mathpzc{L}(\Z_\lambda(s),\a_\lambda(s))\d s+ \mathpzc{V}(t_0+\lambda,\Z_\lambda(t_0+\lambda)),
	\end{align}
	where $\Z_\lambda(\cdot)=\Z(\cdot;t_0,\z_0,\a_\lambda(\cdot))$.
	 Since  $(\mathpzc{V}+\psi)(t_0,\y_0)\leq(\mathpzc{V}+\psi)(t,\y)$ for every $(t,\y)\in(0,T)\times\V$, then from \eqref{vdp1}, we have
	\begin{align}\label{vdp3}
		\lambda^2-\int_{t_0}^{t_0+\lambda}\mathpzc{L}(\Z_\lambda(s),\a_\lambda(s))\d s&\geq \mathpzc{V}(t_0+\lambda,\Z_\lambda(t_0+\lambda))-\mathpzc{V}(t_0,\y_0)\nonumber\\&\geq
		-\upvarphi(t_0+\lambda,\Z_\lambda(t_0+\lambda))+\upvarphi(t_0,\z_0)
		\nonumber\\&\quad- \delta(t_0+\lambda)\|\Z_\lambda(t_0+\lambda)\|_{\V}^2+
		\delta(t_0)\|\z_0\|_{\V}^2.
	\end{align}
	Let us set $\mathpzc{a}=\inf\limits_{t\in[t_0,t_0+\lambda_0]}\delta(t)$ for some $\lambda_0>0$. Then by making the use of chain rule and dividing both sides by $\lambda$, for $\lambda<\lambda_0$, we obtain
	\begin{align}\label{vdp3.1}
		&\lambda-\frac{1}{\lambda}\E\int_{t_0}^{t_0+\lambda}\mathpzc{L}(\Z_\lambda(s),\a_\lambda(s))\d s
		\nonumber\\&\geq
		-\frac{1}{\lambda}\int_{t_0}^{t_0+\lambda}\bigg(\upvarphi_t(s,\Z_\lambda(s))-\big(\mu\mathcal{A}\Z_\lambda(s)+\mathfrak{B}(\Z_\lambda(s))+\beta\mathfrak{C}(\Z_\lambda(s)),\mathcal{D}\upvarphi(s,\Z_\lambda(s))\big)
		\nonumber\\&\qquad+
		\big(\f(s,\a_\lambda(s)),\mathcal{D}\upvarphi(s,\Z_\lambda(s))\big)\bigg)\d s
		-\frac{1}{\lambda}\int_{t_0}^{t_0+\lambda}\delta'(s)\|\Z_\lambda(s)\|_{\V}^2)\d s
		\nonumber\\&\qquad+ \frac{2}{\lambda}\int_{t_0}^{t_0+\lambda}\delta(s)\bigg(\mu\big(\mathcal{A}\Z_\lambda(s),(\mathcal{A}+\I)\Z_\lambda(s)\big)+
		(\mathfrak{B}(\Z_\lambda(s)),(\mathcal{A}+\I)\Z_\lambda(s))
		\nonumber\\&\qquad+
		\beta(\mathfrak{C}(\Z_\lambda(s)),(\mathcal{A}+\I)\Z_\lambda(s))-\big(\f(s,\a_\lambda(s)),
		(\mathcal{A}+\I)\Z_\lambda(s)\big)\big)\bigg)\d s.
	\end{align}
	By the definition of $\mathpzc{a}$ and from the equality \eqref{torusequ}, it then follows that
	\begin{align}\label{vdp4}
		&\frac{2\mathpzc{a}}{\lambda}\int_{t_0}^{t_0+\lambda} \bigg(\beta(\mathfrak{C}(\Z_{\lambda}),(\mathcal{A}+\I)\Z_{\lambda})+\mu\|(\mathcal{A}+\I)\Z_\lambda(s)\|_{\H}^2\bigg)\d s
		\nonumber\\&\leq\lambda+
		\frac{1}{\lambda}\int_{t_0}^{t_0+\lambda}\bigg(-\mathpzc{L}(\Z_\lambda(s),\a_\lambda(s))+\upvarphi_t(s,\Z_\lambda(s))+(\f(s,\a_\lambda(s)),\mathcal{D}\upvarphi(s,\Z_\lambda(s)))\nonumber\\&\qquad-\big(\mu\mathcal{A}\Z_\lambda(s)+\mathfrak{B}(\Z_\lambda(s))+\beta\mathfrak{C}(\Z_\lambda(s)),\mathcal{D}\upvarphi(s,\Z_\lambda(s))\big)\bigg)\d s \nonumber\\&\qquad+
		\frac{1}{\lambda}\int_{t_0}^{t_0+\lambda}(\delta'(s)+2\mu\delta(s))
		\|\Z_\lambda(s)\|_{\V}^2\d s
		-
		\frac{2}{\lambda}\int_{t_0}^{t_0+\lambda}
		\delta(s)\big((\mathfrak{B}(\Z_\lambda(s)),(\mathcal{A}+\I)\Z_\lambda(s))\d s
		\nonumber\\&\qquad+\frac{2}{\lambda}
	    \int_{t_0}^{t_0+\lambda}\delta(s)
		(\f(s,\a_\lambda(s)),(\mathcal{A}+\I)\Z_\lambda(s))\big)\d s.
	\end{align}
	We now estimate all the terms in the right hand side of \eqref{vdp4} separately. By the assumptions on $\upvarphi$, we deduce 
	\begin{align}\label{unsol1}
		\|\mathcal{D}\upvarphi(\cdot,\Z_\lambda)\|_{\H}\leq C(1+\|\Z_\lambda\|_{\H}), \ \ 
		\|\upvarphi_t(\cdot,\Z_\lambda)\|_{\H}\leq C(1+\|\Z_\lambda\|_{\H}).
	\end{align} 
	Using \eqref{unsol1}, we estimate following:
	\begin{align}
		|\mathpzc{L}(\Z_\lambda,\a_\lambda)|&\leq C\big(1+\|\Z_\lambda\|_{\V}^k\big),\label{vdp6}\\
		|(\f(\cdot,\a_\lambda),\mathcal{D}\upvarphi(\cdot,\Z_\lambda))|&\leq 
		C\big(1+\|\Z_\lambda\|_{\H}\big),\label{vdp6.1}\\
		|(\f(\cdot,\a_\lambda),(\mathcal{A}+\I)\Z_\lambda)|&=|((\mathcal{A}+\I)^{\frac12}\f(\cdot,\a_\lambda),(\mathcal{A}+\I)^{\frac12}\Z_\lambda)|\leq C\|\Z_\lambda\|_{\V},\label{vdp6.2}\\
		|(\f(\cdot,\a_\lambda),\mathcal{D}\upvarphi(\cdot,\Z_\lambda))|&\leq C\big(1+\|\Z_\lambda\|_{\H}\big).\label{vdp555}
	\end{align}
	We now consider following two cases for further calculations. 
	\vskip 2mm
	\noindent
	\textbf{Case-I of in Table \ref{Table1}.}
	By making the use of Cauchy Schwarz and Young's inequalities, \eqref{unsol1} and using the similar calculations as in \eqref{syymB3}, we estimate following:
	\begin{align}
	\mu|(\mathcal{A}\Z_\lambda,\mathcal{D}\upvarphi(\cdot,\Z_\lambda))|&\leq\mu|((\mathcal{A}+\I)\Z_\lambda,\mathcal{D}\upvarphi(\cdot,\Z_\lambda))|\nonumber\\&\leq \frac{\mu\mathpzc{a}}{2}\|(\mathcal{A}+\I)\Z_\lambda\|_{\H}^2+C(1+\|\Z_\lambda\|_{\H}^2),
	\label{vdp5}\\
		|(\mathfrak{B}(\Z_\lambda),(\mathcal{A}+\I)\Z_\lambda)|
		&\leq\frac{\mu}{2}
		\|(\mathcal{A}+\I)\Z_\lambda\|_{\H}^2+\frac{\beta}{4}\||\Z_\lambda|^{\frac{r-1}{2}}\nabla\Z_\lambda\|_{\H}^2+\varrho^*\|\nabla\Z_\lambda\|_{\H}^2,\label{vdp5.1}\\
		|(\mathfrak{B}(\Z_\lambda),\mathcal{D}\upvarphi(\cdot,\Z_\lambda))|&\leq C(1+\|\Z_\lambda\|_{\H}^2)
		+\frac{\beta\mathpzc{a}}{4}\||\Z_\lambda|^{\frac{r-1}{2}}\nabla\Z_\lambda\|_{\H}^2+\varrho^{**}\|\nabla\Z_\lambda\|_{\H}^2,\label{vdp5.2}
	\end{align}
	where 
	$\varrho^{**}:=\frac{r-3}{2\mathpzc{a}(r-1)}\left[\frac{4}{\beta\mathpzc{a} (r-1)}\right]^{\frac{2}{r-3}}$ and  $\varrho^*:=\frac{r-3}{2\mu(r-1)}\left[\frac{4}{\beta
		\mu(r-1)}\right]^{\frac{2}{r-3}}$. 
	By using a calculation similar to \eqref{ctsdep7}, \eqref{unsol1} and Remark \ref{rg3L3r}, we find
	\begin{align}\label{vdp55}
		|(\mathfrak{C}(\Z_\lambda),\mathcal{D}\upvarphi(\cdot,\Z_\lambda))|&\leq	\|\Z_\lambda\|_{\wi\L^{r+1}}^{\frac{r+3}{4}}\|\Z\|_{\wi\L^{3(r+1)}}^{\frac{3(r-1)}{4}}\|\mathcal{D}\upvarphi(\cdot,\Z_\lambda)\|_{\H}
		\nonumber\\&\leq
		\frac{\mathpzc{a}}{4}\||\Z_\lambda|^{\frac{r-1}{2}}
		\nabla\Z_\lambda\|_{\H}^2+
		C(1+\|\Z_\lambda\|_{\H})^{r+1}+\mathpzc{a}\|\Z_\lambda\|_{\wi\L^{r+1}}^{r+1}.
	\end{align}
	Also, from the equality \eqref{toreq}, we write
	\begin{align}\label{vdp5.3}
		(\mathfrak{C}(\Z_{\lambda}),(\mathcal{A}+\I)\Z_{\lambda})\geq
		\||\Z_\lambda(s)|^{\frac{r-1}{2}}\nabla\Z_\lambda(s)\|_{\H}^2+
		\|\Z_\lambda(s)\|_{\wi\L^{r+1}}^{r+1}.
	\end{align}
	\vskip 2mm
	\noindent
	\textbf{Case-II of in Table \ref{Table1}.}
	From calculations similar to \eqref{syymB3}, we obtain for $0<\theta<1$
	\begin{align}
		\mu|(\mathcal{A}\Z_\lambda,\mathcal{D}\upvarphi(\cdot,\Z_\lambda))|&\leq \frac{3\mu\mathpzc{a}}{2}\|(\mathcal{A}+\I)\Z_\lambda\|_{\H}^2+C(1+\|\Z_\lambda\|_{\H}^2),\label{vdp5.555}\\
		|(\mathfrak{B}(\Z_\lambda),(\mathcal{A}+\I)\Z_\lambda)|&\leq\||\Z_\lambda|\nabla\Z_\lambda\|_{\H}
		\|(\mathcal{A}+\I)\Z_\lambda\|_{\H}
		\nonumber\\&\leq\frac{\theta\mu}{2}\|(\mathcal{A}+\I)\Z_\lambda\|_{\H}^2+
		\frac{1}{2\theta\mu}\||\Z_\lambda|\nabla\Z_\lambda\|_{\H}^2,
		\label{vdp5.4}\\
		|(\mathfrak{B}(\Z_\lambda),\mathcal{D}\upvarphi(\cdot,\Z_\lambda))|&\leq C(1+\|\Z_\lambda\|_{\H}^2)
		+\frac{\beta\mathpzc{a}}{2}\||\Z_\lambda|\nabla\Z_\lambda\|_{\H}^2.
		\label{vdp5.8}
	\end{align}
	Similar to \eqref{vdp55}, we obtain the following estimate:
	\begin{align*}
		|(\mathfrak{C}(\Z_\lambda),\mathcal{D}\upvarphi(\cdot,\Z_\lambda))|\leq
		\frac{\mathpzc{a}}{2}\||\Z_\lambda|\nabla\Z_\lambda\|_{\H}^2+
		C(1+\|\Z_\lambda\|_{\H})^{r+1}+\mathpzc{a}\|\Z_\lambda\|_{\wi\L^{r+1}}^{r+1}.
	\end{align*}
	For $r=3$, we write \eqref{toreq} as
	\begin{align}\label{vdp5.5}
		(\mathfrak{C}(\Z_\lambda),(\A+\I)\Z_\lambda)\geq\||\Z_\lambda|\nabla \Z_{\lambda}\|_{\H}^{2} +\|\Z_\lambda\|_{\wi\L^{4}}^{4}.
	\end{align}
	
	Combining \eqref{vdp6}-\eqref{vdp5.5} into \eqref{vdp4},
	and using \eqref{eqn-conv-1} and \eqref{eqn-conv-2}, we obtain
	\begin{align}\label{vdp7}
		\frac{\mathpzc{a}}{\lambda}\int_{t_0}^{t_0+\lambda}& \bigg(\mathfrak{a}_3
		\||\Z_\lambda(s)|^{\frac{r-1}{2}}\nabla\Z_\lambda(s)\|_{\H}^2+
		2\beta\|\Z_\lambda(s)\|_{\wi\L^{r+1}}^{r+1}+
		\frac{\mathfrak{a}_2\mu}{2}\|(\A+\I)\Z_\lambda(s)\|_{\H}^2\bigg)\d s\leq C,
	\end{align}
   where the constants $\mathfrak{a}_2$ and $\mathfrak{a}_3$ are defined in \eqref{e3cbf1}-\eqref{e3cbf}. Consequently, there exists a sequence $\{\lambda_n\}_{n\geq1}$, with $\lambda_n\to0$ and a sequence  $\{t_n\}_{n\geq1}$ in $(t_0,t_0+\lambda_n)$ such that
	\begin{align}\label{vdp7.1}
		\|\Z_{\lambda_n}(t_n)\|_{\V_2}^2\leq C, 
	\end{align}  
	and thus by an application of the Banach-Alaoglu theorem, we have 
following weak convergence along a subsequence (still denoted by $\lambda_n$ and $t_n$)
	\begin{align}\label{weakL}
		\Z_{\lambda_n}(t_n)\rightharpoonup\bar{\z} \ \text{ in } \ \V_2, \ \text{ as } \ n\to\infty.
	\end{align}
	for some $\bar{\z}\in\V_2$. However, in view of \eqref{ctsdep0.1}, we have following strong convergence:
	\begin{align}\label{strgL}
		\Z_{\lambda_n}(t_n)\to\z_0 \ \text{ in } \ \H.
	\end{align}
	Therefore, by using the fact that weak limits are uniqu, \eqref{weakL} and \eqref{strgL}, we obtain
	\begin{align*}
		\z_0=\bar{\z}\in\V_2.
	\end{align*}

\vskip 0.2mm
\noindent
\textbf{To prove the supersolution inequality:}  To establish the supersolution inequality, at least along a subsequence, we analyse the limit as $\lambda\to0$ in \eqref{vdp4}. For the convergence of the terms involving the norms of the space $\H$ and $\V$, one may apply the standard compactness argument due to the uniform energy estimates \eqref{eqn-conv-1} and \eqref{eqn-conv-2}. 

We will begin by demonstrating the convergence of the cost term on the right-hand side of \eqref{vdp4}. From uniform energy estimates \eqref{eqn-conv-1} and \eqref{eqn-conv-2}, we have
\begin{align*}
	\sup\limits_{s\in[t_0,T]}\|\Z_\lambda(s)\|_{\V}\leq C(\|\z_0\|_{\V})=R.
\end{align*}
Then, by using Jensen's and Holder's inequalities, \eqref{ctsdep0.2}, we calculate
\begin{align}\label{ss1} 
	\bigg|\frac{1}{\lambda}\int_{t_0}^{t_0+\lambda}\big[\mathpzc{L}(\Z_\lambda(s),\a_\lambda(s))-\mathpzc{L}(\z_0,\a_\lambda(s))\big]\d s \bigg|&\leq
	\frac{1}{\lambda}\int_{t_0}^{t_0+\lambda}\sigma_r\big(\|\Z_\lambda(s)-\z_0\|_{\V}\big)\d s
	\nonumber\\&\leq
	\sigma_r\bigg(\frac{1}{\lambda}\int_{t_0}^{t_0+\lambda}\|\Z_\lambda(s)-\z_0\|_{\V}\d s\bigg)
	\nonumber\\&\leq
	\sigma_r\bigg(\frac{1}{\lambda}\int_{t_0}^{t_0+\lambda}\sqrt{\omega_{\z_0}
		(\lambda)}\d s\bigg),
\end{align}
where right hand side is going to zero by letting $\lambda\to0$.
Similalry, by using the continuity of $\upvarphi_t$ (see Definition \ref{testD}), we obtain
\begin{align*}
	\bigg|\frac{1}{\lambda}\int_{t_0}^{t_0+\lambda}\big(\upvarphi_t(s,\Z_\lambda(s))-
	\upvarphi_t(t_0,\z_0)\big)\d s\bigg|\leq\frac{1}{\lambda}\int_{t_0}^{t_0+\lambda}
	\omega_1(\lambda+\|\Z_\lambda(s)-\z_0\|_{\H})\d s,
\end{align*}
where $\omega_1$ is some local modulus of continuity. Let us now discuss the limit as $\lambda\to0$ in the following terms:
\begin{align}
	&\frac{1}{\lambda}\int_{t_0}^{t_0+\lambda}(\mathcal{A}\Z_\lambda(s),\mathcal{D}\upvarphi(s,\Z_\lambda(s)))\d s, \label{pas1}\\	&\frac{1}{\lambda}\int_{t_0}^{t_0+\lambda}(\mathfrak{B}(\Z_\lambda(s)),\mathcal{D}\upvarphi(s,\Z_\lambda(s)))\d s, \label{pas2}\\
	&\frac{1}{\lambda}\int_{t_0}^{t_0+\lambda}(\beta\mathfrak{C}(\Z_\lambda(s)),\mathcal{D}\upvarphi(s,\Z_\lambda(s)))\d s,  \label{pas3}\\ 
	&\frac{1}{\lambda}\int_{t_0}^{t_0+\lambda}\big(\f(s,\a_\lambda(s)),
	\mathcal{D}\upvarphi(s,\Z_\lambda(s)\big)\d s, \label{pas4}\\
	&\frac{1}{\lambda}\int_{t_0}^{t_0+\lambda}\delta(s)\big(\f(s,\a_\lambda(s)), (\mathcal{A}+\I)\Z_\lambda(s)\big)\big)\d s, \label{pas5}\\
	&\frac{1}{\lambda}\int_{t_0}^{t_0+\lambda}\delta(s)\|(\mathcal{A}+\I)\Z_\lambda(s)\|_{\H}^2\d s, \label{pas6}\\
	&\frac{1}{\lambda}\int_{t_0}^{t_0+\lambda}\delta(s)\big(\Z_\lambda(s),(\mathcal{A}+\I)\Z_\lambda(s)\big)\d s,
	\label{pas8}\\
	&\frac{1}{\lambda}\int_{t_0}^{t_0+\lambda}\delta(s)(\mathfrak{B}(\Z_\lambda(s)),(\mathcal{A}+\I)\Z_\lambda(s))\d s, \label{pas9}\\
	&\frac{1}{\lambda}\int_{t_0}^{t_0+\lambda}\delta(s)(\mathfrak{C}(\Z_\lambda(s)),(\mathcal{A}+\I)\Z_\lambda(s))\d s. \label{pas10}
\end{align}
For the rest of the proof, we will use various moduli, denoted by $\upsigma(\cdot)$.
\vskip 0.2mm
\noindent
\textbf{Passing limit into \eqref{pas4}.} Since $\f:[0,T]\times\U\to\V$ is uniformly continuous, therefore, we have 
\begin{align}\label{modf}
	\|\f(s,\a_\lambda(s))-\f(t,\a_\lambda(s))\|_{\V}\leq\omega_{\f}(|s-t|), \ \text{ for all } \ s,t\in[t_0,t_0+\lambda],
\end{align} 
for some modulus of continuity $\omega_{\f}$. Similarly, due to the uniform continuity of $\mathcal{D}\upvarphi$ on $[t_0,t_0+\lambda]\times\H$, for every $\lambda>0$, (see Definition \ref{testD}), we write
\begin{align}\label{modphi}
	\|\mathcal{D}\upvarphi(s,\Z_\lambda(s)-\mathcal{D}\upvarphi(t_0,\z_0)\|_{\H}\leq
	\omega_2(\lambda+\|\Z_\lambda(s)-\z_0\|_{\H}), \ \text{ for all } \ s,t\in[t_0,t_0+\lambda],
\end{align}
where $\omega_2$ is some local modulus of continuity. From \eqref{modf}-\eqref{modphi}, we estimate \eqref{pas4} as
\begin{align*}
	&\bigg|\frac{1}{\lambda_n}\int_{t_0}^{t_0+\lambda_n}\big[\big(\f(s,\a_{\lambda_n}(s)),
	\mathcal{D}\upvarphi(s,\Z_{\lambda_n}(s)\big)-\big(\f(t_0,\a_{\lambda_n}(s)),
	\mathcal{D}\upvarphi(t_0,\z_0)\big)\big]\d s \bigg|
	\nonumber\\&\leq
	\frac{C}{\lambda_n}\int_{t_0}^{t_0+\lambda_n}\|\f(s,\a_{\lambda_n}(s))\|_{\V}
	\|\mathcal{D}\upvarphi(s,\Z_{\lambda_n}(s)-\mathcal{D}\upvarphi(t_0,\z_0)\|_{\H}\d s
	\nonumber\\&\quad+
	\frac{C}{\lambda_n}\int_{t_0}^{t_0+{\lambda_n}}\|\f(s,\a_{\lambda_n}(s))-\f(t_0,\a_{\lambda_n}(s)\|_{\V}\|\mathcal{D}\upvarphi(t_0,\z_0)\|_{\H}\d s
	\nonumber\\&\leq
	\frac{C}{\lambda_n}\int_{t_0}^{t_0+\lambda_n}
	\omega_2(\lambda_n+\|\Z_{\lambda_n}(s)-\z_0\|_{\H})\d s+
	\frac{C}{\lambda_n}\int_{t_0}^{t_0+{\lambda_n}}\omega_{\f}(\lambda_n)\d s
	\nonumber\\&\leq\upsigma(\lambda_n).
\end{align*}

\vskip 0.2mm
\noindent
\textbf{Passing limit into \eqref{pas6} and \eqref{pas8}.} 
By the application of H\"older's inequality and \eqref{vdp7}, we find
\begin{align}\label{vdp8}
	\left\|\frac{1}{\lambda}\int_{t_0}^{t_0+\lambda}\sqrt{\delta(s)}
	(\mathcal{A}+\I)\Z_\lambda(s)\d s \right\|_{\H}^2\leq
	\frac{1}{\lambda}\int_{t_0}^{t_0+\lambda}
	\delta(s)\|(\mathcal{A}+\I)\Z_\lambda(s)\|_{\H}^2\d s
	\leq C.
\end{align}
Therefore, there exists a sequence $\lambda_n\to0$ and $\wi\z\in\H$ such that 
\begin{align}\label{wknm}
	\wi\z_n:=\frac{1}{\lambda_n}\int_{t_0}^{t_0+\lambda_n}\sqrt{\delta(s)}(\mathcal{A}+\I)\Z_{\lambda_n}(s)\d s\rightharpoonup\wi\z \  \text{ in } \ \H,
\end{align}
as $n\to\infty$. 
Arguing similarly as we did in \eqref{ss1}, one can show that 
\begin{align}\label{wknmdiff}
	(\mathcal{A}+\I)^{-1}\wi\z_n=\frac{1}{\lambda_n}\int_{t_0}^{t_0+\lambda_n}\sqrt{\delta(s)}\Z_{\lambda_n}(s)\d s\to\sqrt{\delta(t_0)}\z_0 \ \text{ in } \H
	\ \text{ as } n\to\infty.
\end{align}
 Then the uniqueness of weak limits yields $$\wi\z=\sqrt{\delta(t_0)}(\mathcal{A}+\I) \z_0.$$
Moreover, in view of \eqref{eqn-conv-1} and \eqref{eqn-conv-2}, \eqref{ctsdep0.1}-\eqref{ctsdep0.2}, one can also verify the following limits:
\begin{align}
\frac{1}{\lambda_n}\int_{t_0}^{t_0+\lambda_n}\delta(s)\|\Z_{\lambda_n}(s)\|_{\H}^2\d s&\to\delta(t_0) \|\z_0\|_{\H}^2
	\label{wknm1.1}
\end{align}
and 
\begin{align}
	\frac{1}{\lambda_n}\int_{t_0}^{t_0+\lambda_n} \delta(s)\|\nabla\Z_{\lambda_n}(s)\|_{\H}^2\d s&\to\delta(t_0) \|\nabla\z_0\|_{\H}^2 ,\label{wknm1.2}
\end{align} 
in $\H$ as $n\to\infty$. From \eqref{vdp8}, \eqref{wknm}, Jensen's inequality and then weak lower semicontinuity property of norm yields the following estimate:
\begin{align}\label{wknm1}
	\liminf_{n\to\infty}\frac{1}{\lambda_n}\int_{t_0}^{t_0+\lambda_n}\delta(s) \|(\mathcal{A}+\I)\Z_{\lambda_n}(s)\|_{\H}^2
	\d s&\geq\liminf_{n\to\infty} \left\|\frac{1}{\lambda_n}\int_{t_0}^{t_0+{\lambda_n}}
	\sqrt{\delta(s)}(\mathcal{A}+\I)\Z_{\lambda_n}(s)\d s \right\|_{\H}^2\nonumber\\&\geq
	\delta(t_0)\|(\mathcal{A}+\I)\z_0\|_{\H}^2.
\end{align}
This will take care of the term \eqref{pas4}. A similar argument as we performed  above yields 
\begin{align}\label{wknm2}
	\frac{1}{\lambda_n}\int_{t_0}^{t_0+\lambda_n}(\mathcal{A}+\I)\Z_{\lambda_n}(s)\d s \rightharpoonup(\mathcal{A}+\I)\z_0 \  \text{ in } \ \H
	\  \text{ as } \ n\to\infty.
\end{align}

\vskip 0.2mm
\noindent
\textbf{Passing limit into the linear term \eqref{pas1}.} 
Let $\omega_{\upvarphi}$ denotes  a modulus of continuity for $\mathcal{D}\upvarphi$. Then, from the application of H\"older's inequality, \eqref{vdp7}, \eqref{ctsdep0.1} and \eqref{wknm2}, we estimate the following:
\begin{align}\label{wknm3}
	&\bigg|\frac{1}{\lambda_n}\int_{t_0}^{t_0+\lambda_n}\big(\mathcal{A}
	\Z_{\lambda_n}(s),\mathcal{D}\upvarphi(s,\Z_{\lambda_n}(s))\big)\d s-
	\big(\mathcal{A}\z_0,\mathcal{D}\upvarphi(t_0,\z_0)\big)\bigg|
	\nonumber\\&\leq
	\left(\frac{1}{\lambda_n}\int_{t_0}^{t_0+\lambda_n}\|\mathcal{A}
	\Z_{\lambda_n}(s)\|_{\H}^2\d s\right)^{\frac12} \left(\frac{1}{\lambda_n}\int_{t_0}^{t_0+\lambda_n}\big(\omega_\upvarphi\big(\lambda_n+\|\Z_{\lambda_n}(s)-\z_0\|_{\H}\big)\big)^2\d s\right)^{\frac12}
	\nonumber\\&\quad+
	\bigg|\bigg(\frac{1}{\lambda_n}\int_{t_0}^{t_0+\lambda_n}
	\mathcal{A}\big(\Z_{\lambda_n}(s)-\z_0\big),\mathcal{D}\upvarphi(t_0,\z_0)
	\bigg)\bigg|\nonumber\\&\leq\upsigma(\lambda_n).
\end{align}

\vskip 0.2mm
\noindent
\textbf{Passing limit into the term \eqref{pas5}.} Using \eqref{ctsdep0.2} and definition of modulus of continuity, one can conclude that 
{\small
	\begin{align*} 
		&\bigg|\frac{1}{\lambda_n}\int_{t_0}^{t_0+\lambda_n}\delta(s)\big(\f(s,\a_{\lambda_n}(s)), (\mathcal{A}+\I)\Z_{\lambda_n}(s)\big)d s-
		\frac{1}{\lambda_n}\int_{t_0}^{t_0+\lambda_n}\delta(t_0)\big(\f(s,\a_{\lambda_n}(s)), (\mathcal{A}+\I)\z_0\big)\d s\bigg|\nonumber\\&\leq
		\frac{1}{\lambda_n}\int_{t_0}^{t_0+\lambda_n}\delta(s)\bigg|\big((\mathcal{A}+\I)^{\frac12}\f(s,\a_{\lambda_n}(s)),(\mathcal{A}+\I)^{\frac12}\Z_{\lambda_n}(s)-(\mathcal{A}+\I)^{\frac12}\z_0\big)\bigg|\d s
		\nonumber\\&\quad+
		\frac{1}{\lambda_n}\int_{t_0}^{t_0+\lambda_n}\big|\delta(s)-\delta(t_0)\big|
		\big|\big((\mathcal{A}+\I)^{\frac12}\f(s,\a_{\lambda_n}(s)),(\mathcal{A}+\I)^{\frac12}\z_0\big)\big|\d s
		\nonumber\\&\leq
		C\left(\frac{1}{\lambda_n}\int_{t_0}^{t_0+\lambda_n}\|\Z_{\lambda_n}(s)-\z_0\|_{\V}^2\d s\right)^{\frac12}+
		\frac{C}{\lambda_n}\int_{t_0}^{t_0+\lambda_n}\big|\delta(s)-\delta(t_0)\big|
		\d s\nonumber\\&\leq\upsigma(\lambda_n).
\end{align*}}

\vskip 0.2mm
\noindent
\textbf{Passing limit into the Navier-Stokes nonlinearity \eqref{pas2} and \eqref{pas9}.} 
By using \eqref{syymB}, we write
\begin{align}\label{wknm5}
	&\bigg|\frac{1}{\lambda_n}\int_{t_0}^{t_0+\lambda_n}(\mathfrak{B}(\Z_{\lambda_n}(s)),\mathcal{D}\upvarphi(s,\Z_{\lambda_n}(s)))\d s-
	(\mathfrak{B}(\z_0),\mathcal{D}\upvarphi(t_0,\z_0))\bigg|
	\nonumber\\&\leq
	\underbrace{\bigg|\frac{1}{\lambda_n}\int_{t_0}^{t_0+\lambda_n}
		(\mathfrak{B}(\Z_{\lambda_n}(s)),\mathcal{D}\upvarphi(s,\Z_{\lambda_n}(s))-\mathcal{D}\upvarphi(t_0,\z_0))\d s \bigg|}_{:=J_1}
	\nonumber\\&\quad+
	\underbrace{\bigg|\frac{1}{\lambda_n}\int_{t_0}^{t_0+\lambda_n}
		(\mathfrak{B}(\Z_{\lambda_n}(s)-\z_0,\Z_{\lambda_n}(s)),\mathcal{D}\upvarphi(t_0,\z_0))\d s \bigg|}_{:=J_2}
	\nonumber\\&\quad+
	\underbrace{\bigg|\frac{1}{\lambda_n}\int_{t_0}^{t_0+\lambda_n}(\mathfrak{B}(\z_0,\Z_{\lambda_n}(s)-\z_0),\mathcal{D}\upvarphi(t_0,\z_0))\d s \bigg|}_{:=J_3}.
\end{align}	
By using Agmon's and H\"older's inequalities, and energy estimates \eqref{eqn-conv-1} and \eqref{eqn-conv-2}, and \eqref{vdp7}, we calculate
\begin{align}\label{J1B}
	J_1&\leq\frac{1}{\lambda_n}\int_{t_0}^{t_0+\lambda_n}
	\|\mathfrak{B}(\Z_{\lambda_n}(s))\|_{\H}\left(\omega_\upvarphi\big(\lambda_n+\|\Z_{\lambda_n}(s)-\z_0\|_{\H}\big)\right)\d s
	\nonumber\\&\leq
	\frac{1}{\lambda_n}\int_{t_0}^{t_0+\lambda_n}
	\|\Z_{\lambda_n}(s)\|_{\H}^{1-\frac{d}{4}}\|(\mathcal{A}+\I)\Z_{\lambda_n}(s)\|_{\H}^{\frac{d}{4}}\|\nabla\Z_{\lambda_n}(s)\|_{\H}
	\left(\omega_\upvarphi\big(\lambda_n+\|\Z_{\lambda_n}(s)-\z_0\|_{\H}\big)
	\right)\d s   
	\nonumber\\&\leq
	\frac{C}{\lambda_n}\int_{t_0}^{t_0+\lambda_n}
     \|(\mathcal{A}+\I)\Z_{\lambda_n}(s)\|_{\H}^{\frac{d}{4}}
	\left(\omega_\upvarphi\big(\lambda_n+\|\Z_{\lambda_n}(s)-\z_0\|_{\H}\big)
	\right)\d s     
	\nonumber\\&\leq
     \left(\frac{1}{\lambda_n}\int_{t_0}^{t_0+\lambda_n}\|(\mathcal{A}+\I)\Z_{\lambda_n}(s)\|_{\H}^{2}\d s
	\right)^{\frac{d}{8}} \left(\frac{1}{\lambda_n}\int_{t_0}^{t_0+\lambda_n}
	\big(\omega_\upvarphi\big(\lambda_n+\|\Z_{\lambda_n}(s)-\z_0\|_{\H}
	\big)\big)^{\frac{8}{8-d}}\d s\right)^{\frac{8-d}{8}}.
\end{align}
Similalry, we calculate
\begin{align}\label{J2B}
	J_2&\leq\frac{1}{\lambda_n}\int_{t_0}^{t_0+\lambda_n}
	\|\mathfrak{B}(\Z_{\lambda_n}(s)-\z_0,\Z_{\lambda_n}(s))\|_{\H}\left(\omega_\upvarphi\big(\lambda_n+\|\Z_{\lambda_n}(s)-\z_0\|_{\H}\big)\right)\d s
	\nonumber\\&\leq
	\frac{1}{\lambda_n}\int_{t_0}^{t_0+\lambda_n}
	\|\Z_{\lambda_n}(s)-\z_0\|_{\H}^{1-\frac{d}{4}}\|(\mathcal{A}+\I)(\Z_{\lambda_n}(s)-\z_0)\|_{\H}^{\frac{d}{4}}\|\nabla\Z_{\lambda_n}(s)\|_{\H}
	\nonumber\\&\qquad\times
	\left(\omega_\upvarphi\big(\lambda_n+\|\Z_{\lambda_n}(s)-\z_0\|_{\H}\big)
	\right)\d s \nonumber\\&\leq
	\frac{C}{\lambda_n}\int_{t_0}^{t_0+\lambda_n}
	\|\Z_{\lambda_n}(s)-\z_0\|_{\H}^{1-\frac{d}{4}}\left(\|(\mathcal{A}+\I)\Z_{\lambda_n}(s)\|_{\H}^{\frac{d}{4}}+\|(\mathcal{A}+\I)\z_0\|_{\H}^{\frac{d}{4}}\right)
	\nonumber\\&\qquad\times
	\left(\omega_\upvarphi\big(\lambda_n+\|\Z_{\lambda_n}(s)-\z_0\|_{\H}\big)
	\right)\d s
	\nonumber\\&\leq
	\left(\frac{1}{\lambda_n}\int_{t_0}^{t_0+\lambda_n}
	\|\Z_{\lambda_n}(s)-\z_0\|_{\H}^2 \d s\right)^{\frac{4-d}{8}} \bigg[\left(\frac{1}{\lambda_n}\int_{t_0}^{t_0+\lambda_n}\|(\mathcal{A}+\I)\Z_{\lambda_n}(s)\|_{\H}^{2}\d s\right)^{\frac{d}{8}}+1\bigg] \nonumber\\&\qquad\times
	\left(\frac{1}{\lambda_n}\int_{t_0}^{t_0+\lambda_n}
	\big(\omega_\upvarphi\big(\lambda_n+\|\Z_{\lambda_n}(s)-\z_0\|_{\H}
	\big)\big)^4\d s\right)^{\frac{1}{4}}.
\end{align}
Also, we estimate the final term in \eqref{wknm5} as
\begin{align}\label{J3B}
	J_3&\leq\frac{1}{\lambda_n}\int_{t_0}^{t_0+\lambda_n}
	\|\mathfrak{B}(\z_0,\Z_{\lambda_n}(s)-\z_0)\|_{\H}\left(\omega_\upvarphi\big(\lambda_n+\|\Z_{\lambda_n}(s)-\z_0\|_{\H}\big)\right)\d s
	\nonumber\\&\leq
	\frac{1}{\lambda_n}\int_{t_0}^{t_0+\lambda_n}
	\|\z_0\|_{\H}^{1-\frac{d}{4}}\|(\mathcal{A}+\I)\z_0\|_{\H}^{\frac{d}{4}}
	\|\nabla(\Z_{\lambda_n}(s)-\z_0)\|_{\H}
	\left(\omega_\upvarphi\big(\lambda_n+\|\Z_{\lambda_n}(s)-\z_0\|_{\H}\big)
	\right)\d s
	\nonumber\\&\leq C
	\left(\frac{1}{\lambda_n}\int_{t_0}^{t_0+\lambda_n}
	\|\nabla(\Z_{\lambda_n}(s)-\z_0)\|_{\H}^{2}\d s\right)^{\frac{1}{2}} 
	\left(\frac{1}{\lambda_n}\int_{t_0}^{t_0+\lambda_n}
	\big(\omega_\upvarphi\big(\lambda_n+\|\Z_{\lambda_n}(s)-\z_0\|_{\H}
	\big)\big)^2\d s\right)^{\frac{1}{2}}.
\end{align}
Therefore \eqref{wknm5} together with \eqref{J1B}-\eqref{J3B}, yields
\begin{align*}
\bigg|\frac{1}{\lambda_n}\int_{t_0}^{t_0+\lambda_n}(\mathfrak{B}(\Z_{\lambda_n}(s)),\mathcal{D}\upvarphi(s,\Z_{\lambda_n}(s)))\d s-
(\mathfrak{B}(\z_0),\mathcal{D}\upvarphi(t_0,\z_0))\bigg|
\leq\upsigma(\lambda_n).
\end{align*}
Similar calculations can be performed for \eqref{pas9} (see Appendix \ref{wknBA1} for a detailed explanation)  to obtain 
\begin{align}\label{appB}
	\bigg|\frac{1}{\lambda_n}\int_{t_0}^{t_0+\lambda_n}\delta(s)
	\big(\mathfrak{B}(\Z_{\lambda_n}(s)),\mathcal{A}\Z_{\lambda_n}(s)\big)\d s-
	\delta(t_0)\big(\mathfrak{B}(\z_0),\mathcal{A}\z_0\big)
	\bigg|\leq\upsigma(\lambda_n).
\end{align}

\vskip 0.2mm
\noindent
\textbf{Passing limit into Forchheimer nonlinearity \eqref{pas3}.} 
We now deal with the term \eqref{pas3}. For this, we write 
\begin{align}\label{wknm4}
	&\bigg|\frac{1}{\lambda_n}\int_{t_0}^{t_0+\lambda_n}(\mathfrak{C}(\Z_{\lambda_n}(s)),\mathcal{D}\upvarphi(s,\Z_{\lambda_n}(s)))\d s-
	(\mathfrak{C}(\z_0),\mathcal{D}\upvarphi(t_0,\z_0))\bigg|
	\nonumber\\&\leq
	\bigg|\frac{1}{\lambda_n}\int_{t_0}^{t_0+\lambda_n}(\mathfrak{C}
	(\Z_{\lambda_n}(s)),\mathcal{D}\upvarphi(s,\Z_{\lambda_n}(s))-\mathcal{D}\upvarphi(t_0,\z_0)) 
	\d s\bigg|\nonumber\\&\quad+
	\bigg|\frac{1}{\lambda_n}\int_{t_0}^{t_0+\lambda_n}(\mathfrak{C}
	(\Z_{\lambda_n}(s))-\mathfrak{C}(\z_0),\mathcal{D}\upvarphi(t_0,\z_0))\d s \bigg|.
\end{align}	
Using the calculation \eqref{vdp55}, H\"older's inequality 
and \eqref{eqn-conv-2}, we calculate
\begin{align}\label{wknm4.1}
	&\bigg|\frac{1}{\lambda_n}\int_{t_0}^{t_0+\lambda_n}(\mathfrak{C}
	(\Z_{\lambda_n}(s)),\mathcal{D}\upvarphi(s,\Z_{\lambda_n}(s))-\mathcal{D}\upvarphi(t_0,\z_0)) 
	\d s\bigg|\nonumber\\&\leq C
	\bigg|\frac{1}{\lambda_n}\int_{t_0}^{t_0+\lambda_n}
	\left(\||\Z_{\lambda_n}(s)|^{\frac{r-1}{2}}\nabla\Z_{\lambda_n}(s)\|_{\H}^2+\|\Z_{\lambda_n}\|_{\wi\L^{r+1}}^{r+1}\right)^{\frac{r}{r+1}}\omega_\upvarphi\big(\lambda_n+\|\Z_{\lambda_n}(s)-\z_0\|_{\H}\big)\d s\bigg|\nonumber\\&\leq
	C\bigg|\frac{1}{\lambda_n}\int_{t_0}^{t_0+\lambda_n} \left(\||\Z_{\lambda_n}(s)|^{\frac{r-1}{2}}\nabla\Z_{\lambda_n}(s)\|_{\H}^{\frac{2r}{r+1}}+
	\|\Z_{\lambda_n}(s)\|_{\wi\L^{r+1}}^r\right)\omega_\upvarphi\big(\lambda_n+\|\Z_{\lambda_n}(s)-\z_0\|_{\H}\big)\d s\bigg|
	\nonumber\\&\leq C
	\left(\frac{1}{\lambda_n}\int_{t_0}^{t_0+\lambda_n}
	\||\Z_{\lambda_n}(s)|^{\frac{r-1}{2}}\nabla\Z_{\lambda_n}(s)\|_{\H}^2\d s \right)^{\frac{r}{r+1}}\left(\frac{1}{\lambda_n}\int_{t_0}^{t_0+\lambda_n}
	\big(\omega_\upvarphi\big(\lambda_n+\|\Z_{\lambda_n}(s)-\z_0\|_{\H}\big)
	\big)^{r+1}\d s\right)^{\frac{1}{r+1}}\nonumber\\&\quad+
	C\left(\frac{1}{\lambda_n}\int_{t_0}^{t_0+\lambda_n}
	\|\Z_{\lambda_n}(s)\|_{\wi\L^{r+1}}^{r+1}\d s\right)^{\frac{r}{r+1}} \left(\frac{1}{\lambda_n}\int_{t_0}^{t_0+\lambda_n}
	\big(\omega_\upvarphi\big(\lambda_n+\|\Z_{\lambda_n}(s)-\z_0\|_{\H}\big)
	\big)^2\d s\right)^{\frac{1}{2}}.
\end{align}
In a similar way, the application of Taylor's formula, H\"older's inequality, interpolation inequality, $\V\hookrightarrow\wi\L^6$, and \eqref{ctsdep0.2}, we calculate
\begin{align}\label{wknm4.2}
	&\bigg|\frac{1}{\lambda_n}\int_{t_0}^{t_0+\lambda_n}(\mathfrak{C}
	(\Z_{\lambda_n}(s))-\mathfrak{C}(\z_0),\mathcal{D}\upvarphi(t_0,\z_0))\d s \bigg|
\nonumber\\&\leq
	\frac{1}{\lambda_n}\int_{t_0}^{t_0+\lambda_n}
	\||\Z_{\lambda_n}(s)|+|\z_0|\|_{\wi\L^{3(r+1)}}^{r-1}\|\Z_{\lambda_n}(s)-\z_0\|_{\wi\L^{\frac{6(r+1)}{r+5}}}\|\mathcal{D}\upvarphi(t_0,\z_0)\|_{\H}\d s
	\nonumber\\&\leq
    \left(\frac{1}{\lambda_n}\int_{t_0}^{t_0+\lambda_n}
	\||\Z_{\lambda_n}(s)|+|\z_0|\|_{\wi\L^{3(r+1)}}^{r+1}\d s \right)^ {\frac{r-1}{r+1}}\left(\frac{1}{\lambda_n}\int_{t_0}^{t_0+\lambda_n}
	\|\Z_{\lambda_n}(s)-\z_0\|_{\V}^2\d s \right)^{\frac12} 
	\nonumber\\&\leq C
	\left(\frac{1}{\lambda_n} \int_{t_0}^{t_0+\lambda_n}
	\omega_{\z_0}(\lambda_n)\d s\right),
\end{align}
where in the last line of above inequality, we used \eqref{vdp7}. On combining, \eqref{wknm4.1}-\eqref{wknm4.2}, we finally obtain
\begin{align*}
\bigg|\frac{1}{\lambda_n}\int_{t_0}^{t_0+\lambda_n}(\mathfrak{C}(\Z_{\lambda_n}(s)),\mathcal{D}\upvarphi(s,\Z_{\lambda_n}(s)))\d s-
(\mathfrak{C}(\z_0),\mathcal{D}\upvarphi(t_0,\z_0))\bigg|\leq\upsigma(\lambda_n).
\end{align*}
In a similar way, one can establish that (see Appendix \ref{wknCA1} for a detailed explanation) 
\begin{align}\label{wknCA12}
	\bigg|\frac{1}{\lambda_n}\int_{t_0}^{t_0+\lambda_n}&\delta(s)
	\big(\mathfrak{C}(\Z_{\lambda_n}(s)),(\mathcal{A}+\I)\Z_{\lambda_n}(s)\big)\d s-
	\delta(t_0)\big(\mathfrak{C}(\z_0),(\mathcal{A}+\I)\z_0\big)\bigg|
	\leq\upsigma(\lambda_n),
\end{align}
for $r$ given  in Table \ref{Table2}. Finally, we combine the above convergences to establish the supersolution inequality.

\vskip 0.2mm
\noindent
\textbf{Passing limit into \eqref{vdp3.1}: Supersolution inequality.} Let us rewrite \eqref{vdp3.1}, along a subsequence, as follows:
{\small	\begin{align}\label{vdp3.11}
		&\lambda_n-\frac{1}{\lambda_n}\int_{t_0}^{t_0+\lambda_n}\mathpzc{L}(\Z_{\lambda_n}(s),\a_{\lambda_n}(s))\d s
		\nonumber\\&\geq
		-\frac{1}{\lambda_n}\int_{t_0}^{t_0+\lambda_n}\bigg(\upvarphi_t(s,\Z_{\lambda_n}(s))-\big(\mu\mathcal{A}\Z_{\lambda_n}(s)+\mathfrak{B}(\Z_{\lambda_n}(s))+\beta\mathfrak{C}(\Z_{\lambda_n}(s)),\mathcal{D}\upvarphi(s,\Z_{\lambda_n}(s))\big)\bigg)\d s
		\nonumber\\&\quad-
	    \int_{t_0}^{t_0+\lambda_n}\big(\f(s,\a_{\lambda_n}(s)),\mathcal{D}\upvarphi(s,\Z_{\lambda_n}(s))\big)\d s-
		\frac{2}{\lambda_n}\int_{t_0}^{t_0+\lambda_n}\delta(s)
		\big(\f(s,\a_{\lambda_n}(s)),(\mathcal{A}+\I)\Z_{\lambda_n}(s)\big)\d s
		\nonumber\\&\quad
		-\frac{1}{\lambda_n}\int_{t_0}^{t_0+\lambda_n}\delta'(s)\|\Z_{\lambda_n}(s))\|_{\V}^2\d s+\frac{2\mu}{\lambda_n} \int_{t_0}^{t_0+\lambda_n}\delta(s)\big(\mathcal{A}\Z_{\lambda_n}(s),(\mathcal{A}+\I)\Z_{\lambda_n}(s)\big)\d s
		\nonumber\\&\quad+
		\frac{2}{\lambda_n}\int_{t_0}^{t_0+\lambda_n}\delta(s)(\mathfrak{B}(\Z_{\lambda_n}(s)),(\mathcal{A}+\I)\Z_{\lambda_n}(s))\d s+
		\frac{2\beta}{\lambda_n}\int_{t_0}^{t_0+\lambda_n}\delta(s)(\mathfrak{C}(\Z_{\lambda_n}(s)),(\mathcal{A}+\I)\Z_{\lambda_n}(s))\d s.
\end{align}}
On combining \eqref{ss1}-\eqref{pas10} and rearranging the term, we finally arrive at (along a subsequence) the following inequality:
\begin{align}\label{supsoLe}
	&-\psi_t(t_0,\z_0)+\mu\big(\mathcal{A}\z_0,\mathcal{D}\psi(t_0,\z_0)\big)+
	\big(\mathfrak{B}(\z_0),\mathcal{D}\psi(t_0,\z_0)\big)+\beta\big(\mathfrak{C}(\z_0),\mathcal{D}\psi(t_0,\z_0)\big)
	\nonumber\\&\quad+\frac{1}{\lambda_n}
	\bigg[\int_{t_0}^{t_0+\lambda_n}\big[(\f(t_0,\a_{\lambda_n}(s)),-\mathcal{D}\psi(t_0,\z_0))+\mathpzc{L}(\z_0,\a_{\lambda_n}(s))\big]\d s\bigg] 
	\nonumber\\&\leq\upsigma(\lambda_n).
\end{align}
Finally, by taking the infimum over the control $\a\in\U$ inside the integral and then letting $n\to\infty$, we obtain the supersolution inequality. This establishes that $\mathpzc{V}$ is indeed a viscosity solution of the HJBE \eqref{HJBE}, thereby completing the proof.
\end{proof}

	\begin{appendix}\renewcommand{\thesection}{\Alph{section}}
		\numberwithin{equation}{section}
		
		\section{Some useful calculations}\label{useca}
		The aim of this appendix is to justify the convergences given in \eqref{pas9} and \eqref{pas10}. 
		
	\subsection{Explanation of \eqref{pas9} in Theorem \ref{extunqvisc}}\label{wknBA1}
     Let us calculate \eqref{pas9} as $n\to\infty$. We write 
		\begin{align}\label{wknBA}
		&\bigg|\frac{1}{\lambda_n}\int_{t_0}^{t_0+\lambda_n}\delta(s)
		\big(\mathfrak{B}(\Z_{\lambda_n}(s)),\mathcal{A}\Z_{\lambda_n}(s)\big)-
		\delta(t_0)\big(\mathfrak{B}(\z_0),\mathcal{A}\z_0\big)\bigg|\nonumber\\&\leq 
			\frac{1}{\lambda_n}\int_{t_0}^{t_0+\lambda_n}\delta(s)\bigg|\big(\mathfrak{B}(\Z_{\lambda_n}(s)),\mathcal{A}\Z_{\lambda_n}(s)\big)-\big(\mathfrak{B}(\z_0),\mathcal{A}\z_0
			\big)\bigg|\d s
			\nonumber\\&\quad+
			\frac{1}{\lambda_n}\int_{t_0}^{t_0+\lambda_n}\big|\delta(s)-\delta(t_0)\big|
			\big|\big(\mathfrak{B}(\z_0),\mathcal{A}\z_0\big)\big|\d s
			\nonumber\\&\leq
			\underbrace{\frac{1}{\lambda_n}\int_{t_0}^{t_0+\lambda_n}\delta(s)\bigg|\big(\mathfrak{B}(\Z_{\lambda_n}(s))-\mathfrak{B}(\z_0),\mathcal{A}\Z_{\lambda_n}(s)\big)\bigg|\d s}_{\text{$J_4$}}
			\nonumber\\&\quad+
			\underbrace{	 \frac{1}{\lambda_n}\int_{t_0}^{t_0+\lambda_n}\delta(s)\bigg|\big(\mathfrak{B}(\z_0),\mathcal{A}\Z_{\lambda_n}(s)-\mathcal{A}\z_0\big)\bigg|\d s}_{\text{$J_5$}}
			\nonumber\\&\quad+
			\underbrace{ \frac{1}{\lambda_n}\int_{t_0}^{t_0+\lambda_n}|\delta(s)-\delta(t_0)|
				\big|\big(\mathfrak{B}(\z_0),\mathcal{A}\z_0\big)\big|\d s}_{\text{$J_7$}}.
		\end{align}
		From \eqref{agmon} and H\"older's inequality, we estimate $J_4$ as
		\begin{align*}
		J_4&\leq\frac{1}{\lambda_n}\int_{t_0}^{t_0+\lambda_n}\delta(s)\bigg|\big(\mathfrak{B}(\Z_{\lambda_n}(s))-\mathfrak{B}(\z_0),\mathcal{A}\Z_{\lambda_n}(s)\big)\bigg|\d s
		\nonumber\\&\leq
		\frac{1}{\lambda_n}\int_{t_0}^{t_0+\lambda_n}\delta(s)\|\mathfrak{B}(\Z_{\lambda_n}(s)-\z_0,\Z_{\lambda_n}(s))\|_{\H}\|\mathcal{A}\Z_{\lambda_n}(s)\|_{\H}\d s
		\nonumber\\&\quad+
		\frac{1}{\lambda_n}\int_{t_0}^{t_0+\lambda_n}\delta(s)\|\mathfrak{B}(\z_0,\Z_{\lambda_n}(s)-\z_0)\|_{\H}\|\mathcal{A}\Z_{\lambda_n}(s)\|_{\H}\d s
		\nonumber\\&\leq
		\frac{1}{\lambda_n}\int_{t_0}^{t_0+\lambda_n}\delta(s)
		\|\Z_{\lambda_n}(s)-\z_0\|_{\H}^{1-\frac{d}{4}}\|(\mathcal{A}+\I)(\Z_{\lambda_n}(s)-\z_0)\|_{\H}^{\frac{d}{4}}\|\nabla\Z_{\lambda_n}(s)\|_{\H}
		\|\mathcal{A}\Z_{\lambda_n}(s)\|_{\H}\d s
		\nonumber\\&\quad+
		\frac{1}{\lambda_n}\int_{t_0}^{t_0+\lambda_n}\delta(s)\|\z_0\|_{\H}^{1-\frac{d}{4}}\|(\mathcal{A}+\I)\z_0\|_{\H}^{\frac{d}{4}}
		\|\nabla(\Z_{\lambda_n}(s)-\z_0)\|_{\H}\|\mathcal{A}\Z_{\lambda_n}(s)\|_{\H}\d s
		\nonumber\\&\leq
		\frac{C}{\lambda_n}\int_{t_0}^{t_0+\lambda_n}
		\|\Z_{\lambda_n}(s)-\z_0\|_{\H}^{1-\frac{d}{4}}\|(\mathcal{A}+\I)\Z_{\lambda_n}(s)\|_{\H}^{\frac{d}{4}+1}\d s \nonumber\\&\quad+
		\frac{C}{\lambda_n}\int_{t_0}^{t_0+\lambda_n}
		\|\Z_{\lambda_n}(s)-\z_0\|_{\H}^{1-\frac{d}{4}}\|\mathcal{A}\Z_{\lambda_n}(s)\|_{\H}
	    \d s\nonumber\\&\quad+
	    \frac{C}{\lambda_n}\int_{t_0}^{t_0+\lambda_n}
		\|\nabla(\Z_{\lambda_n}(s)-\z_0)\|_{\H}\|\mathcal{A}\Z_{\lambda_n}(s)\|_{\H}\d s
		\nonumber\\&\leq
		C\left(\frac{1}{\lambda_n}\int_{t_0}^{t_0+\lambda_n}
		\|\Z_{\lambda_n}(s)-\z_0\|_{\H}^4 \d s\right)^{\frac{4-d}{8}}  \left(\frac{1}{\lambda_n}\int_{t_0}^{t_0+\lambda_n}
		\|(\mathcal{A}+\I)\Z_{\lambda_n}(s)\|_{\H}^{2}\d s
		\right)^{\frac{4+d}{8}}
		\nonumber\\&\quad+
		C\left(\frac{1}{\lambda_n}\int_{t_0}^{t_0+\lambda_n}
		\|\Z_{\lambda_n}(s)-\z_0\|_{\H}^{2-\frac{d}{2}}\d s\right)^{\frac12}  
		\left(\frac{1}{\lambda_n}\int_{t_0}^{t_0+\lambda_n}
		\|\mathcal{A}\Z_{\lambda_n}(s)\|_{\H}^{2}\d s
		\right)^{\frac12}
		\nonumber\\&\quad+
		C\left(\frac{1}{\lambda_n}\int_{t_0}^{t_0+\lambda_n}
		\|\Z_{\lambda_n}(s)-\z_0\|_{\V}^2 \d s\right)^{\frac{1}{2}}  \left(\frac{1}{\lambda_n}\int_{t_0}^{t_0+\lambda_n}\|\mathcal{A}\Z_{\lambda_n}(s)\|_{\H}^{2}\d s\right)^{\frac{1}{2}},
		\end{align*}
		where the last inequality follows from \eqref{eqn-conv-1},\eqref{eqn-conv-2}, \eqref{ctsdep0.1}-\eqref{ctsdep0.2} and \eqref{vdp7}.
		Similarly, by using H\"older's inequality, \eqref{wknm2}, \eqref{eqn-conv-1} and \eqref{eqn-conv-2}, one can justify the convergence of the terms $J_5$ and $J_7$ in \eqref{wknBA} as $n\to\infty$. With this, we finally establish \eqref{appB}.
		\subsection{Explanation of \eqref{pas10} in Theorem \ref{extunqvisc}}\label{wknCA1}
	     Let us now estimate the limit of the term \eqref{pas10} as $n\to\infty$. We write 
			\begin{align}\label{wknCA}
				&\bigg|\frac{1}{\lambda_n}\int_{t_0}^{t_0+\lambda_n}\delta(s)\big(\mathfrak{C}(\Z_{\lambda_n}(s)),(\mathcal{A}+\I)\Z_{\lambda_n}(s)\big)\d s-
				\delta(t_0)\big(\mathfrak{C}(\z_0),(\mathcal{A}+\I)\z_0\big)\bigg|\nonumber\\&\leq 
				\frac{1}{\lambda_n}\int_{t_0}^{t_0+\lambda_n}\delta(s)\bigg|\big(\mathfrak{C}(\Z_{\lambda_n}(s)),(\mathcal{A}+\I)\Z_{\lambda_n}(s)\big)-\big(\mathfrak{C}(\z_0),(\mathcal{A}+\I)\z_0\big)\bigg|\d s
				\nonumber\\&+
				\frac{1}{\lambda_n}\int_{t_0}^{t_0+\lambda_n}\big|\delta(s)-\delta(t_0)\big|\big|\big(\mathfrak{C}(\z_0),(\mathcal{A}+\I)\z_0\big)\big|\d s
				\nonumber\\&\leq
				\underbrace{\frac{1}{\lambda_n}\int_{t_0}^{t_0+\lambda_n}\delta(s)\bigg|\big(\mathfrak{C}(\Z_{\lambda_n}(s))-\mathfrak{C}(\z_0),(\mathcal{A}+\I)\Z_{\lambda_n}(s)\big)\bigg|\d s}_{\text{$J_4$}}
				\nonumber\\&+
				\underbrace{	 \frac{1}{\lambda_n}\int_{t_0}^{t_0+\lambda_n}\delta(s)\bigg|\big(\mathfrak{C}(\z_0),(\mathcal{A}+\I)\Z_{\lambda_n}(s)-(\mathcal{A}+\I)\z_0\big)\bigg|
				\d s}_{\text{$J_5$}}
				\nonumber\\&+
				\underbrace{ \frac{1}{\lambda_n}\int_{t_0}^{t_0+\lambda_n}|\delta(s)-\delta(t_0)|
				\big|\big(\mathfrak{C}(\z_0),(\mathcal{A}+\I)\z_0\big)\big|\d s}_{\text{$J_7$}}.
			\end{align}
			We now estimate $J_4$, which presents the greatest difficulty among the terms. By applying Taylor's formula along with the H\"older's and interpolation inequalities, we obtain the following estimate for $J_4$:
			\begin{align*}
				J_4&\leq\frac{1}{\lambda_n}\int_{t_0}^{t_0+\lambda_n}\delta(s)\bigg|\big(\mathfrak{C}(\Z_{\lambda_n}(s))-\mathfrak{C}(\z_0),(\mathcal{A}+\I)\Z_\lambda(s)\big)\bigg|\d s
				\nonumber\\&\leq
				\frac{C}{\lambda_n}\int_{t_0}^{t_0+\lambda_n}
				\||\Z_{\lambda_n}(s)|^{r-1}+|\z_0|^{r-1}\|_{\wi\L^3}
				\|\Z_{\lambda_n}(s)-\z_0\|_{\wi\L^6}
				\|(\mathcal{A}+\I)\Z_{\lambda_n}(s)\|_{\H}\d s
				\nonumber\\&\leq
				\frac{C}{\lambda_n}\int_{t_0}^{t_0+\lambda_n}
				\left(\|\Z_{\lambda_n}(s)\|_{\wi\L^{3(r-1)}}^{r-1}+\|\z_0\|_{\wi\L^{3(r-1)}}^{r-1}\right)\|\Z_{\lambda_n}(s)-\z_0\|_{\wi\L^6}
				\|(\mathcal{A}+\I)\Z_{\lambda_n}(s)\|_{\H}\d s 
			\end{align*}
			
	We now consider following cases:
	\vskip 2mm
	\noindent
	\textbf{Case-I of in Table \ref{Table2}.}
	The application of Sobolev embedding $\V\hookrightarrow\wi\L^{r+1}$, for any $r\in(3,\infty)$, H\"older's inequality, uniform energy estimates \eqref{eqn-conv-1} and \eqref{eqn-conv-2}, and continuous dependence estimate \eqref{ctsdep0.2} yields the following estimate:
	\begin{align*}
			J_4&\leq\frac{1}{\lambda_n}\int_{t_0}^{t_0+\lambda_n}\delta(s)\bigg|\big(\mathfrak{C}(\Z_{\lambda_n}(s))-\mathfrak{C}(\z_0),(\mathcal{A}+\I)\Z_\lambda(s)\big)\bigg|\d s
			\nonumber\\&\leq
			\frac{C}{\lambda_n}\int_{t_0}^{t_0+\lambda_n}
			\left(\|\Z_{\lambda_n}(s)\|_{\V}^{r-1}+\|\z_0\|_{\V}^{r-1}\right)\|\Z_{\lambda_n}(s)-\z_0\|_{\V}
			\|(\mathcal{A}+\I)\Z_{\lambda_n}(s)\|_{\H}\d s
			\nonumber\\&\leq
			\frac{C}{\lambda_n}\int_{t_0}^{t_0+\lambda_n}
			\|\Z_{\lambda_n}(s)-\z_0\|_{\V}
			\|(\mathcal{A}+\I)\Z_{\lambda_n}(s)\|_{\H}\d s
			\nonumber\\&\leq C
			\left(\int_{t_0}^{t_0+\lambda_n}
			\|(\mathcal{A}+\I)\Z_{\lambda_n}(s)\|_{\H}^{2}\d s\right)^{\frac12}
			\left(\frac{1}{\lambda_n}\int_{t_0}^{t_0+\lambda_n}
			\|\Z_{\lambda_n}(s)-\z_0\|_{\V}^2\d s \right)^{\frac12}
			\nonumber\\&\leq\upsigma(\lambda_n).
	\end{align*}
    
	\vskip 2mm
	\noindent
		\textbf{Case-II of in Table \ref{Table2}.} By using the Sobolev embedding $\V\hookrightarrow\wi\L^{r+1}$ for any $r\in(3,5)$ and similar calculations as we did above, we find
	{\small	\begin{align*}
			&J_4\leq\frac{C}{\lambda_n}\int_{t_0}^{t_0+\lambda_n}
			\|\Z_{\lambda_n}(s)\|_{\wi\L^{3(r-1)}}^{r-1}\|\Z_{\lambda_n}(s)-\z_0\|_{\V}
			\|(\mathcal{A}+\I)\Z_{\lambda_n}(s)\|_{\H}\d s
			\nonumber\\&\quad+
			\frac{C}{\lambda_n}\int_{t_0}^{t_0+\lambda_n}
			\|\Z_{\lambda_n}(s)-\z_0\|_{\V}\|(\mathcal{A}+\I)\Z_{\lambda_n}(s)\|_{\H}
			\d s
			\nonumber\\&\leq
			\frac{C}{\lambda_n}\int_{t_0}^{t_0+\lambda_n}
			\|\Z_{\lambda_n}(s)\|_{\wi\L^{r+1}}\|\Z_{\lambda_n}(s)\|_{\wi\L^{3(r+1)}}^{r-2}\|\Z_{\lambda_n}(s)-\z_0\|_{\V}
			\|(\mathcal{A}+\I)\Z_{\lambda_n}(s)\|_{\H}\d s
			\nonumber\\&\quad+
			\frac{C}{\lambda_n}\int_{t_0}^{t_0+\lambda_n}
			\|\Z_{\lambda_n}(s)-\z_0\|_{\V}\|(\mathcal{A}+\I)\Z_{\lambda_n}(s)\|_{\H}\d s
			\nonumber\\&\leq C
			\bigg[\left(\int_{t_0}^{t_0+\lambda_n}\|\Z_{\lambda_n}(s)\|_{\wi\L^{3(r+1)}}^{r+1}\d s\right)^{\frac{r-2}{r+1}}
			\left(\frac{1}{\lambda_n}\int_{t_0}^{t_0+\lambda_n}
			\|\Z_{\lambda_n}(s)-\z_0\|_{\V}^{\frac{r+1}{5-r}}\d s \right)^{\frac{5-r}{r+1}}
			\nonumber\\&\quad+
			\left(\frac{1}{\lambda_n}\int_{t_0}^{t_0+\lambda_n}
			\|\Z_{\lambda_n}(s)-\z_0\|_{\V}^2\d s \right)^{\frac12}\bigg]\left(\int_{t_0}^{t_0+\lambda_n}
			\|(\mathcal{A}+\I)\Z_{\lambda_n}(s)\|_{\H}^{2}\d s\right)^{\frac12}
			\nonumber\\&\leq\upsigma(\lambda_n).
	\end{align*}}

	   The convergence of other terms in \eqref{wknCA} can be shown similarly as we have demonstrated for the convective term in \eqref{wknBA}.

   \section{Existence of viscosity solution for $r=5$}\label{casere5}
   In this appendix, we demonstrate the existence of a viscosity solution for the HJBE \eqref{detHJB1} when $r=5$. Specifically, we prove the existence of the limit specified in \eqref{pas10} as $n\to\infty$ for the case where the absorption exponent $r=5$. It represents a significant advantage of our work compared to the stochastic case (see \cite{smtm1}), where we were not able establish the case for $r=5$. 
   
   \subsection{Energy estimates in $\D(\mathcal{A})$}\label{daengca}
  In this section, we establish the energy estimates for $\D(\mathcal{A})$, which is useful to justify the limit \eqref{pas10} in the proof of the existence of viscosity solutions.
   \begin{proposition}\label{DAENG} 
   		Assume $\Z(\cdot)=\Z(\cdot;t,\z,\a(\cdot))$ is a strong solution to the system \eqref{stap} with $\z\in\D(\mathcal{A})$. Then, for $r$ given in Table \ref{Table2}, the following uniform energy estimate holds: 
   		\begin{align*}
   			&\|\mathcal{A}\Z(s)\|_{\H}^2+
   			\int_t^{s}\|\mathcal{A}^{\frac32}\Z(\tau)\|_{\H}^2\d\tau
   			\nonumber\\&\leq 
   			C\left(\mu,\beta,T,\|\z\|_{\D(\mathcal{A})},\sup\limits_{\tau\in[t,T]} \|\Z(\tau)\|_{\V},\int_0^T\|\Z(s)\|_{\wi\L^{3(r+1)}}^{r+1},\int_t^{T}\|\f(\tau,\a(\tau))\|_{\V}^2\d\tau\right),
   	\end{align*}
   for all $s\in[t, T]$.
   \end{proposition}
   
   \begin{proof}
   According to \cite[Theorem 4.4]{SMTM}, if $\f\in\W^{1,1}(0,T;\H)$ and $\z\in\D(\mathcal{A})$, then \eqref{stap}  admits a unique strong solution $\Z$ such that 
   $$\Z\in\W^{1,\infty}([t,T];\H)\ \text{ and }\ \mathcal{A}\Z\in\mathrm{L}^{\infty}(t,T;\H).$$ For $\f\in\mathrm{L}^2(t,T;\V)\subset \mathrm{L}^2(t,T;\H)$, one can use the density of $\W^{1,1}(t,T;\H)$ in $\mathrm{L}^2(t,T;\H)$  to obtain 
   $$\Z \in\mathrm{L}^{\infty}(t,T;\D(\mathcal{A}))\cap\mathrm{L}^{2}(t,T;\D(\mathcal{A}^{\frac{3}{2}})),$$  such that $\Z_t\in\mathrm{L}^2(t,T;\H).$ 
   Therefore, by using the absolute continuity lemma (\cite[Lemma 1.2, Chapter III]{Te}), we take the inner product with $\mathcal{A}^2\Z$ in \eqref{stap} to see that 
   	\begin{align}\label{daeng}
   		&\|\mathcal{A}\Z(s)\|_{\H}^2+\mu
   		\int_t^{s}\|\mathcal{A}^{\frac32}\Z(\tau)\|_{\H}^2\d\tau
   		\nonumber\\&= \|\mathcal{A}\z\|_{\H}^2 -\int_t^{s}\big(\mathfrak{B}(\Z(\tau)),\mathcal{A}^2\Z(\tau)\big)\d\tau
   		-\beta \int_t^{s}
   		\big(\mathfrak{C}(\Z(\tau)),\mathcal{A}^2\Z(\tau)\big)\d\tau
   		\nonumber\\&\quad+ \int_t^{s}\big(\f(\tau,\a(\tau)),\mathcal{A}^2\Z(\tau)\big)\d\tau,
   	\end{align}
   	for all $s\in[t,T]$. From Young's inequality, equality \eqref{torusequ} and the estimate \eqref{syymB3}, we calculate following:
   	\begin{align}\label{daeng1}
   		\big(\f,\mathcal{A}^2\Z\big)=\big(\mathcal{A}^{\frac{1}{2}}\f,
   		\mathcal{A}^{\frac{3}{2}}\Z\big)
   		\leq
   		\|\mathcal{A}^{\frac{1}{2}}\f\|_{\H}\|\mathcal{A}^{\frac{3}{2}}\Z\|_{\H}
   		\leq
   		\frac{\mu}{4}\|\mathcal{A}^{\frac{3}{2}}\Z\|_{\H}^2+
   		\frac{1}{\mu}\|\f\|_{\V}^2.
   	\end{align}
   	Let us now estimate the terms containing bilinear and nonlinear operators in the right hand side of \eqref{daeng}. We first estimate $(\mathfrak{B}(\Z),\mathcal{A}^2\Z)$, by using \eqref{fracLeb}, H\"older's and Young's inequalities as
   	\begin{align}\label{daeng2}
   		\big|(\mathfrak{B}(\Z),\mathcal{A}^2\Z)\big|&=
   		\big|(\mathcal{A}^{\frac{1}{2}}\mathfrak{B}(\Z),\mathcal{A}^{\frac{3}{2}}\Z)\big|
   		\nonumber\\&\leq
   		\|\mathcal{A}^{\frac{1}{2}}\mathfrak{B}(\Z)\|_{\H}\|\mathcal{A}^{\frac{3}{2}}\Z\|_{\H}
   		\nonumber\\&\leq C
   		\|\Z\|_{\wi\L^6}\|\mathcal{A}\Z\|_{\wi\L^3}
   		\|\mathcal{A}^{\frac{3}{2}}\Z\|_{\H}
   		\nonumber\\&\leq C
   		\|\Z\|_{\wi\L^6}\|\mathcal{A}^{\frac{5}{4}}\Z\|_{\H}
   		\|\mathcal{A}^{\frac{3}{2}}\Z\|_{\H}
   		\nonumber\\&\leq C
   		\|\Z\|_{\wi\L^6}\|\mathcal{A}^{\frac12}\Z\|_{\H}^{\frac{1}{4}}
   		\|\mathcal{A}^{\frac{3}{2}}\Z\|_{\H}^{\frac{7}{4}}
   		\nonumber\\&\leq 
   		\frac{\mu}{8}\|\mathcal{A}^{\frac{3}{2}}\Z\|_{\H}^2
   		+C\|\Z\|_{\V}^{10}.
   	\end{align}
   	Similarly, we calculate $(\mathfrak{C}(\Z),\mathcal{A}^2\Z)$ as 
   	\begin{align}\label{daeng3}
   		\big|(\mathfrak{C}(\Z),\mathcal{A}^2\Z)\big|&\leq C
   		\|\Z\|_{\wi\L^{3(r-1)}}^{r-1}\|\mathcal{A}^{\frac{1}{2}}\Z\|_{\wi\L^6}
   		\|\mathcal{A}^{\frac{3}{2}}\Z\|_{\H}
   		\nonumber\\&\leq
   		C\|\Z\|_{\wi\L^{r+1}}^2 \|\Z\|_{\wi\L^{3(r+1)}}^{2(r-2)}
   		(\|\mathcal{A}\Z\|_{\H}^2+\|\mathcal{A}^{\frac12}\Z\|_{\H}^2)
   	+	\frac{\mu}{8}\|\mathcal{A}^{\frac{3}{2}}\Z\|_{\H}^2.
   	\end{align}
   	Combining \eqref{daeng1}-\eqref{daeng3} in the equality \eqref{daeng} gives the following estimate:
   	\begin{align}\label{daeng4}
   		&\|\mathcal{A}\Z(s)\|_{\H}^2+\frac{\mu}{2}
   		\int_t^{s}\|\mathcal{A}^{\frac32}\Z(\tau)\|_{\H}^2\d\tau
   		\nonumber\\&\leq\|\mathcal{A}\z\|_{\H}^2+ C\int_t^{s}\|\Z(\tau)\|_{\V}^{10}\d\tau
   		+C\int_t^{s}
   		\|\Z(\tau)\|_{\wi\L^{r+1}}^2 \|\Z(\tau)\|_{\wi\L^{3(r+1)}}^{2(r-2)}
   		\|\mathcal{A}\Z(\tau)\|_{\H}^2\d\tau
   		\nonumber\\&\quad+C\int_t^{s}
   		\|\Z(\tau)\|_{\wi\L^{r+1}}^2 \|\Z(\tau)\|_{\wi\L^{3(r+1)}}^{2(r-2)}
   		\|\mathcal{A}^{\frac12}\Z(\tau)\|_{\H}^2\d\tau
   		+\frac{1}{\mu}\int_t^{s}
   		\|\f(\tau,\a(\tau))\|_{\V}^2\d\tau,
   	\end{align}
   	for all $s\in[t,T]$. 
   	To proceed further, we consider following two cases:
   	
   	\vskip 0.2mm
   	\noindent
   	\textbf{Case-I of Table \ref{Table2}.} By using $\V\hookrightarrow\wi\L^{r+1}$ for any $r\geq1$ and \eqref{eqn-conv-1} and \eqref{eqn-conv-2}, inequality \eqref{daeng4} yields
   	\begin{align}\label{daeng5}
   		&\|\mathcal{A}\Z(s)\|_{\H}^2+\frac{\mu}{2}
   		\int_t^{s}\|\mathcal{A}^{\frac32}\Z(\tau)\|_{\H}^2\d\tau
   		\nonumber\\&\leq C+\|\mathcal{A}\z\|_{\H}^2+\frac{1}{\mu}\int_t^{s}
   		\|\f(\tau,\a(\tau))\|_{\V}^2\d\tau+C\int_t^{s}\|\mathcal{A}\Z\|_{\H}^2\d\tau.
   	\end{align}
   	On employing Gr\"onwall's inequality, we obtain
   	\begin{align*}
   		&\sup\limits_{s\in[t,T]}\|\mathcal{A}\Z(s)\|_{\H}^2+
   		\int_t^{T}\|\mathcal{A}^{\frac32}\Z(\tau)\|_{\H}^2\d\tau
   		\nonumber\\&\leq C\left(\mu,\beta,T,\|\z\|_{\D(\mathcal{A})},\sup\limits_{\tau\in[t,T]} \|\Z(\tau)\|_{\V},\int_t^{T}\|\f(\tau,\a(\tau))\|_{\V}^2\d\tau\right).
   	\end{align*}
   	\vskip 0.2mm
   	\noindent
   	\textbf{Case-II and Case-III of Table \ref{Table2}.} Using  the energy estimates \eqref{eqn-conv-1} and \eqref{eqn-conv-2}, inequality \eqref{daeng} reduces to 
   	\begin{align}\label{daeng6}
   		&\|\mathcal{A}\Z(s)\|_{\H}^2+\frac{\mu}{2}
   		\int_t^{s}\|\mathcal{A}^{\frac32}\Z(\tau)\|_{\H}^2\d\tau
   		\nonumber\\&\leq 
   		C+\|\mathcal{A}\z\|_{\H}^2+\frac{1}{\mu}\int_t^{s}
   		\|\f(\tau,\a(\tau))\|_{\V}^2\d\tau+C\int_t^{s}
   		\|\Z(\tau)\|_{\wi\L^{r+1}}^2 \|\Z(\tau)\|_{\wi\L^{3(r+1)}}^{2(r-2)}\d\tau
   		\nonumber\\&\quad+
   		C\int_t^{s}
   		\|\Z(\tau)\|_{\wi\L^{r+1}}^2 \|\Z(\tau)\|_{\wi\L^{3(r+1)}}^{2(r-2)}
   		\|\mathcal{A}\Z(\tau)\|_{\H}^2\d\tau,
   	\end{align}
   	for all $s\in[t,T]$. By using H\"older's inequality and utilizing the uniform energy estimates \eqref{eqn-conv-1} and \eqref{eqn-conv-2}, one can show the following bound:
   	\begin{align}\label{daeng7}
   		&\int_t^{T}
   		\|\Z(\tau)\|_{\wi\L^{r+1}}^2 \|\Z(\tau)\|_{\wi\L^{3(r+1)}}^{2(r-2)}\d\tau
   		\nonumber\\&=
   		\left\{
   		\begin{aligned}
   			&\left(\int_t^T\|\Z(\tau)\|_{\wi\L^{r+1}}^{r+1}\d\tau\right)^{\frac{2}{r+1}}
   			\left(\int_t^T\|\Z(\tau)\|_{\wi\L^{3(r+1)}}^{r+1}\d\tau\right)^{\frac{2(r-2)}{r+1}}, \  &&\text{when} \ r\in(3,5),\\
   			&\sup_{\tau\in [t,T]}	\|\Z(\tau)\|_{\wi\L^6}^2\int_t^{T}
   		 \|\Z(\tau)\|_{\wi\L^{18}}^{6}\d\tau,
   			\  &&\text{when} \ r=5.
   		\end{aligned}
   		\right.
   		\nonumber\\&\leq C. 
   	\end{align}
   	Therefore, on employing Gr\"onwall's inequality together with \eqref{daeng7}, we conclude from \eqref{daeng6} that 
   	\begin{align*}
   		&\|\mathcal{A}\Z(s)\|_{\H}^2+
   		\int_t^{s}\|\mathcal{A}^{\frac32}\Z(\tau)\|_{\H}^2\d\tau
   		\nonumber\\&\leq 
   		C\left(\mu,\beta,T,\|\z\|_{\D(\mathcal{A})},\sup\limits_{\tau\in[t,T]} \|\Z(\tau)\|_{\V},\int_0^T\|\Z(s)\|_{\wi\L^{3(r+1)}}^{r+1},\int_t^{T}\|\f(\tau,\a(\tau))\|_{\V}^2\d\tau\right),
   	\end{align*}
   	for all $s\in[t,T]$.
   \end{proof}

  \subsection{Justification of the term $J_4$ in \eqref{wknCA} as $n\to\infty$ for $r=5$}\label{dar5}
	By using Sobolev's embedding $\V\hookrightarrow\wi\L^{6}$, H\"older's inequality, uniform energy estimates \eqref{eqn-conv-1} and \eqref{eqn-conv-2}, and continuous dependence estimate \eqref{ctsdep0.2}, we calculate 
	\begin{align*}
		&J_4\leq\frac{C}{\lambda_n}\int_{t_0}^{t_0+\lambda_n}
		\|\Z_{\lambda_n}(s)\|_{\wi\L^{12}}^{4}\|\Z_{\lambda_n}(s)-\z_0\|_{\V}
		\|(\mathcal{A}+\I)\Z_{\lambda_n}(s)\|_{\H}\d s
		\nonumber\\&\quad+
		\frac{C}{\lambda_n}\int_{t_0}^{t_0+\lambda_n}
		\|\Z_{\lambda_n}(s)-\z_0\|_{\V}\|(\mathcal{A}+\I)\Z_{\lambda_n}(s)\|_{\H}
		\d s
		\nonumber\\&\leq
		\frac{C}{\lambda_n}\int_{t_0}^{t_0+\lambda_n}
		\|\Z_{\lambda_n}(s)\|_{\wi\L^{6}}\|\Z_{\lambda_n}(s)\|_{\wi\L^{18}}^{3}\|\Z_{\lambda_n}(s)-\z_0\|_{\V}
		\|(\mathcal{A}+\I)\Z_{\lambda_n}(s)\|_{\H}\d s
		\nonumber\\&\quad+
		\frac{C}{\lambda_n}\int_{t_0}^{t_0+\lambda_n}
		\|\Z_{\lambda_n}(s)-\z_0\|_{\V}\|(\mathcal{A}+\I)\Z_{\lambda_n}(s)\|_{\H}\d s
		\nonumber\\&\leq C
		\bigg[\left(\int_{t_0}^{t_0+\lambda_n}\|\Z_{\lambda_n}(s)\|_{\wi\L^{18}}^{6}\d s\right)^{\frac{1}{2}}
		\left(\frac{1}{\lambda_n}\int_{t_0}^{t_0+\lambda_n}
		\|\Z_{\lambda_n}(s)-\z_0\|_{\V}^{2}\d s \right)^{\frac{1}{2}}
		\nonumber\\&\quad+
		\left(\frac{1}{\lambda_n}\int_{t_0}^{t_0+\lambda_n}
		\|\Z_{\lambda_n}(s)-\z_0\|_{\V}^2\d s \right)^{\frac12}\bigg]\left(\int_{t_0}^{t_0+\lambda_n}
		\|(\mathcal{A}+\I)\Z_{\lambda_n}(s)\|_{\H}^{2}\d s\right)^{\frac12}
		\nonumber\\&\leq\upsigma(\lambda_n).
\end{align*}
With this, we finally show that \eqref{wknCA12} holds for $r=5$. As a consequence, the supersolution inequality \eqref{supsoLe} holds true for $r=5$ also, thus proving the existence of viscosity solution for $r=5$.
	\end{appendix}

	\medskip
	\noindent
	\textbf{Acknowledgments:} The first author gratefully acknowledges the financial support provided by the Ministry of Education, Government of India (MHRD). M. T. Mohan's research is supported by the National Board of Higher Mathematics (NBHM), Department of Atomic Energy, Government of India, under Project No. 02011/13/2025/NBHM(R.P)/R\&D II/1137.

	\medskip\noindent	\textbf{Declarations:} 
	
	\noindent 	\textbf{Ethical Approval:}   Not applicable 
	
	
	\noindent  \textbf{Conflict of interest: }On behalf of all authors, the corresponding author states that there is no conflict of interest.
	
	\noindent 	\textbf{Authors' contributions: } All authors have contributed equally. 
	
	\noindent 	\textbf{Funding: } NBHM, India, 02011/13/2025/NBHM(R.P)/R\&D II/1137 (M. T. Mohan)
	
	\noindent 	\textbf{Availability of data and materials: } Not applicable.

\end{document}